\documentclass[12pt,reqno]{amsart}
\usepackage{amsmath}
\usepackage{amssymb}
\usepackage{verbatim}
\usepackage{mathrsfs}
\usepackage{graphicx}
\usepackage{caption}
\usepackage{subcaption}
\usepackage{epstopdf}
\usepackage{makecell}
\usepackage{tabularx}
\usepackage{longtable}
\usepackage{tikz}
\usepackage{pgfplots}
\usepackage{xcolor}
\usepackage{pgfplots}
\usepackage{pgfplotstable}
\usepackage[ruled,vlined]{algorithm2e}
\usepackage[noend]{algpseudocode}
\usepackage[toc]{appendix}
\usepackage[capitalise]{cleveref}
\usepackage[noadjust]{cite}
\usepackage{color}

\usepackage[top=0.6in,bottom=0.5in,left=0.7in,right=0.7in]{geometry}

\newtheorem{lemma}{Lemma}

\newtheorem{theorem}[lemma]{Theorem}
\newtheorem{remark}[lemma]{Remark}

\newtheorem{example}{Example}

\numberwithin{equation}{section}
\numberwithin{lemma}{section}

\newcommand{\N}{\mathbb{N}}    
\newcommand{\NN}{\mathbb{N}_0} 
\newcommand{\R}{\mathbb{R}}    

\newcommand{\nv}{\vec{n}}
\newcommand{\bo}{\mathcal{O}}
\newcommand{\Op}{\Omega^+}
\newcommand{\Om}{\Omega^-}

\newcommand{\be}{ \begin{equation} }
\newcommand{\ee}{ \end{equation} }
\newcommand{\odd}{\operatorname{odd}}
\newcommand{\ind}{\Lambda}

\begin{document}
		
\title[A High Order Compact Scheme for Interface Problems]{A High Order Compact Finite Difference Scheme for Elliptic Interface Problems with Discontinuous
and High-Contrast Coefficients }

\author{Qiwei Feng, Bin Han and Peter Minev}

\thanks{
Research supported in part by
Natural Sciences and Engineering Research Council (NSERC) of Canada}

\address{Department of Mathematical and Statistical Sciences,
University of Alberta, Edmonton,\quad Alberta, Canada T6G 2G1.
\quad {\tt qfeng@ualberta.ca }\qquad {\tt bhan@ualberta.ca}
\quad{\tt minev@ualberta.ca}
}

\begin{abstract}
	
The elliptic interface problems with discontinuous and high-contrast coefficients appear in many applications and often lead to huge condition numbers of the corresponding linear systems. 
Thus, it is highly desired to construct high order schemes to solve the elliptic interface problems with discontinuous and high-contrast coefficients.
Let $\Gamma$ be a smooth curve inside
a rectangular region $\Omega$.
In this paper, we consider the elliptic interface problem $-\nabla\cdot (a \nabla u)=f$ in $\Omega\setminus \Gamma$ with Dirichlet boundary conditions, where the coefficient $a$ and the source term $f$ are smooth in $\Omega\setminus \Gamma$ and the two nonzero jump condition functions $[u]$ and $[a\nabla u\cdot \vec{n}]$ across $\Gamma$ are smooth along $\Gamma$.	To solve such elliptic interface problems, we propose a high order compact finite difference scheme for numerically computing both the solution $u$ and the gradient $\nabla u$ on uniform Cartesian grids without changing coordinates into local coordinates.
Our numerical experiments confirm the fourth order accuracy for computing the solution $u$,  the gradient $\nabla u$ and the velocity $a \nabla u$ of the proposed compact finite difference scheme on uniform meshes for the elliptic interface problems with discontinuous and high-contrast coefficients.
\end{abstract}

\keywords{Elliptic interface equations, high order compact finite difference schemes, discontinuous, cell-wise smooth and high-contrast coefficients, two non-homogeneous jump conditions}

\subjclass[2010]{65N06, 35J05, 76S05, 41A58}
\maketitle

\maketitle

\pagenumbering{arabic}

\section{Introduction and problem formulation}

Elliptic interface problems with discontinuous coefficients arise in many applications such as modelling of underground waste
disposal, solidification processes, mechanics of composite materials, oil reservoir simulations and other flows in porous media, multiphase flows, and many others.

Most of the numerical techniques for such problems are based on (continuous and discontinuous) finite element and finite volume methods (e.g., see \cite{GO94,APHansbo2002,GongLiLi08,Babu1970,BrambleKing96,EwingLLL99,EwingIL01,LinLZ15,HeLL2011}).  Since the goal of our paper is
to develop a compact high-order finite difference scheme, we focus our literature review on the works employing such discretizations.
The most important contributions involving the finite difference method (IIM) are due to LeVeque and  Li (see \cite{LiIto06,LeLi94,Li98,Li96} and the references therein). In particular,
\cite[Section~7.2.7]{LiIto06} proposes a fourth order compact finite difference scheme for numerical approximations of  elliptic problems with piecewise  constant coefficients,  continuous source terms and
two homogeneous jump conditions and \cite[Section~7.5.4]{LiIto06} provides some numerical results for the proposed fourth order compact scheme on uniform grids.
\cite{CFL19} derives a second order compact finite difference method for the solution globally and its gradient at the interface for the interface elliptic problems with piecewise  smooth coefficients and
two non-homogeneous jump conditions. \cite{DFL20} considers anisotropic elliptic interface problems whose coefficient matrix is symmetric, semi-positive-definite, and derives a hybrid discretization involving finite elements away
of the interfaces, and an immersed interface finite difference approximation near or at the interfaces.  The error in the maximum norm is of order $\bo(h^2\log \frac{1}{h})$. Based on the
fast iterative immersed interface method (FIIIM) proposed in \cite{Li98}, \cite{WB00} constructs
a second order explicit-jump immersed interface method (EJIIM) for elliptic interface problems with
discontinuous coefficients and singular sources.  In fact this approach of EJIIM is quite similar to the famous immersed boundary method (IBM) of Peskin \cite{Peskin02}. For the elliptic interface problems
with discontinuous coefficients and singular sources, a high-order method is constructed by combining a Discontinuous Galerkin (DG) spatial discretization and IBM in \cite{BG15}.
For elliptic problems with sharp-edged interfaces, the
matched interface and boundary (MIB) method is considered in \cite{YuZhouWei07,YuWei07}.
In \cite{ZZFW06}, a
high order MIB method is introduced to solve the  elliptic equations with singular sources. Moreover, the fourth order compact finite difference schemes for the elliptic equations on irregular domains are derived in \cite{ItoLiKyei05,LiIto06}.

In \cite{FengHanMinev21}, we derived a sixth order compact finite difference scheme for the Poisson equation with  singular sources, whose solution has a discontinuity across a smooth interface.
The most important feature of the scheme is that the matrix of the resulting linear system is independent of the location of the singularity in the source term.  In the present paper, we consider the more general
case of an elliptic interface problem with a discontinuous, piecewise smooth, and high-contrast coefficient, and a discontinuous source term. The problem involves two non-homogeneous jump conditions across an interface curve, one on the solution, and one the normal component of its gradient.

To fix the ideas, let $\Omega=(l_1, l_2)\times(l_3, l_4)$ be a two-dimensional rectangular region. We define a smooth curve $\Gamma:=\{(x,y)\in \Omega: \psi(x,y)=0\}$, which partitions $\Omega$ into two subregions:
$\Op:=\{(x,y)\in \Omega\; :\; \psi(x,y)>0\}$ and $\Om:=\{(x,y)\in \Omega\; : \; \psi(x,y)<0\}$, where $\psi(x,y)$ is a smooth function in 2D.
We also define
$
a_{\pm}:=a \chi_{\Omega_{\pm}}$, $f_{\pm}:=f \chi_{\Omega_{\pm}}$ and $u_{\pm}:=u \chi_{\Omega_{\pm}}.
$
%
\begin{figure}[htbp]
	\hspace{6mm}
	\vspace{-0.3mm}
	\centering	
\begin{subfigure}[b]{0.4\textwidth}
\begin{tikzpicture}
	\draw[domain =0:360,smooth]
	 plot({sqrt(2)*cos(\x)}, {sqrt(2)*sin(\x)});
	\draw
(-pi, -pi) -- (-pi, pi) -- (pi, pi) -- (pi, -pi) --(-pi,-pi);
	\node (A) at (0,0) {$\Omega^{-}$};
	\node (A) at (2,2) {$\Omega^{+}$};
	\node (A) at (1,2) {$a^{+}$};
	\node (A) at (0.7,-0.5) {$a_{-}$};
	\node (A) at (0.7,0.5) {$f_{-}$};
    \node (A) at (2.5,1) {$f_{+}$}; 	
	\node (A) at (-1,-1.3) {$\Gamma$}; 	
	\node (A) at (1.5,-2.8) {$\Omega\backslash \Gamma={\Omega}^{+} \cup {\Omega}^{-}$};
	\node (A) at (-2.7,-1) {$\partial\Omega$};
	\node (A) at (-2.9,-2) {$g$};
	\node (A) at (1.8,-0.6) {$\nv$};
	\node (A) at (-1.8,1) {$[u]=g_1$};
	\node (A) at (-1.3,1.6) {$[a\nabla u \cdot \nv]=g_2$};
	\draw[->] (1.3,-0.5)--(1.65,-0.62);
\end{tikzpicture}
\end{subfigure}
	\caption
	{The problem region $\Omega=(-\pi,\pi)^2$ and
the two subregions $\Omega^+=\{(x,y)\in \Omega\; :\; \psi(x,y)>0\}$ and $\Omega^-=\{(x,y)\in \Omega\; :\; \psi(x,y)<0\}$ partitioned by the interface curve $\Gamma=\{(x,y)\in \Omega\; :\; \psi(x,y)=0\}$ with the function
$\psi(x,y)=x^2+y^2-2$.}
\label{fig:figure0}
\vspace{-3.9mm}
\end{figure}
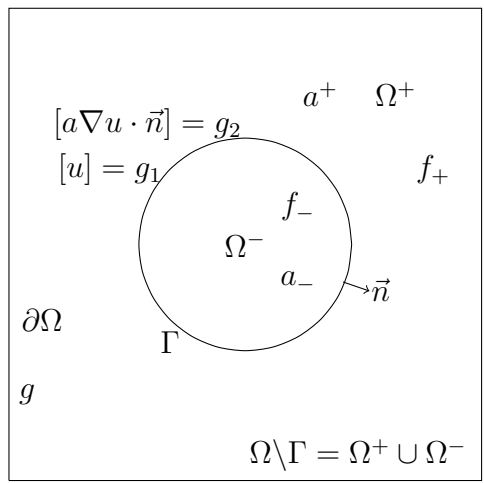
The goal of this paper is to derive a high order compact finite difference scheme for the elliptic interface problem with piecewise smooth coefficients and sources:
\begin{equation} \label{Qeques1}
\begin{cases}
-\nabla \cdot( a\nabla u)=f &\text{in $\Omega \setminus \Gamma$},\\	
\left[u\right]=g_1 &\text{on $\Gamma$},\\
\left[a\nabla  u \cdot \nv \right]=g_2  &\text{on $\Gamma$},\\
u=g &\text{on $\partial\Omega$}.
\end{cases}
\end{equation}
Here $\nv$ is the unit normal vector of $\Gamma$ pointing towards $\Op$, and for a point $(x_0,y_0)\in \Gamma$,
\begin{align}
&\hspace{12mm}[u](x_0,y_0):=\lim_{(x,y)\in \Op \atop (x,y) \to (x_0,y_0) }u(x,y)- \lim_{(x,y)\in \Om \atop (x,y) \to (x_0,y_0) }u(x,y),\label{jumpCD1}\\
&[ a\nabla  u \cdot \nv](x_0,y_0):=  \lim_{(x,y)\in \Op \atop (x,y) \to (x_0,y_0) } a(x,y)\nabla  u(x,y) \cdot \nv- \lim_{(x,y)\in \Om \atop (x,y) \to (x_0,y_0) }  a(x,y) \nabla  u(x,y) \cdot \nv. \label{jumpCD2}
\end{align}
 For the convenience of readers, an example for \eqref{Qeques1} with $\psi(x,y)=x^2+y^2-2$ is illustrated
in \cref{fig:figure0}. Furthermore, \cite{Vazqu07} provides the physical background for the problem \eqref{Qeques1}.


In this paper we consider the elliptic interface problem  in \eqref{Qeques1} under the following assumptions:

\begin{itemize}
	
\item $a(x,y)$ is smooth and positive in each of the subregions $\Op$ and $\Om$, and $a(x,y)$ is discontinuous across the interface curve $\Gamma$.	
	
\item $f(x,y)$ is smooth in each of the subregions $\Op$ and $\Om$, and $f(x,y)$ may be discontinuous across the interface curve $\Gamma$.

\item All functions $\psi(x,y)$, $g_1(x,y)$, $g_2(x,y)$  and $g(x,y)$ are smooth.

\item The exact solution $u(x,y)$ is piecewise smooth in the sense that $u(x,y)$ has uniformly continuous partial derivatives of (total) order up to five in each of the subregions $\Op$ and $\Om$.
\end{itemize}

The paper is organized as follows.
In \cref{sec:Solution:Regular}, we  construct the fourth order compact finite difference scheme for the numerical solution at regular points. The explicit formulas at regular points are shown in \cref{thm:regular}.
\cref{thm:regular} also shows that the maximum order of compact schemes at regular points is six.
In \cref{sec:Solution:Irregular}, we derive the third order compact finite difference scheme for the numerical solution at irregular points,  and discuss its accuracy order in \cref{thm:irregular}. \cref{thm:Max:Order} proves that the maximum order of compact finite difference schemes at irregular points is three.

The explicit formulas for the gradient approximation at regular and irregular points are shown in \cref{thm:gradient:regular} and  \cref{thm:gradient:irregular}, respectively. Furthermore, \cref{thm:gradient:regular} shows that the maximum order of compact schemes for the approximated gradients at regular points is four.  Note that the gradient computation is done explicitly.

In \cref{sec:Numeri}, we provide numerical results to verify the convergence rate measured in the numerical approximated $L^2$ norms for the numerical  solution $u_h$, the gradient approximation $\nabla u_h$, and the flux approximation $a\nabla u_h$. We consider two test cases: (1) the exact solution is known and $\Gamma$ does not intersect $\partial \Omega$ and (2) the exact solution is unknown and $\Gamma$ does not intersect $\partial \Omega$. Since we achieve fourth order at the regular points and third order at the irregular points for the solution and its gradient, the convergence rates for $u_h$, $\nabla u_h$ and $a\nabla u_h$ are between 3 and 4. Note that, we choose the coefficient contrast as $\sup(a_+)/\inf(a_-)=10^{-3},10^{-2},10^{3},10^{4}$ in the numerical tests.

In \cref{sec:Conclu}, we summarize the main contributions of this paper.

\section{Preliminaries}

Since $\Omega=(l_1,l_2)\times (l_3,l_4)$ is a rectangular domain and
we use uniform Cartesian meshes, we can assume that $l_4-l_3=N_0 (l_2-l_1)$ for some positive integer $N_0$.
For any positive integer $N_1\in \mathbb{N}$, we define $N_2:=N_0 N_1$ and then the grid size is  $h:=(l_2-l_1)/N_1=(l_4-l_3)/N_2$.

Let $x_i=l_1+i h$ and $y_j=l_3+j h$ for $i=1,\ldots,N_1-1$, and $j=1,\ldots,N_2-1$.
As in this paper we are interested in compact finite difference schemes on uniform Cartesian grids, the compact scheme involves only nine points $(x_i+kh, y_j+lh)$ for $k,l\in \{-1,0,1\}$.
It is convenient to use a level set function $\psi$, which is a two-dimensional smooth function, to describe a given smooth interface curve $\Gamma$ through
\[
\Gamma:=\{(x,y)\in \Omega \; : \; \psi(x,y)=0\}.
\]
Then the interface curve $\Gamma$ splits the problem domain $\Omega$ into two subregions:
$\Op:=\{ (x,y)\in \Omega \; : \; \psi(x,y)>0\}$ and
$\Om:=\{ (x,y)\in \Omega \; : \; \psi(x,y)<0\}$.
Now the interface curve $\Gamma$ splits these nine points into two groups depending on whether these points lie inside $\Op$ or $\Om$.
If a grid point lies on the curve $\Gamma$, then the grid point lies on the boundaries of both $\Op$ and $\Om$. For simplicity we may assume that the grid point belongs to $\bar{\Omega}^+$ and we can use the interface conditions to handle such a grid point.
Thus, we naturally define
\[
d_{i,j}^+:=\{(k,\ell) \; : \;
k,\ell\in \{-1,0,1\}, \psi(x_i+kh, y_j+\ell h)\ge 0\}
\]
and
\[
d_{i,j}^-:=\{(k,\ell) \; : \;
k,\ell\in \{-1,0,1\}, \psi(x_i+kh, y_j+\ell h)<0\}.
\]
That is, the interface curve $\Gamma$ splits the nine points of a compact scheme into two disjoint sets  $\{(x_{i+k}, y_{j+\ell})\; : \; (k,\ell)\in d_{i,j}^+\} \subseteq \Op$ and
$\{(x_{i+k}, y_{j+\ell})\; : \; (k,\ell)\in d_{i,j}^-\} \subseteq \Om$.
We say that a grid/center point $(x_i,y_j)$ is
\emph{a regular point} if  $d_{i,j}^+=\emptyset$ or $d_{i,j}^-=\emptyset$.
That is, the center point $(x_i,y_j)$  is regular if all its nine points are completely inside $\Op$ (hence $d_{i,j}^-=\emptyset$) or inside $\Om$ (i.e., $d_{i,j}^+=\emptyset$).
Otherwise, the center point $(x_i,y_j)$ is called \emph{an irregular point}
if $d_{i,j}^+\ne \emptyset$ and $d_{i,j}^-\ne \emptyset$. That is, the interface curve $\Gamma$ splits the nine points into two disjoint nonempty sets.

Before we discuss the compact  schemes at a regular or an irregular point $(x_i,y_j)$, let us introduce some notations.
We first pick up and fix a base point $(x_i^*,y_j^*)$ inside the open square $(x_i-h,x_i+h)\times (y_j-h,y_j+h)$, i.e., we can say
\be \label{base:pt}
x_i^*=x_i-v_0h  \quad \mbox{and}\quad y_j^*=y_j-w_0h  \quad \mbox{with}\quad
-1<v_0, w_0<1.
\ee
For simplicity, we shall use the following notions:
\be \label{ufmn}
a^{(m,n)}:=\frac{\partial^{m+n} a}{ \partial^m x \partial^n y}(x_i^*,y_j^*), \quad
u^{(m,n)}:=\frac{\partial^{m+n} u}{ \partial^m x \partial^n y}(x_i^*,y_j^*)
 \quad\mbox{and}\quad
f^{(m,n)}:=\frac{\partial^{m+n} f}{ \partial^m x \partial^n y}(x_i^*,y_j^*),
\ee
which are just their $(m,n)$th partial derivatives at the base point $(x_i^*,y_j^*)$.
Define $\NN:=\N\cup\{0\}$, the set of all nonnegative integers.
For a nonnegative integer $K\in \NN$, we define
\be \label{Sk}
\ind_{K}:=\{(m,n-m) \; : \; n=0,\ldots,K
\; \mbox{ and }\; m=0,\ldots,n\}, \qquad K\in \NN.
\ee
For a smooth function $u$, its value $u(x+x_i^*,y+y_j^*)$ for small $x,y$ can be well approximated through its Taylor polynomial below:
\be \label{u:approx}
u(x+x_i^*,y+y_j^*)=
\sum_{(m,n)\in \ind_{M+1}} \frac{u^{(m,n)}}{m!n!}x^m y^{n}
+\bo(h^{M+2}), \qquad x, y \in (-2h,2h).
\ee
In other words, in a neighborhood of the base point $(x_i^*,y_j^*)$,
the function $u$ is well approximated and completely determined by the partial derivatives of $u$ of total degree less than $M+2$ at the base point $(x_i^*,y_j^*)$, i.e.,  by the unknown quantities $u^{(m,n)}, (m,n)\in \ind_{M+1}$.
$a(x+x_i^*,y+y_j^*)$ and $f(x+x_i^*,y+y_j^*)$ can  be approximated similarly for small $x,y$.
%
%
For $x\in \R$, the floor function $\lfloor x\rfloor$ is defined to be the largest integer less than or equal to $x$.
For an integer $m$, we define
\[
\odd(m):=\begin{cases}
0, &\text{if $m$ is even},\\
1, &\text{if $m$ is odd}.
\end{cases}
\]
That is, $\odd(m)=m-2\lfloor m/2\rfloor$ and $\lfloor m/2\rfloor=\frac{m-\odd(m)}{2}$.
Since the function $u$ is a solution for the partial differential equation in \eqref{Qeques1}, all the quantities $u^{(m,n)}, (m,n)\in \ind_{M+1}$ are not independent of each other.
Similar to the Lemma 2.1 in \cite{FengHanMinev21}, we have the following result:

\begin{lemma}\label{lem:uderiv}
Let $u$ be a function satisfying $-\nabla \cdot (a \nabla u)=f$ in $\Omega\setminus \Gamma$.
If a point $(x_i^*,y_j^*)\in \Omega\setminus \Gamma$, then
\be \label{uderiv:relation}
\begin{split}
u^{(m',n')}&
=(-1)^{\lfloor\frac{m'}{2}\rfloor}
u^{(\odd(m'),n'+m'-\odd(m'))}+
\sum_{(m,n)\in \ind_{m'+n'-1}^1}
A^{u}_{m',n',m,n}u^{(m,n)} \\
&+\sum_{\ell=1}^{\lfloor m'/2\rfloor}
\frac{(-1)^{\ell} f^{(m'-2\ell,n'+2\ell-2)}}{a^{(0,0)}}+\sum_{(m,n)\in \ind_{m'+n'-3}}
A^{f}_{m',n',m,n}f^{(m,n)} ,\qquad \forall\; (m',n')\in \ind_{M+1}^2,
\end{split}
\ee
where the subsets $\ind_{M+1}^1$ and $\ind_{M+1}^2$ of $\ind_{M+1}$ are defined by
\be \label{Sk1}
\ind_{M+1}^2:=\ind_{M+1}\setminus \ind_{M+1}^1\quad \mbox{with}\quad
\ind_{M+1}^1:=\{(\ell,k-\ell) \; :   k=\ell,\ldots, M+1-\ell\; \mbox{and} \;\ell=0,1\; \},
\ee
and
\be\label{AuAf1}
A^{u}_{m',n',m,n}=\frac{1}{(a^{(0,0)})^{d^u_{m',n',m,n}}} \sum_{k}  C^{u}_{m',n',m,n,k} \Bigg(\prod_{(i,j)\in \ind_{m'+n'-1}}\big(a^{(i,j)}\big)^{d^u_{m',n',m,n,i,j,k}}\Bigg),
\ee
\be\label{AuAf2}
A^{f}_{m',n',m,n}=\frac{1}{(a^{(0,0)})^{d^f_{m',n',m,n}}}  \sum_{k}  C^{f}_{m',n',m,n,k} \Bigg(\prod_{(i,j)\in \ind_{m'+n'-3}}\big(a^{(i,j)}\big)^{d^f_{m',n',m,n,i,j,k}}\Bigg),
\ee
where all $d^u_{m',n',m,n}$, $d^f_{m',n',m,n}$, $d^u_{m',n',m,n,i,j,k}$ and $d^f_{m',n',m,n,i,j,k}$ are no-negative integers, $C^{u}_{m',n',m,n,k}$ and $C^{f}_{m',n',m,n,k}$ are two constants. All above constants
are uniquely determined by the identity in \eqref{u20}.
\end{lemma}
See \cref{fig:umn1} and \cref{fig:amn1ANDfmn1} for an illustration
of the quantities $u^{(m,n)}$ with $(m,n)\in \ind_{M+1}^1$,
$u^{(m,n)}$ with $(m,n)\in \ind_{M+1}^2$, $a^{(m,n)}$ with $(m,n)\in \ind_{M}$ and $f^{(m,n)}$ with $(m,n)\in \ind_{M-1}$ in \cref{lem:uderiv} with $M=4$.
\begin{proof} By our assumption, we have $au_{xx}+au_{yy}+a_xu_x+a_yu_y=-f$ in $\Omega\setminus \Gamma$, i.e.,
		\be \label{u20}
	 u^{(2,0)}=-\frac{a^{(1,0)}u^{(1,0)}+a^{(0,1)}u^{(0,1)}}{a^{(0,0)}}-u^{(0,2)}-\frac{f^{(0,0)}}{a^{(0,0)}}.
	\ee
Then it is clear that for all $2+n'\le M+1$,
\[
u^{(2,n')}=-u^{(0,n'+2)}+\sum_{(m,n)\in \ind_{n'+1}^1}
A^{u}_{2,n',m,n}u^{(m,n)}
-\frac{f^{(0,n')}}{a^{(0,0)}} +\sum_{(m,n)\in \ind_{n'-1}}
A^{f}_{2,n',m,n}f^{(m,n)}.
\]
	where $A^{u}_{2,n',m,n}$ and $A^{f}_{2,n',m,n}$ are defined in \eqref{AuAf1} and \eqref{AuAf2} respectively.
	Similarly to \eqref{u20}, we have  $a_xu_{xx}+au_{xxx}+a_xu_{yy}+au_{xyy}+a_{xx}u_x+a_xu_{xx}+a_{xy}u_y+a_yu_{xy}=-f_x$ in $\Omega\setminus \Gamma$. So
		\be \label{u30}
u^{(3,0)}
=\frac{2a^{(1,0)}u^{(2,0)}+a^{(1,0)}u^{(0,2)}+a^{(2,0)}u^{(1,0)}+a^{(1,1)}u^{(0,1)}+a^{(0,1)}u^{(1,1)}}{-a^{(0,0)}}-u^{(1,2)}-\frac{f^{(1,0)}}{a^{(0,0)}}.
\ee
Plugging \eqref{u20} into the right hand of \eqref{u30}, we obtain
\[
u^{(3,0)}=-u^{(1,2)}+\sum_{(m,n)\in \ind_{2}^1}
A^{u}_{3,0,m,n}u^{(m,n)} -\frac{f^{(1,0)}}{a^{(0,0)}} +\sum_{(m,n)\in \ind_{0}}
A^{f}_{3,0,m,n}f^{(m,n)}.
\]
Then for all $3+n'\le M+1$,
\[
u^{(3,n')}=-u^{(1,n'+2)}+\sum_{(m,n)\in \ind_{n'+2}^1}
A^{u}_{3,n',m,n}u^{(m,n)}
 -\frac{f^{(1,n')}}{a^{(0,0)}}  +\sum_{(m,n)\in \ind_{n'}}
A^{f}_{3,n',m,n}f^{(m,n)}.
\]
	Calculate the left $u^{(m',n')}, (m',n')\in \ind_{M+1}^2$ by the order $\{u^{(4,0)},u^{(4,1)},\dots, u^{(4,M-3)}\}$,
	$\{u^{(5,0)},u^{(5,1)},\dots,$\\ $u^{(5,M-4)}\}$,  $\dots$, $\{u^{(M+1,0)}\}$ and use the above identities recursively, to obtain \eqref{uderiv:relation}.
\end{proof}
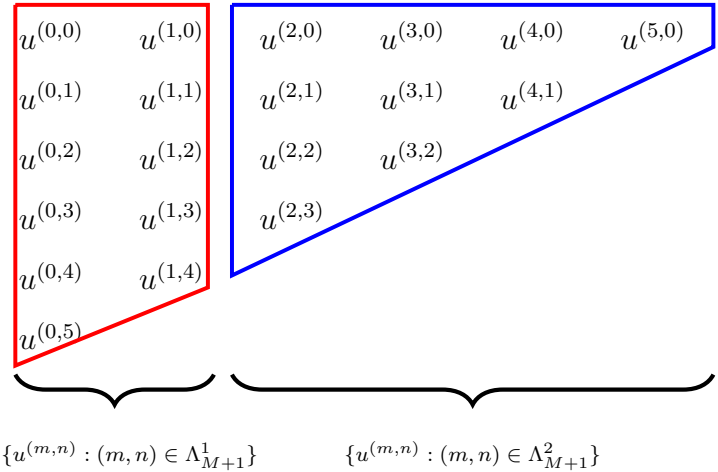
\begin{figure}[h]
	\centering
	\hspace{12mm}	
	\begin{tikzpicture}[scale = 0.8]
		\node (A) at (0,7) {$u^{(0,0)}$};
		\node (A) at (0,6) {$u^{(0,1)}$};
		\node (A) at (0,5) {$u^{(0,2)}$};
		\node (A) at (0,4) {$u^{(0,3)}$};
		\node (A) at (0,3) {$u^{(0,4)}$};
		\node (A) at (0,2) {$u^{(0,5)}$};
		\node (A) at (2,7) {$u^{(1,0)}$};
		\node (A) at (2,6) {$u^{(1,1)}$};
		\node (A) at (2,5) {$u^{(1,2)}$};
		\node (A) at (2,4) {$u^{(1,3)}$};
		\node (A) at (2,3) {$u^{(1,4)}$};
		\node (A) at (4,7) {$u^{(2,0)}$};
		\node (A) at (4,6) {$u^{(2,1)}$};
		\node (A) at (4,5) {$u^{(2,2)}$};
		\node (A) at (4,4) {$u^{(2,3)}$};
		\node (A) at (6,7) {$u^{(3,0)}$};
		\node (A) at (6,6) {$u^{(3,1)}$};
		\node (A) at (6,5) {$u^{(3,2)}$};
		\node (A) at (8,7) {$u^{(4,0)}$};
		\node (A) at (8,6) {$u^{(4,1)}$};
		\node (A) at (10,7) {$u^{(5,0)}$};    	
		\draw[line width=1.5pt, red]  plot [tension=0.8]
		coordinates {(-0.6,7.5) (-0.6,1.5) (2.6,2.8) (2.6,7.5) (-0.6,7.5)};	
		\draw[line width=1.5pt, blue]  plot [tension=0.8]
		coordinates {(3,7.5)  (3,3) (11,6.8) (11,7.5) (3,7.5)};
		 \draw[decorate,decoration={brace,mirror,amplitude=4mm},xshift=0pt,yshift=10pt,ultra thick] (-0.6,1) -- node [black,midway,yshift=0.6cm]{} (2.7,1);
    	\node (A) at (1.3,0)	 {\tiny{$\{u^{(m,n)}:(m,n)\in \ind_{M+1}^1\}$}};
        \draw[decorate,decoration={brace,mirror,amplitude=4mm},xshift=0pt,yshift=10pt,ultra thick] (3,1) -- node [black,midway,yshift=0.6cm]{} (11,1);
       \node (A) at (7,0)	 {\tiny{$\{u^{(m,n)}:(m,n)\in \ind_{M+1}^2\}$}};
	\end{tikzpicture}
	\caption
	{Red trapezoid: $\{u^{(m,n)}:(m,n)\in \ind_{M+1}^1\}$ with $M=4$. Blue trapezoid: $\{u^{(m,n)}:(m,n)\in \ind_{M+1}^2\}$ with $M=4$. Note that $\ind_{M+1}=\ind_{M+1}^1 \cup\ind_{M+1}^2$.}
	\label{fig:umn1}
\end{figure}
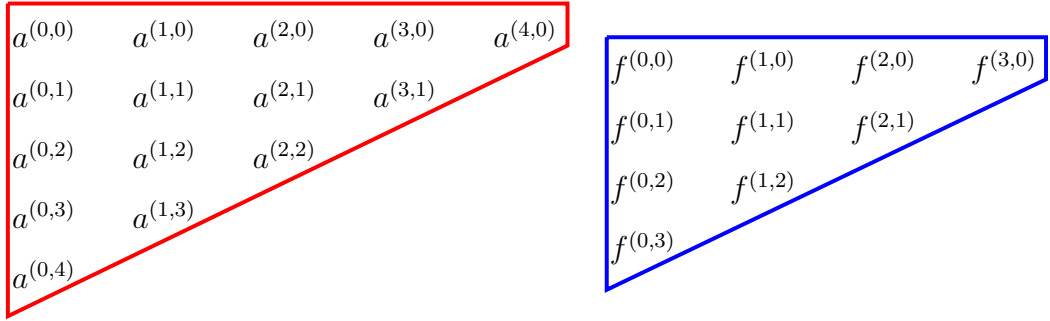
\begin{figure}[h]
	\centering
	\hspace{12mm}	
	\begin{tikzpicture}[scale = 0.8]
		\node (A) at (0,7) {$a^{(0,0)}$};
		\node (A) at (0,6) {$a^{(0,1)}$};
		\node (A) at (0,5) {$a^{(0,2)}$};
		\node (A) at (0,4) {$a^{(0,3)}$};
		\node (A) at (0,3) {$a^{(0,4)}$};
		\node (A) at (2,7) {$a^{(1,0)}$};
		\node (A) at (2,6) {$a^{(1,1)}$};
		\node (A) at (2,5) {$a^{(1,2)}$};
		\node (A) at (2,4) {$a^{(1,3)}$};
		\node (A) at (4,7) {$a^{(2,0)}$};
		\node (A) at (4,6) {$a^{(2,1)}$};
		\node (A) at (4,5) {$a^{(2,2)}$};
		\node (A) at (6,7) {$a^{(3,0)}$};
		\node (A) at (6,6) {$a^{(3,1)}$};
		\node (A) at (8,7) {$a^{(4,0)}$};  	
		\draw[line width=1.5pt, red]  plot [tension=0.8]
		coordinates {(-0.6,7.5) (-0.6,2.3) (8.7,6.8) (8.7,7.5) (-0.6,7.5)};
	\end{tikzpicture}
	\begin{tikzpicture}[scale = 0.8]
		\node (A) at (0,7) {$f^{(0,0)}$};
		\node (A) at (0,6) {$f^{(0,1)}$};
		\node (A) at (0,5) {$f^{(0,2)}$};
		\node (A) at (0,4) {$f^{(0,3)}$};
		\node (A) at (0,3) {$\hspace{14mm}$};
		\node (A) at (2,7) {$f^{(1,0)}$};
		\node (A) at (2,6) {$f^{(1,1)}$};
		\node (A) at (2,5) {$f^{(1,2)}$};
		\node (A) at (2,4) {$\hspace{14mm}$};
		\node (A) at (4,7) {$f^{(2,0)}$};
		\node (A) at (4,6) {$f^{(2,1)}$};
		\node (A) at (4,5) {$\hspace{14mm}$};
		\node (A) at (6,7) {$f^{(3,0)}$};
		\node (A) at (6,6) {$\hspace{14mm}$};	
		\draw[line width=1.5pt, blue]  plot [tension=0.8]
		coordinates {(-0.6,7.5) (-0.6,3.3) (6.7,6.8) (6.7,7.5) (-0.6,7.5)};
	\end{tikzpicture}
	\caption
	{Red trapezoid: $\{a^{(m,n)}:(m,n)\in \ind_{M}\}$ with $M=4$. Blue trapezoid: $\{f^{(m,n)}:(m,n)\in \ind_{M-1}\}$ with $M=4$.}
	\label{fig:amn1ANDfmn1}
\end{figure}
Note that $\ind_{M+1}=\ind_{M+1}^1 \cup\ind_{M+1}^2$.
The identities in \eqref{uderiv:relation} of \cref{lem:uderiv}
show that every $u^{(m,n)}$ in
$\{u^{(m,n)}: (m,n)\in \ind_{M+1}\}$
can be written as a combination of
the quantities $\{u^{(m,n)}: (m,n)\in \ind_{M+1}^1\}$, $\{a^{(m,n)}: (m,n)\in \ind_{M}\}$ and $\{f^{(m,n)}: (m,n)\in \ind_{M-1}\}$.
As the coefficient $a$ and the source term $f$  are available in \eqref{Qeques1}, \eqref{uderiv:relation} could reduce the number of constraints on $\{u^{(m,n)}: (m,n)\in \ind_{M+1}\}$ to
$\{u^{(m,n)}: (m,n)\in \ind_{M+1}^1\}$.
By (2.9) and (2.10) in \cite{FengHanMinev21} and \eqref{uderiv:relation} of this paper, the approximation of $u(x+x_i^*,y+y_j^*)$ in \eqref{u:approx} can be written as
{\footnotesize{
\begin{align*}
&\sum_{(m,n)\in \ind_{M+1}} \frac{u^{(m,n)}}{m!n!} x^m y^{n}=\sum_{(m,n)\in \ind_{M+1}^1} \frac{u^{(m,n)}}{m!n!} x^m y^{n}+\sum_{(m',n')\in \ind_{M+1}^2} \frac{u^{(m',n')}}{m'!n'!} x^{m'} y^{n'}\\
&=\sum_{(m',n')\in \ind_{M+1}^2} \sum_{(m,n)\in \ind_{m'+n'-1}^1}
A^{u}_{m',n',m,n}u^{(m,n)}\frac{ x^{m'} y^{n'}}{m'!n'!}  +\sum_{(m',n')\in \ind_{M+1}^2} \sum_{(m,n)\in \ind_{m'+n'-3}}
A^{f}_{m',n',m,n}f^{(m,n)}\frac{ x^{m'} y^{n'}}{m'!n'!}\\
&+ \sum_{(m',n')\in \ind_{M+1}^2} (-1)^{\lfloor\frac{m'}{2}\rfloor}
u^{(\odd(m'),n'+m'-\odd(m'))} \frac{ x^{m'} y^{n'}}{m'!n'!}+\sum_{(m',n')\in \ind_{M+1}^2} \sum_{\ell=1}^{\lfloor m'/2\rfloor}
\frac{(-1)^{\ell} f^{(m'-2\ell,n'+2\ell-2)}}{a^{(0,0)}}\frac{ x^{m'} y^{n'}}{m'!n'!}\\
&+\sum_{(m,n)\in \ind_{M+1}^1} \frac{u^{(m,n)}}{m!n!} x^m y^{n}\\
&= \sum_{(m,n)\in \ind_{M}^1}\Bigg(\sum_{(m',n')\in \ind_{M+1}^2 \atop m'+n' \ge m+n+1 }
A^{u}_{m',n',m,n}\frac{ x^{m'} y^{n'}}{m'!n'!}\Bigg)  u^{(m,n)}+ \sum_{(m,n)\in \ind_{M-2}} \Bigg(\sum_{(m',n')\in \ind_{M+1}^2 \atop m'+n' \ge m+n+3}
A^{f}_{m',n',m,n}\frac{ x^{m'} y^{n'}}{m'!n'!}\Bigg)f^{(m,n)}\\
&+\sum_{(m,n)\in \ind_{M+1}^1}\Bigg(\sum_{\ell=0}^{\lfloor \frac{n}{2}\rfloor}
\frac{(-1)^\ell x^{m+2\ell} y^{n-2\ell}}{(m+2\ell)!(n-2\ell)!}\Bigg)  u^{(m,n)}+\sum_{(m,n)\in \ind_{M-1}}\Bigg( \sum_{\ell=1-\lfloor \frac{m}{2}\rfloor}^{1+\lfloor \frac{n}{2}\rfloor} \frac{(-1)^{\ell} x^{m+2\ell} y^{n-2\ell+2}}{(m+2\ell)!(n-2\ell+2)!}\frac{1}{a^{(0,0)}} \Bigg) f^{(m,n)}\\
&=\sum_{(m,n)\in \ind_{M+1}^1}
u^{(m,n)} G_{m,n}(x,y) +\sum_{(m,n)\in \ind_{M-1}}
f^{(m,n)} H_{m,n}(x,y),
\end{align*}
}}
where for all $(m,n)\in \ind_{M+1}^1$,
\be\label{Gmn}
G_{m,n}(x,y):=\sum_{\ell=0}^{\lfloor \frac{n}{2}\rfloor}
 \frac{(-1)^\ell x^{m+2\ell} y^{n-2\ell}}{(m+2\ell)!(n-2\ell)!}+\sum_{(m',n')\in \ind_{M+1}^2 \setminus \ind_{m+n}^2 }A^{u}_{m',n',m,n} \frac{ x^{m'} y^{n'}}{m'!n'!}
\ee
and for all $(m,n)\in \ind_{M-1}$,
\be\label{Hmn}
\begin{split}
H_{m,n}(x,y):=\sum_{\ell=1-\lfloor \frac{m}{2}\rfloor}^{1+\lfloor \frac{n}{2}\rfloor} \frac{(-1)^{\ell} x^{m+2\ell} y^{n-2\ell+2}}{(m+2\ell)!(n-2\ell+2)!}\frac{1}{a^{(0,0)}}
+\sum_{(m',n')\in \ind_{M+1}^2 \setminus \ind_{m+n+2}^2 }A^{f}_{m',n',m,n} \frac{ x^{m'} y^{n'}}{m'!n'!}.
\end{split}
\ee
From \eqref{Gmn} and  \eqref{Hmn}, we observe that
$G_{m,n}(x,y)$ and $H_{m,n}(x,y)$ are  homogeneous polynomials of total degree $M+1$ for all $(m,n)\in \ind_{M+1}^1$ and all $(m,n)\in \ind_{M-1}$, respectively. Moreover, each entry of ${G}_{m,n}(x,y)$ is a homogeneous polynomial of degree $\ge m+n$ and each entry of ${H}_{m,n}(x,y)$ is a homogeneous polynomial of degree $\ge m+n+2$.
Thus, the approximation in \eqref{u:approx} becomes
\be \label{u:approx:key}
u(x+x_i^*,y+y_j^*)
=
\sum_{(m,n)\in \ind_{M+1}^1}
u^{(m,n)} G_{m,n}(x,y) +\sum_{(m,n)\in \ind_{M-1}}
f^{(m,n)} H_{m,n}(x,y)+\bo(h^{M+2}),
\ee
for $x,y\in (-2h,2h)$, where  $u$ is the exact solution for \eqref{Qeques1} and  $(x_i^*,y_j^*)$ is the base point.  Note that \eqref{u:approx:key} is the key point to derive compact difference schemes for regular and irregular points with the maximum accuracy order.

\section{A high order compact finite difference scheme for computing $u$}
\label{sec:Solution}
In this section, we construct the compact finite difference scheme for the numerical solution of the elliptic equation at regular and irregular points.
\subsection{Regular points\label{sec:Solution:Regular}}
In this subsection, we discuss the derivation of a compact scheme centered at a regular point $(x_i,y_j)$.
For the sake of brevity, we choose $(x_i^*,y_j^*)=(x_i,y_j)$, i.e.,
$(x_i^*,y_j^*)$ is defined in \eqref{base:pt} with $v_0=w_0=0$. Consider the following equation:
\be \label{stencil:s1s2:regular1}
\begin{split}
\sum_{k=-1}^1 \sum_{\ell=-1}^1
	C_{k,\ell}(h) u(x_i+kh,y_j+\ell h)=
	\sum_{(m,n)\in \ind_{M-1}} f^{(m,n)}C_{f,m,n}(h)+\bo(h^{M+2}),\qquad h\to 0,
\end{split}
\ee
where  $u(x,y)$ is defined in \eqref{u:approx:key}, the nontrivial $C_{k,\ell}(h)$ and $C_{f,m,n}(h)$ are to-be-determined polynomials of $h$ with degree less than $M+2$. Precisely,
\be\label{Cfij:regular}
C_{k,\ell}(h)=\sum_{i=0}^{M+1} c_{k,\ell,i}h^i, \qquad C_{f,m,n}(h)=\sum_{j=0}^{M+1} c_{f,m,n,j}h^j,
\ee
where all $c_{k,\ell,i}$ and $c_{f,m,n,j}$ are to-be-determined constants. Similar to \cite{FengHanMinev21}, we observe that the coefficients of a compact scheme are nontrivial if $C_{k,\ell}(0)\ne 0$ for at least some $k,\ell=-1,0,1$, that is, $c_{k,\ell,0}\ne 0$ for at least some $k,\ell=-1,0,1$.
Similar to  Eq.(7.31) to Eq.(7.34) in \cite[Section~7.2.1]{LiIto06}, Eq.(5) and Eq.(6) in \cite{WZ09} and Theorem 3.2 in \cite{HanMW2021}, \eqref{stencil:s1s2:regular1} and \eqref{Cfij:regular} together imply
\[
-\nabla \cdot( a\nabla u)\Big|_{(x,y)=(x_i,y_j)}=f\Big|_{(x,y)=(x_i,y_j)}+\bo(h^{M}),\qquad h\to 0.
\]
Thus, we can achieve an  accuracy order $M$ for the numerical approximated solution.

{
Substituting \eqref{u:approx:key} into \eqref{stencil:s1s2:regular1} with $(x_i^*,y_j^*)=(x_i,y_j)$, we obtain:
\[\begin{split}
	\sum_{k=-1}^1 \sum_{\ell=-1}^1
	C_{k,\ell}(h)\Big(\sum_{(m,n)\in \ind_{M+1}^1}
	u^{(m,n)} G_{m,n}(kh, \ell h) &+\sum_{(m,n)\in \ind_{M-1}}
	f^{(m,n)} H_{m,n}(kh, \ell h) \Big)\\
	&=\sum_{(m,n)\in \ind_{M-1}} f^{(m,n)}C_{f,m,n}(h)+\bo(h^{M+2}),\qquad h\to 0,
\end{split}\]
\[\begin{split}
\sum_{(m,n)\in \ind_{M+1}^1} \Big(	 \sum_{k=-1}^1 \sum_{\ell=-1}^1
C_{k,\ell}(h)
	 G_{m,n}(kh, \ell h)\Big)u^{(m,n)} &+\sum_{(m,n)\in \ind_{M-1}} 	 \Big(\sum_{k=-1}^1 \sum_{\ell=-1}^1
	C_{k,\ell}(h)
	 H_{m,n}(kh, \ell h)\Big) f^{(m,n)}\\
	&=\sum_{(m,n)\in \ind_{M-1}} f^{(m,n)}C_{f,m,n}(h)+\bo(h^{M+2}),\qquad h\to 0,
\end{split}\]
Thus,
the conditions in \eqref{stencil:s1s2:regular1} can be rewritten as
\be \label{stencil:s1s2:regular:2}
\sum_{(m,n)\in \ind_{M+1}^1}
u^{(m,n)} I_{m,n}(h)+
\sum_{(m,n)\in \ind_{M-1}} f^{(m,n)}
\left(J_{m,n}(h)-C_{f,m,n}(h)\right)
=\bo(h^{M+2}),
\ee
where
%
\be \label{J:s1s2:regular}
I_{m,n}(h):=\sum_{k=-1}^1 \sum_{\ell=-1}^1 C_{k,\ell}(h) G_{m,n}(kh, \ell h)
\quad \mbox{and}\quad
%
J_{m,n}(h):=\sum_{k=-1}^1 \sum_{\ell=-1}^1 C_{k,\ell}(h) H_{m,n}(kh, \ell h).
\ee
}

Because \eqref{stencil:s1s2:regular:2} must hold for all unknowns in $\{u^{(m,n)}: (m,n)\in \ind_{M+1}^1\}$, to find the nontrivial  $C_{k,\ell}(h)$ for $k,\ell=-1,0,1$, solving \eqref{stencil:s1s2:regular:2} is equivalent to solving
\be \label{stencil:s1s2:regular:u}
I_{m,n}(h)=\bo(h^{M+2}) \quad h\to 0,\; \mbox{ for all }\; (m,n)\in \ind_{M+1}^1,
\ee
and
\be \label{stencil:s1s2:regular:f}
C_{f,m,n}(h)=J_{m,n}(h) +\bo(h^{M+2}),\qquad h\to 0, \; \mbox{ for all }\; (m,n)\in \ind_{M-1}.
\ee

By calculation, the largest integer $M$ for the linear system in \eqref{stencil:s1s2:regular:u} to have a nontrivial solution $\{C_{k,\ell}(h)\}_{k,\ell=-1,0,1}$ is $M=6$. Because in this paper we are only interested in $M=4$, one nontrivial solution $\{C_{k,\ell}(h)\}_{k,\ell=-1,0,1}$ to \eqref{stencil:s1s2:regular:u} with $M=4$ is explicitly given by
{\tiny{
\be\label{Ch1:M=4}
\begin{split}
C_{-1,-1}(h)
& = (((2a^{(1, 2)}+a^{(2, 1)}+a^{(0, 3)}+2a^{(3, 0)})a^{(1, 0)}-a^{(0, 1)}(a^{(2, 1)}+a^{(0, 3)}))(a^{(0, 0)})^2+((-2a^{(0, 2)}-8a^{(2, 0)}\\
&-3a^{(1, 1)})(a^{(1, 0)})^2-(2(a^{(0, 2)}+(3/2)a^{(2, 0)}+(3/2)a^{(1, 1)}))a^{(0, 1)}a^{(1, 0)}+2(a^{(0, 1)})^2(a^{(0, 2)}+(3/2)a^{(2, 0)}))a^{(0, 0)}\\
&-(a^{(0, 1)})^4-2(a^{(0, 1)})^2(a^{(1, 0)})^2+4a^{(0, 1)}(a^{(1, 0)})^3+7(a^{(1, 0)})^4)h^4-a^{(0, 0)}((a^{(1, 2)}-a^{(2, 1)}-a^{(0, 3)}+a^{(3, 0)})(a^{(0, 0)})^2\\
&+((a^{(0, 2)}-6a^{(2, 0)}+3a^{(1, 1)})a^{(1, 0)}+(5(a^{(2, 0)}-(3/5)a^{(1, 1)}))a^{(0, 1)})a^{(0, 0)}+(-a^{(1, 0)}+a^{(0, 1)})((a^{(0, 1)})^2+3a^{(0, 1)}a^{(1, 0)}\\
&-6(a^{(1, 0)})^2))h^3+(-2(a^{(0, 1)})^2+2(a^{(1, 0)})^2)(a^{(0, 0)})^2h^2-4a^{(1, 0)}h(a^{(0, 0)})^3+4(a^{(0, 0)})^4,
\end{split}
\ee
}}
{\tiny{
		\be\label{Ch2:M=4}
		\begin{split}
C_{-1,0}(h) &= (((-a^{(1, 2)}+a^{(2, 1)}+a^{(0, 3)}-a^{(3, 0)})a^{(0, 1)}-a^{(1, 0)}(a^{(2, 1)}+a^{(0, 3)}))(a^{(0, 0)})^2+((-2a^{(0, 2)}-3a^{(2, 0)}+a^{(1, 1)})(a^{(0, 1)})^2\\
&+3a^{(1, 0)}(a^{(0, 2)}+(7/3)a^{(2, 0)}+(1/3)a^{(1, 1)})a^{(0, 1)}+(a^{(1, 0)})^2a^{(1, 1)})a^{(0, 0)}+a^{(0, 1)}((a^{(0, 1)})^3-(a^{(0, 1)})^2a^{(1, 0)}\\
&+3a^{(0, 1)}(a^{(1, 0)})^2-7(a^{(1, 0)})^3))h^4-2a^{(0, 0)}((a^{(2, 1)}+a^{(0, 3)})(a^{(0, 0)})^2+(-5a^{(0, 1)}a^{(2, 0)}-2a^{(1, 0)}(a^{(0, 2)}-a^{(2, 0)}\\
&+(3/2)a^{(1, 1)}))a^{(0, 0)}-(a^{(0, 1)})^3-(a^{(0, 1)})^2a^{(1, 0)}+7a^{(0, 1)}(a^{(1, 0)})^2-(a^{(1, 0)})^3)h^3-8a^{(1, 0)}(-a^{(1, 0)}+a^{(0, 1)})h^2(a^{(0, 0)})^2\\
&+(8(a^{(0, 1)}-2a^{(1, 0)}))h(a^{(0, 0)})^3+16(a^{(0, 0)})^4,
		\end{split}
		\ee
}}
{\tiny{
		\be\label{Ch3:M=4}
		\begin{split}
C_{-1,1}(h)& = -a^{(0, 0)}(((a^{(1, 2)}-a^{(2, 1)}-a^{(0, 3)}+a^{(3, 0)})(a^{(0, 0)})^2+((a^{(0, 2)}-6a^{(2, 0)}-a^{(1, 1)})a^{(1, 0)}+5a^{(0, 1)}(a^{(2, 0)}\\
&+(1/5)a^{(1, 1)}))a^{(0, 0)}+(a^{(0, 1)})^3-7a^{(0, 1)}(a^{(1, 0)})^2+6(a^{(1, 0)})^3)h^3+(4(a^{(0, 0)})^2a^{(1, 1)}-2a^{(0, 0)}(a^{(1, 0)})^2)h^2-(4(-a^{(1, 0)}\\
&+a^{(0, 1)}))h(a^{(0, 0)})^2-4(a^{(0, 0)})^3),
		\end{split}
		\ee
}}
{\tiny{
		\be\label{Ch4:M=4}
		\begin{split}
C_{0,-1}(h)& = ((a^{(0, 1)}-2a^{(1, 0)})(a^{(1, 2)}+a^{(3, 0)})(a^{(0, 0)})^2+((2a^{(0, 2)}+8a^{(2, 0)}+2a^{(1, 1)})(a^{(1, 0)})^2-a^{(0, 1)}(a^{(0, 2)}+4a^{(2, 0)}\\
&-2a^{(1, 1)})a^{(1, 0)}-(a^{(0, 1)})^2a^{(1, 1)})a^{(0, 0)}+a^{(1, 0)}((a^{(0, 1)})^3-(a^{(0, 1)})^2a^{(1, 0)}+3a^{(0, 1)}(a^{(1, 0)})^2-7(a^{(1, 0)})^3))h^4\\
&-2a^{(0, 0)}((a^{(2, 1)}+a^{(0, 3)})(a^{(0, 0)})^2+(-3a^{(1, 0)}a^{(1, 1)}-2a^{(0, 1)}(a^{(0, 2)}+(3/2)a^{(2, 0)}))a^{(0, 0)}-2a^{(0, 1)}a^{(1, 0)}(a^{(0, 1)}\\
&-3a^{(1, 0)}))h^3+4(a^{(0, 0)})^2((a^{(0, 2)}-a^{(2, 0)})a^{(0, 0)}-(3/2)(a^{(0, 1)})^2+(3/2)(a^{(1, 0)})^2)h^2-8a^{(1, 0)}h(a^{(0, 0)})^3+16(a^{(0, 0)})^4,
			\end{split}
		\ee
}}
{\tiny{
		\be\label{Ch5:M=4}
		\begin{split}
C_{0,0}(h) & = 2a^{(0, 0)}(((a^{(1, 2)}+a^{(2, 1)}+a^{(0, 3)}+a^{(3, 0)})(a^{(0, 0)})^2+((-a^{(0, 2)}-4a^{(2, 0)}-5a^{(1, 1)})a^{(1, 0)}-2a^{(0, 1)}(a^{(0, 2)}+(3/2)a^{(2, 0)}\\
&+(1/2)a^{(1, 1)}))a^{(0, 0)}-2(a^{(0, 1)})^2a^{(1, 0)}+5a^{(0, 1)}(a^{(1, 0)})^2+5(a^{(1, 0)})^3)h^3-4a^{(0, 0)}((a^{(0, 2)}-a^{(2, 0)}-a^{(1, 1)})a^{(0, 0)}\\
&-(a^{(0, 1)})^2-(3/2)a^{(0, 1)}a^{(1, 0)}+3(a^{(1, 0)})^2)h^2-(20(-a^{(1, 0)}+a^{(0, 1)}))h(a^{(0, 0)})^2-40(a^{(0, 0)})^3),
		\end{split}
		\ee
}}
{\tiny{
		\be\label{Ch6:M=4}
		\begin{split}
			C_{0,1}(h)& = (4(((a^{(0, 2)}-a^{(2, 0)})a^{(0, 0)}+(1/2(-a^{(1, 0)}+a^{(0, 1)}))(a^{(0, 1)}-3a^{(1, 0)}))h^2+(4(a^{(0, 1)}-(1/2)a^{(1, 0)}))a^{(0, 0)}h+4(a^{(0, 0)})^2))\\
			&\times (a^{(0, 0)})^2,
		\end{split}
		\ee
}}
{\tiny{
\be	\label{Ch789:M=4}
	\begin{split}
&C_{1,-1}(h) = -4(a^{(0, 0)})^2(-(a^{(0, 0)})^2+h^2a^{(0, 0)}a^{(1, 1)}+\frac{1}{2}h^2a^{(0, 1)}(a^{(0, 1)}-2a^{(1, 0)})),\\
&C_{1,0}(h) = 8(a^{(0, 0)})^3(ha^{(0, 1)}+2a^{(0, 0)}),\quad
C_{1,1}(h) = 4(a^{(0, 0)})^3(ha^{(0, 1)}+a^{(0, 0)}).
\end{split}
\ee
}}
Substitute \eqref{Ch1:M=4} to \eqref{Ch789:M=4} into \eqref{J:s1s2:regular}.
All $C_{f,m,n}(h)$ satisfying \eqref{stencil:s1s2:regular:f} can be calculated by
\be \label{Cfmn:M=4}
C_{f,m,n}(h):=
\sum_{k=-1}^1 \sum_{\ell=-1}^1 C_{k,\ell}(h)
H_{m,n}(h), \qquad (m,n)\in \ind_3.
\ee
Thus, for a regular point $(x_i,y_j)$, the following theorem proves a fourth order of accuracy for the compact scheme.   This result is well known in the literature (e.g., see \cite{SDKS13,ZFH13,WZ09,WangGuoWu14,Medina2019,MaGe2019,MaGe2020,ZHANG98}).

\begin{theorem}\label{thm:regular}
Let   $(x_i, y_j)$ be a regular point and $(u_{h})_{i,j}$ be the numerical approximation of the exact solution $u$ of the partial differential equation \eqref{Qeques1} at $(x_i, y_j)$. Then
the following compact scheme centered at the regular point $(x_i,y_j):$
\begin{equation}\label{stencil:regular}
\sum_{k=-1}^1 \sum_{\ell=-1}^1 C_{k,\ell}(h)(u_{h})_{i+k,j+\ell}
=C_{f,m,n}(h),
\end{equation}
has a fourth order consistency error at the regular point $(x_i,y_j)$, i.e., the accuracy order for $u_h$ is four, where  $C_{k,\ell}(h)$ are defined in \eqref{Ch1:M=4} to \eqref{Ch789:M=4},  $C_{f,m,n}(h)$ is defined in \eqref{Cfmn:M=4}, $a^{(m,n)}:=\frac{\partial^{m+n} a}{ \partial^m x \partial^n y}(x_i,y_j)$ and $f^{(m,n)}:=\frac{\partial^{m+n} f}{ \partial^m x \partial^n y}(x_i,y_j)$.
Furthermore, the maximum accuracy order $M$ for the numerical approximated solution at the regular point of the compact finite difference scheme is $M=6$.
\end{theorem}

\subsection{Irregular points\label{sec:Solution:Irregular}}

Let $(x_i,y_j)$ be an irregular  point and we can take a base point $(x^*_i,y^*_j)\in \Gamma \cap (x_i-h,x_i+h)\times (y_j-h,y_j+h)$ on the interface $\Gamma$ and inside $(x_i-h,x_i+h)\times (y_j-h,y_j+h)$.
That is, as in \eqref{base:pt},
\be \label{xiyj}
x_i=x_i^*+v_0h  \quad \mbox{and}\quad y_j=y_j^*+w_0h \quad \mbox{with}\quad
-1<v_0, w_0<1 \quad \mbox{and}\quad (x_i^*,y_j^*)\in \Gamma.
\ee
Let $a_{\pm}$, $u_{\pm}$ and $f_{\pm}$ represent the coefficient $a$, the solution $u$ and source term $f$ in $\Omega^{\pm}$.
As in \eqref{ufmn},  we define
\be\label{aufpm}
a_{\pm}^{(m,n)}:=\frac{\partial^{m+n} a_{\pm}}{ \partial^m x \partial^n y}(x^*_i,y^*_j),\qquad  u_{\pm}^{(m,n)}:=\frac{\partial^{m+n} u_{\pm}}{ \partial^m x \partial^n y}(x^*_i,y^*_j),\qquad f_{\pm}^{(m,n)}:=\frac{\partial^{m+n} f_{\pm}}{ \partial^m x \partial^n y}(x^*_i,y^*_j),
\ee
and
\be\label{g1g2pm}
g_1^{(m,n)}:=\frac{\partial^{m+n} g_1}{ \partial^m x \partial^n y}(x^*_i,y^*_j),\qquad
g_2^{(m,n)}:=\frac{\partial^{m+n} g_2}{ \partial^m x \partial^n y}(x^*_i,y^*_j).
\ee
Similarly as the discussion for the irregular points in \cite{FengHanMinev21},
the identities in \eqref{uderiv:relation} and \eqref{u:approx:key} hold by replacing $a$, $u$ and $f$ by $a_\pm$, $u_\pm$ and $f_\pm$, i.e.,
\be \label{u:approx:ir:key}
u_\pm (x+x_i^*,y+y_j^*)
=\sum_{(m,n)\in \ind_{M+1}^1}
u_\pm^{(m,n)} G^{\pm}_{m,n}(x,y) +\sum_{(m,n)\in \ind_{M-1}}
f_\pm ^{(m,n)} H^{\pm}_{m,n}(x,y)+\bo(h^{M+2}),
\ee
for $x,y\in (-2h,2h)$, where the index sets $\ind_{M+1}^1$ and $\ind_{M-1}$ are defined in \eqref{Sk1} and \eqref{Sk}, respectively,  and the polynomials $G^{\pm}_{m,n}(x,y)$ and $H^{\pm}_{m,n}(x,y)$ are defined in  \eqref{Gmn} and \eqref{Hmn} by replacing $a$, $u$ and $f$ by $a_\pm$, $u_\pm$ and $f_\pm$.

{
	In \cref{sec:Solution:Regular}, we use \eqref{stencil:s1s2:regular1} to approximate the operator 	
	$
	-\nabla \cdot( a\nabla u)=f.
	$
	In this section, we need to use the two jump conditions in \eqref{Qeques1}.
	According to \eqref{stencil:s1s2:regular1} and the two jump functions $g_1$ and $g_2$, we consider the following equation:
{\small{\be
		 \begin{split}\label{stencil:irregular:s1s2}
		\sum_{k=-1}^1 \sum_{\ell=-1}^1C_{k,\ell}(h) u(x_i+kh,&y_j+\ell h)=\sum_{(m,n)\in \ind_{M-1}} C_{f_{+},m,n}(h) f_+^{(m,n)}+\sum_{(m,n)\in \ind_{M-1}} C_{f_{-},m,n}(h)f_-^{(m,n)}\\
		&+\sum_{(m,n)\in \ind_{M+1}} C_{g_1,m,n}(h) g_1^{(m,n)}+\sum_{(m,n)\in \ind_{M}} C_{g_2,m,n}(h) g_2^{(m,n)}+\bo(h^{M+2}),
		\end{split}
		\ee
}}
where  $h\to 0$, the nontrivial $C_{k,\ell}(h), C_{f_{\pm},m,n}(h)$,
$C_{g_1,m,n}(h)$ and $C_{g_2,m,n}(h)$
are to-be-determined polynomials of $h$ having degree less than $M+2$. Precisely,
\be\label{Cfij:irregular}
\begin{split}
&C_{k,\ell}(h)=\sum_{i=0}^{M+1} c_{k,\ell,i}h^i, \qquad C_{f_{\pm},m,n}(h)=\sum_{i=0}^{M+1} c_{f_{\pm},m,n,i}h^i,  \\
&C_{g_{1},m,n}(h)=\sum_{i=0}^{M+1} c_{g_{1},m,n,i}h^i, \qquad C_{g_{2},m,n}(h)=\sum_{i=0}^{M+1} c_{g_{2},m,n,i}h^i,
\end{split}
\ee
where all $c_{k,\ell,i}$, $c_{f_{\pm},m,n,i}$, $c_{g_{1},m,n,i}$ and $c_{g_{2},m,n,i}$ are to-be-determined constants. Similarly to \cref{sec:Solution:Regular}, the coefficients of a compact scheme are nontrivial if  $c_{k,\ell,0}\ne 0$ for at least some $k,\ell=-1,0,1$.

Similarly to  the derivation of Eq.(4.37), Theorem 4.1 and Theorem 4.2 in \cite{CFL19},  Eq.(7.73) in \cite[Section~7.2.7]{LiIto06}, \cite[Section~3.3]{HanMW2021} and \cite[Section~2]{Angelova2007}, we find that
 \eqref{u:approx:ir:key}, \eqref{stencil:irregular:s1s2} and \eqref{Cfij:irregular}
 can achieve accuracy order $M+1$ for the numerical approximated solution. We can observe that  for the same integer $M$, the accuracy order at irregular points is one order higher than at the regular points. More details about this phenomenon can be found in  \cite{LiIto06,CFL19,HanMW2021,Angelova2007}.
}

As in \cite{FengHanMinev21}, consider one of the following two simple parametric representations of $\Gamma$:
\be\label{parametric:2 simple}
x=t+x_i^*, \quad y=r(t)+y_j^* \qquad \mbox{or}\quad
x=r(t)+x_i^*,\quad y=t+y_j^*, \quad \mbox{for}\;\; t\in (-\epsilon,\epsilon) \quad \mbox{with}\quad \epsilon>0,
\ee
for the base point $(x^*_i,y^*_j)$ and a smooth function $r(t)$, where $r(0)=0$. Note that from the Implicit Function Theorem one can derive  $\frac{d^n(r(0))}{dt^n}$  without knowing the explicit formula for  $r(t)$. To cover the above two cases of parametric equations in \eqref{parametric:2 simple} for $\Gamma$ together, we discuss the following general parametric equation for $\Gamma$:
\be \label{parametric}
x=r(t)+x_i^*,\quad y=s(t)+y_j^*,\quad
(r'(t))^2+(s'(t))^2>0
\quad \mbox{for}\;\; t\in (-\epsilon,\epsilon) \quad \mbox{with}\quad \epsilon>0.
\ee
Note that the parameter $t$ for the base point $(x_i^*,y_j^*)$ is $t=0$, and $r(0)=s(0)=0$.

According to the two jump conditions for the solution and flux  in \eqref{Qeques1},
we can link the two sets
$\{u_-^{(m,n)}:(m,n)\in \ind_{M+1}^1\}$ and $\{u_+^{(m,n)}:(m,n)\in \ind_{M+1}^1\}$ by the following theorem, whose proof is given in \cref{sec:proof}.

\begin{theorem}\label{thm:interface}
Let $u$ be the exact solution to the elliptic interface problem in \eqref{Qeques1}, and assume that the base point $(x_i^*,y_j^*)\in \Gamma$,  $\Gamma$ being parameterized near $(x_i^*,y_j^*)$ by \eqref{parametric}.
Then
\be \label{tranmiss:cond}
\begin{split}
u_-^{(m',n')}&=\sum_{(m,n)\in \ind_{M+1}^1}T^{u_{+}}_{m',n',m,n}u_+^{(m,n)}+\sum_{(m,n)\in \ind_{M-1}} \left(T^+_{m',n',m,n} f_+^{(m,n)}
+ T^-_{m',n',m,n} f_{-}^{(m,n)}\right)\\
&+\sum_{(m,n)\in \ind_{M+1}} T^{g_1}_{m',n',m,n} g_{1}^{(m,n)}
+\sum_{(m,n)\in \ind_{M}} T^{g_2}_{m',n',m,n} g_2^{(m,n)},\qquad \forall\; (m',n')\in \ind_{M+1}^1,
\end{split}
\ee
where all the transmission coefficients $T^{u_+}, T^\pm, T^{g_1}, T^{g_2}$ are uniquely determined by
$r^{(k)}(0)$, $s^{(k)}(0)$  for $k=0,\ldots,M+1$ and $a_{\pm}^{(m,n)}$ for $(m,n)\in \ind_{M}$.
\end{theorem}

Now we discuss how to find a compact scheme at an irregular point $(x_i,y_j)$ with the supposed accuracy order for the numerical approximated solution. As the set $\{-1,0,1\}^2$ is the disjoint union of $d_{i,j}^+$ and $d_{i,j}^-$, we have
\[
\begin{split}
	&\sum_{k=-1}^1 \sum_{\ell=-1}^1
	C_{k,\ell}(h) u(x_i+kh,y_j+\ell h)\\
	&=\sum_{(k,\ell)\in d_{i,j}^+}
	C_{k,\ell}(h) u(x_i^*+(v_0+k)h,y_j^*+(w_0+\ell) h)
	+\sum_{(k,\ell)\in d_{i,j}^-}
	C_{k,\ell}(h) u(x_i^*+(v_0+k)h,y_j^*+(w_0+\ell) h).
\end{split}
\]
By \eqref{u:approx:ir:key},
\[
\begin{split}
	\sum_{(k,\ell)\in d_{i,j}^{\pm}}
	C_{k,\ell}(h) u(x_i^*+v_0h+kh,y_j^*+w_0h+\ell h)&=\sum_{(m,n)\in \ind_{M+1}^1}
	u_\pm^{(m,n)} I^\pm_{m,n}(h)
	\\
	&+\sum_{(m,n)\in \ind_{M-1}}
	f_\pm ^{(m,n)} J^{\pm,0}_{m,n}(h)+\bo(h^{M+2}),
\end{split}
\]
where
\be \label{IJpm}
\begin{split}
&I^\pm_{m,n}(h):=\sum_{(k,\ell)\in d_{i,j}^\pm}
C_{k,\ell}(h) G^{\pm}_{m,n}(v_0h+kh,w_0h+\ell h),\\
&J^{\pm,0}_{m,n}(h):=\sum_{(k,\ell)\in d_{i,j}^\pm} C_{k,\ell}(h) H^{\pm}_{m,n}(v_0h +kh,w_0h+\ell h).
\end{split}
\ee
Using \eqref{tranmiss:cond} in \cref{thm:interface}, we obtain
\begin{align*}
	\sum_{(m',n')\in \ind_{M+1}^1}
	u_-^{(m',n')} I^-_{m',n'}(h)
	=&\sum_{(m,n)\in \ind_{M+1}^1} u_+^{(m,n)}
	J^{u_+,T}_{m,n}(h)+
	\sum_{(m,n)\in \ind_{M-1}}
	\left(f_+^{(m,n)} J^{+,T}_{m,n}(h)+
	f_-^{(m,n)} J^{-,T}_{m,n}(h)\right)\\
	&+\sum_{(m,n)\in \ind_{M+1}} g_1^{(m,n)}J^{g_1}_{m,n}(h)+
	\sum_{(m,n)\in \ind_M} g_2^{(m,n)}J^{g_2}_{m,n}(h),
\end{align*}
where
\be
\label{JTno}
\begin{split}
	&J^{u_+,T}_{m,n}(h):=
	\sum_{(m',n')\in \ind_{M+1}^1} I^-_{m',n'}(h) T^{u_+}_{m',n',m,n}, \quad J^{\pm,T}_{m,n}(h):=
	\sum_{(m',n')\in \ind_{M+1}^1} I^-_{m',n'}(h) T^\pm_{m',n',m,n},\\
	&J^{g_1}_{m,n}(h):=
	\sum_{(m',n')\in \ind_{M+1}^1} I^-_{m',n'}(h) T^{g_1}_{m',n',m,n},\quad
	J^{g_2}_{m,n}(h):=
	\sum_{(m',n')\in \ind_{M+1}^1} I^-_{m',n'}(h) T^{g_2}_{m',n',m,n}.
\end{split}
\ee
Consequently,
\be \label{u:sum:s1s2}
\begin{split}
	&	\sum_{k=-1}^1 \sum_{\ell=-1}^1
	C_{k,\ell}(h) u(x_i+kh,y_j+\ell h)=\sum_{(m,n)\in \ind_{M+1}^1}
	u_+^{(m,n)} I_{m,n}(h)+
	\sum_{(m,n)\in \ind_{M-1} } f_+^{(m,n)}J^+_{m,n}(h)\\
	&+\sum_{(m,n)\in \ind_{M-1}} f_-^{(m,n)}J^-_{m,n}(h)+\sum_{(m,n)\in \ind_{M+1}} g_1^{(m,n)}J^{g_1}_{m,n}(h)
	+\sum_{(m,n)\in \ind_{M}} g_2^{(m,n)}J^{g_2}_{m,n}(h),
\end{split}
\ee
where
%
\be\label{IJ:ir1:s1s2}
	 \hspace{3cm} I_{m,n}(h):=I^+_{m,n}(h)+J^{u_+,T}_{m,n}(h),\quad J^{\pm}_{m,n}(h):=
	J_{m,n}^{\pm,0}(h)+J^{\pm,T}_{m,n}(h).
\ee
%
Since $\{f_\pm^{(m,n)}:(m,n)\in \ind_{M-1}\}$ , $\{g_1^{(m,n)}:(m,n)\in \ind_{M+1}\}$ and $\{g_2^{(m,n)}:(m,n)\in \ind_M\}$  are available and all the unknowns in \eqref{u:sum:s1s2} only belong to the set $\{u_+^{(m,n)}:(m,n)\in \ind_{M+1}^1\}$, \eqref{stencil:irregular:s1s2} can be equivalently written as
\begin{align}
	&I_{m,n}(h)=\bo(h^{M+2}),  &&h\to 0, \; \mbox{ for all }\; (m,n)\in \ind_{M+1}^1, \label{stencil:regular:u:ir}\\
	 &J^\pm_{m,n}(h)=C_{f_\pm,m,n}(h)+\bo(h^{M+2}),
	&&h\to 0,\; \mbox{ for all }\; (m,n)\in \ind_{M-1},\label{stencil:regular:f:ir}\\
	 &J^{g_1}_{m,n}(h)=C_{g_1,m,n}(h)+\bo(h^{M+2}),
	&&h\to 0,\; \mbox{ for all }\; (m,n)\in \ind_{M+1},\label{stencil:regular:g1:ir}\\
	 &J^{g_2}_{m,n}(h)=C_{g_2,m,n}(h)+\bo(h^{M+2}),
	&&h\to 0,\; \mbox{ for all }\; (m,n)\in \ind_{M}.\label{stencil:regular:g2:ir}
\end{align}
Then similar to \eqref{Cfmn:M=4}, substituting $\{C_{k,\ell}(h)\}_{k,\ell=-1,0,1}$ of \eqref{stencil:regular:u:ir} into \eqref{stencil:regular:f:ir}-\eqref{stencil:regular:g2:ir}, all other coefficients of the compact scheme can be calculated by
\begin{align}
	&C_{f_\pm,m,n}(h):=J_{m,n}^\pm(h), \quad (m,n)\in \ind_{M-1},\label{coeff:f:ir:gradient}\\
	&C_{g_1,m,n}(h):=J^{g_1}_{m,n}(h),\quad (m,n)\in \ind_{M+1}\quad
	\mbox{and}\quad
	C_{g_2,m,n}(h)=J^{g_2}_{m,n}(h),\quad (m,n)\in \ind_M.\label{coeff:g:ir:gradient}
\end{align}
We can check that the maximum $M$, such that a nontrivial
solution $\{C_{k,\ell}(h)\}_{k,\ell=-1,0,1}$ exists for \eqref{stencil:regular:u:ir}, is $M=2$.
Thus, we obtain the following theorem for a compact scheme at irregular points.

\begin{theorem}\label{thm:irregular}
Let   $(x_i, y_j)$ be an irregular point and $(u_{h})_{i,j}$ be the numerical approximation of the exact solution $u$ of the partial differential equation \eqref{Qeques1} at $(x_i, y_j)$. Pick a base point
$(x_i^*,y_j^*)$ as in \eqref{base:pt}.
Then
the following compact scheme centered at the irregular point $(x_i,y_j):$
\begin{equation}\label{stencil:irregular1}
\begin{split}
    &\sum_{k=-1}^1 \sum_{\ell=-1}^1 C_{k,\ell}(h)(u_{h})_{i+k,j+\ell}=\sum_{(m,n)\in \ind_{1} } f_+^{(m,n)}J^+_{m,n}(h)+\sum_{(m,n)\in \ind_{1}} f_-^{(m,n)}J^-_{m,n}(h)\\
	 &\hspace{4cm}+\sum_{(m,n)\in \ind_{3}} g_1^{(m,n)}J^{g_1}_{m,n}(h)
	 +\sum_{(m,n)\in \ind_{2}} g_2^{(m,n)}J^{g_2}_{m,n}(h),
\end{split}
\end{equation}
has a third order consistency error at the irregular point $(x_i,y_j)$, i.e., the accuracy order for $u_h$ is three,
where the quantities $\{C_{k,\ell}(h)\}_{k,\ell=-1,0,1}$ are the nontrivial solutions of  \eqref{stencil:regular:u:ir} with $M=2$, $J^\pm_{m,n}, (m,n)\in \ind_1$, $J^{g_1}_{m,n}, (m,n)\in \ind_3$ and $J^{g_2}_{m,n}, (m,n)\in \ind_2$ are given in  \eqref{IJ:ir1:s1s2} and \eqref{JTno}.
\end{theorem}

\begin{theorem}\label{thm:Max:Order}
The maximum accuracy order for the  numerical approximation $u_h$ at an irregular point of a compact finite difference scheme is three, i.e., the largest $M$ such that the nontrivial
solution $\{C_{k,\ell}(h)\}_{k,\ell=-1,0,1}$ exists for \eqref{stencil:regular:u:ir} is $M=2$.
\end{theorem}

\begin{proof}
Let us consider the following simple case:  $\Gamma=\{(x,y)\in \Omega\; :\; \psi(x,y)=0\}$ with
$\psi(x,y)=2x-y$, $x_i=y_j=0$, $x_{i-1}=y_{j-1}=-h$,
 $x_{i+1}=y_{j+1}=h$,
 $x_i^*=x_i=0$, $y_j^*=y_j=0$ and $\nv=\frac{(2,-1)}{\sqrt{5}}$ (see \cref{fig:simple:case} for an illustration). From \eqref{stencil:regular:u:ir},  the source term $f_{\pm}$ and the two jump functions $g_1$ and $g_2$ do not affect the existence of the nontrivial solution $\{C_{k,\ell}(h)\}_{k,\ell=-1,0,1}$ of \eqref{stencil:regular:u:ir}. To further simplify the calculation, we can assume that
$f_{\pm}=g_1=g_2=0.$ Then it is easy to check that  all $\{C_{k,\ell}(h)\}_{k,\ell=-1,0,1}$ of  \eqref{stencil:regular:u:ir} are zeros for $M=3$ and $h=0$. So \eqref{stencil:regular:u:ir} only has a trivial solution for $M=3$.
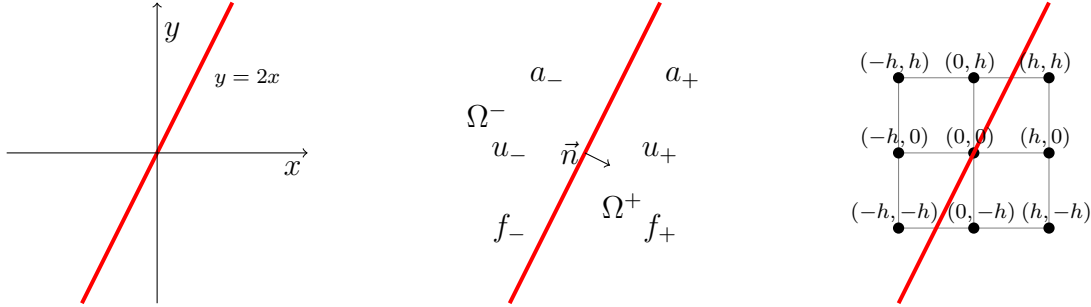
\begin{figure}[h]
	\centering
		\hspace{14mm}
\begin{subfigure}[b]{0.3\textwidth}
		\begin{tikzpicture}[scale = 1]
			\draw[line width=1.5pt, red]  plot [smooth,tension=0.8]
			coordinates {(1,0) (2,2) (3,4)};
			\draw[->] (0,2)--(4,2);
			\draw[->] (2,0)--(2,4);
		  \node (A) at (3.8,1.8) {$x$};
			\node (A) at (2.2,3.6) {$y$};
				\node (A) at (3.2,3) {\tiny{$y=2x$}};
		\end{tikzpicture}
	\end{subfigure}
   \begin{subfigure}[b]{0.3\textwidth}
   	\hspace{3mm}
		\begin{tikzpicture}[scale = 1]
    \draw[line width=1.5pt, red]  plot [smooth,tension=0.8]
coordinates {(1,0) (2,2) (3,4)};
	\node (A) at (1,2) {$u_{-}$};
    \node (A) at (3,2) {$u_{+}$};
  	\node (A) at (1.5,3) {$a_{-}$};
   \node (A) at (3.3,3) {$a_{+}$};
   \node (A) at (1,1) {$f_{-}$};
   \node (A) at (3,1) {$f_{+}$};
   	\node (A) at (0.7,2.5) {$\Omega^{-}$};
   \node (A) at (2.5,1.3) {$\Omega^{+}$};
   	\node (A) at (1.8,2) {$\nv$};
   \draw[->] (2,2)--(2+1/3,2-1/6);
		\end{tikzpicture}
	\end{subfigure}
	\begin{subfigure}[b]{0.3\textwidth}
		\begin{tikzpicture}[scale = 1]
			\draw[help lines,step = 1]
			(1,1) grid (3,3);
			\node at (1,1)[circle,fill,inner sep=1.5pt,color=black]{};
			\node at (1,2)[circle,fill,inner sep=1.5pt,color=black]{};
			\node at (1,3)[circle,fill,inner sep=1.5pt,color=black]{};	
			\node at (2,1)[circle,fill,inner sep=1.5pt,color=black]{};
			\node at (2,2)[circle,fill,inner sep=1.5pt,color=black]{};
			\node at (2,3)[circle,fill,inner sep=1.5pt,color=black]{};	
			\node at (3,1)[circle,fill,inner sep=1.5pt,color=black]{};
			\node at (3,2)[circle,fill,inner sep=1.5pt,color=black]{};
			\node at (3,3)[circle,fill,inner sep=1.5pt,color=black]{};	
			\draw[line width=1.5pt, red]  plot [smooth,tension=0.8]
			coordinates {(1,0) (2,2) (3,4)};
	\node (A) at (0.95,1.2) {\tiny{$(-h,-h)$}};
	\node (A) at (0.95,2.2) {\tiny{$(-h,0)$}};
    \node (A) at (0.95,3.2) {\tiny{$(-h,h)$}};
    \node (A) at (2.1,1.2) {\tiny{$(0,-h)$}};
    \node (A) at (1.95,2.2) {\tiny{$(0,0)$}};
    \node (A) at (1.95,3.2) {\tiny{$(0,h)$}};
    \node (A) at (3.1,1.2) {\tiny{$(h,-h)$}};
    \node (A) at (2.95,2.2) {\tiny{$(h,0)$}};
    \node (A) at (2.95,3.2) {\tiny{$(h,h)$}};
		\end{tikzpicture}
	\end{subfigure}
	\caption
	{\footnotesize{One simple example for irregular points. The curve in red color is the interface curve $\Gamma=\{(x,y)\in \Omega\; :\; 2x-y=0\}$, the left of $\Gamma$ is $\Omega^{-}$ and the right of $\Gamma$ is $\Op$.}}
	\label{fig:simple:case}
\end{figure}
\end{proof}

\section{A High order compact approximation for computing $\nabla u$
\label{sec:Gradient}}

In \cref{sec:Solution}, we derived a high order compact finite difference scheme for the elliptic interface problem. After obtaining the numerical solution defined  by \cref{thm:regular} and \cref{thm:irregular}, we could locally compute the  gradient approximation  without constructing and solving a global linear system. For the convenience of the readers, in this section,  we also derive the high order compact approximation for the gradient by using the already computed numerical solution in \cref{sec:Solution}.

\subsection{Regular points\label{sec:Gradient:Regular}}
In this section, we will discuss the derivation of a compact approximation of the gradient at regular points.  The scheme is local and does not require the solution of a global linear system.
As in \cref{sec:Solution:Regular}, we choose $(x_i^*,y_j^*)=(x_i,y_j)$, i.e.,
$v_0=w_0=0$ in \eqref{base:pt}, and consider the following equation:
\be \label{stencil:theta:regular1}
\begin{split}
	h(u^{(1,0)},u^{(0,1)})\cdot (\cos(\theta),\sin(\theta))&=\sum_{k=-1}^1 \sum_{\ell=-1}^1
	C_{k,\ell}(h) u(x_i+kh,y_j+\ell h)\\
	&-
	\sum_{(m,n)\in \ind_{M-1}} f^{(m,n)}C_{f,m,n}(h)+\bo(h^{M+2}),\qquad h\to 0,
\end{split}
\ee
where $C_{k,\ell}(h)$ and $C_{f,m,n}(h)$ are to-be-determined polynomials of $h$ with degree less than $M+2$. Note that $\theta \in [0,2\pi]$ and $u(x,y)$ is defined in \eqref{u:approx:key}.

For $(\cos(\theta),\sin(\theta))=(1,0)$ with $\theta=0$ or $(0,1)$ with $\theta=\pi/2$, if the coefficients $\{C_{k,\ell}(h)\}_{k,\ell=-1,0,1}$ and $\{C_{f,m,n}(h)\}_{(m,n)\in \ind_{M-1}}$ satisfy \eqref{stencil:theta:regular1}, we readily obtain:
\[
\begin{split}
	u^{(1,0)} \quad  \mbox{or} \quad u^{(0,1)}&=\sum_{k=-1}^1 \sum_{\ell=-1}^1
	\frac{1}{h}C_{k,\ell}(h) u(x_i+kh,y_j+\ell h)\\
	&-
	\sum_{(m,n)\in \ind_{M-1}} \frac{1}{h}f^{(m,n)}C_{f,m,n}(h)+\bo(h^{M+1}),\qquad h\to 0.
\end{split}
\]
In other words, the gradient can be computed locally with an accuracy of $M+1$. Moreover, \eqref{stencil:theta:regular1} yields the same order for the approximated gradient in any direction corresponding to $\theta \in [0,2\pi]$.

As in \cref{sec:Solution:Regular}, it is straightforward to show that \eqref{stencil:theta:regular1} is equivalent to :
\be \label{stencil:theta:regular:2}
\sum_{(m,n)\in \ind_{M+1}^1}
u^{(m,n)} I_{m,n}(h)+
\sum_{(m,n)\in \ind_{M-1}} f^{(m,n)}
\left(J_{m,n}(h)-C_{f,m,n}(h)\right)
=\bo(h^{M+2}),
\ee
where
\be \label{I:theta:regular1}
\begin{split}
&I_{m,n}(h):=\sum_{k=-1}^1 \sum_{\ell=-1}^1 C_{k,\ell}(h) G_{m,n}(kh, \ell h), \quad \mbox{for} \quad m+n\ne 1,\\
&I_{m,n}(h):=\sum_{k=-1}^1 \sum_{\ell=-1}^1 C_{k,\ell}(h) G_{m,n}(kh, \ell h)-h(m,n)\cdot (\cos(\theta),\sin(\theta)), \quad \mbox{for} \quad m+n= 1,
\end{split}
\ee
and $J_{m,n}(h)$ is defined in \eqref{J:s1s2:regular}. Furthermore, \eqref{stencil:theta:regular:2} is equivalent to
\be \label{stencil:theta:regular:u}
I_{m,n}(h)=\bo(h^{M+2}), \quad h\to 0,\; \mbox{ for all }\; (m,n)\in \ind_{M+1}^1,
\ee
and
\be \label{stencil:theta:regular:f}
C_{f,m,n}(h)=J_{m,n}(h) +\bo(h^{M+2}),\qquad h\to 0, \; \mbox{ for all }\; (m,n)\in \ind_{M-1}.
\ee
We can check that the largest integer $M$ for the linear system in \eqref{stencil:theta:regular:u} with $(\cos(\theta),\sin(\theta))=(1,0)$ and $(0,1)$ to have a nontrivial solution $\{C_{k,\ell}(h)\}_{k,\ell=-1,0,1}$ is $M=3$.
One nontrivial solution $\{C_{k,\ell}(h)\}_{k,\ell=-1,0,1}$ to \eqref{stencil:theta:regular:u} with $M=3$ and $(\cos(\theta),\sin(\theta))=(1,0)$ is given by
{\tiny{
		\be\label{Ch1:gradient:M=3}
		\begin{split}
			C_{-1,-1}(h)
			& = ((-11(a^{(1, 2)})-11(a^{(2, 1)})-11(a^{(0, 3)})-11(a^{(3, 0)}))(a^{(0, 0)})^2+((22(a^{(0, 2)})+35(a^{(2, 0)})+11(a^{(1, 1)}))(a^{(0, 1)})\\
			&+(11((a^{(0, 2)})+(46/11)(a^{(2, 0)})+(a^{(1, 1)})))(a^{(1, 0)}))(a^{(0, 0)})-11(a^{(0, 1)})^3+(a^{(0, 1)})^2(a^{(1, 0)})-47(a^{(0, 1)})(a^{(1, 0)})^2\\
			&-59(a^{(1, 0)})^3)h^3+(22(((a^{(0, 2)})-(12/11)(a^{(2, 0)})+(23/11)(a^{(1, 1)}))(a^{(0, 0)})+(1/22)(a^{(0, 1)})^2-(71/22)(a^{(0, 1)})(a^{(1, 0)})\\
			&+(24/11)(a^{(1, 0)})^2))(a^{(0, 0)})h^2+2h((a^{(0, 1)})-11(a^{(1, 0)}))(a^{(0, 0)})^2+20(a^{(0, 0)})^3,
		\end{split}
		\ee
}}
{\tiny{
		\be\label{Ch2:gradient:M=3}
		\begin{split}
			C_{-1,0}(h)
			&= ((11(a^{(0, 3)})+11(a^{(2, 1)}))(a^{(0, 0)})^2+((-22(a^{(0, 2)})-35(a^{(2, 0)}))(a^{(0, 1)})-11(a^{(1, 0)})(a^{(1, 1)}))(a^{(0, 0)})+11(a^{(0, 1)})^3\\
			&-12(a^{(0, 1)})^2(a^{(1, 0)})+59(a^{(0, 1)})(a^{(1, 0)})^2)h^3-44(a^{(0, 0)})(((a^{(0, 2)})-(a^{(2, 0)})+(12/11)(a^{(1, 1)}))(a^{(0, 0)})+(1/22)(a^{(0, 1)})^2\\
			&-(24/11)(a^{(0, 1)})(a^{(1, 0)})+(1/2)(a^{(1, 0)})^2)h^2-88h(a^{(0, 0)})^2(a^{(1, 0)})+80(a^{(0, 0)})^3,	
		\end{split}
		\ee
}}
{\tiny{
		\be\label{Ch3:gradient:M=3}
		\begin{split}
			C_{-1,1}(h)
			& =22(a^{(0, 0)})((((a^{(0, 2)})-(12/11)(a^{(2, 0)})+(1/11)(a^{(1, 1)}))(a^{(0, 0)})+(1/22((a^{(0, 1)})-(a^{(1, 0)})))((a^{(0, 1)})-48(a^{(1, 0)})))h^2\\
			&+h((a^{(0, 1)})-(a^{(1, 0)}))(a^{(0, 0)})+(10/11)(a^{(0, 0)})^2),
		\end{split}
		\ee
}}
{\tiny{
		\be\label{Ch4:gradient:M=3}
		\begin{split}
			C_{0,-1}(h)
			& =((11(a^{(3, 0)})+11(a^{(1, 2)}))(a^{(0, 0)})^2+((-11(a^{(0, 2)})-46(a^{(2, 0)}))(a^{(1, 0)})-11(a^{(0, 1)})(a^{(1, 1)}))(a^{(0, 0)})+11(a^{(0, 1)})^2\\
			&\times (a^{(1, 0)})-12(a^{(0, 1)})(a^{(1, 0)})^2+59(a^{(1, 0)})^3)h^3-48(a^{(0, 0)})((a^{(0, 0)})(a^{(1, 1)})+(1/2)(a^{(0, 1)})((a^{(0, 1)})-4(a^{(1, 0)})))h^2\\
			&-44(a^{(0, 0)})^2((a^{(0, 1)})+(12/11)(a^{(1, 0)}))h+88(a^{(0, 0)})^3,
		\end{split}
		\ee
}}
{\tiny{
		\be\label{Ch5:gradient:M=3}
		\begin{split}
			C_{0,0}(h)
			& = 4(a^{(0, 0)})((((a^{(2, 0)})+12(a^{(1, 1)}))(a^{(0, 0)})+6(a^{(0, 1)})^2-18(a^{(0, 1)})(a^{(1, 0)})-(37/2)(a^{(1, 0)})^2)h^2-12(a^{(0, 0)})((a^{(0, 1)})\\
			&-(19/4)(a^{(1, 0)}))h-110(a^{(0, 0)})^2),
		\end{split}
		\ee
}}
{\tiny{
		\be\label{Ch6:gradient:M=3}
		\begin{split}
			C_{0,1}(h)& = 44(a^{(0, 0)})^2(2(a^{(0, 0)})+h((a^{(0, 1)})-(12/11)(a^{(1, 0)}))),
		\end{split}
		\ee
}}
\vspace{-3mm}
{\tiny{
		\be	\label{Ch789:gradient:M=3}
		\begin{split}
			&C_{1,-1}(h) = 24(a^{(0, 0)})^3, \quad C_{1,0}(h) =96(a^{(0, 0)})^3,\\
			&C_{1,1}(h) =24(a^{(0, 0)})^2(h(a^{(0, 1)})+(a^{(0, 0)})).
		\end{split}
		\ee
}}

Similarly to \eqref{Cfmn:M=4}, we have:
\be \label{Cfmn:gradient:M=3}
C_{f,m,n}(h):=
\sum_{k=-1}^1 \sum_{\ell=-1}^1 C_{k,\ell}(h)
H_{m,n}(h), \qquad (m,n)\in \ind_2.
\ee
These observations prove the following theorem.

\begin{theorem}\label{thm:gradient:regular}
		Let   $(x_i, y_j)$ be a regular point and $\big((u_{h})_x\big)_{i,j},\big((u_{h})_y\big)_{i,j}$ be the numerical approximation of the exact gradient $u_x$ and $u_y$ at $(x_i, y_j)$.	 Then
	the following compact  approximation to the gradient of the solution of problem (\ref{Qeques1}) at $(x_i,y_j):$
	 \begin{equation}\label{stencil:gradient:regular}
		 \big((u_{h})_x\big)_{i,j}=\frac{1}{h}\sum_{k=-1}^1 \sum_{\ell=-1}^1 C_{k,\ell}(h)(u_{h})_{i+k,j+\ell}
		-\frac{1}{h}C_{f,m,n}(h),
	\end{equation}
	achieves fourth  order of accuracy for the approximation $(u_{h})_x$ at the regular point $(x_i,y_j)$,
where $(u_{h})_{i,j}$ is the numerical solution at $(x_i,y_j)$ from \cref{sec:Solution},  $\{C_{k,\ell}(h)\}_{k,\ell=-1,0,1}$ is defined in \eqref{Ch1:gradient:M=3} to \eqref{Ch789:gradient:M=3},  $C_{f,m,n}(h)$ is defined in \eqref{Cfmn:gradient:M=3}, $a^{(m,n)}:=\frac{\partial^{m+n} a}{ \partial^m x \partial^n y}(x_i,y_j)$ and $f^{(m,n)}:=\frac{\partial^{m+n} f}{ \partial^m x \partial^n y}(x_i,y_j)$.
	Furthermore, the compact finite difference scheme of fourth order of accuracy for $(u_{h})_y$ at the regular point $(x_i,y_j)$ can be obtained similarly.  The maximum order of accuracy $M+1$ for the gradient  approximation  at a regular point is four.
\end{theorem}

\subsection{Irregular points\label{sec:Gradient:Irregular}}

In this section, we will discuss the derivation of the compact scheme for the  local computation of  the  gradient approximation  at irregular points.
Similarly to \cref{sec:Solution:Irregular}, in case of an irregular  point $(x_i,y_j)$,  the base point is taken to be on the interface $\Gamma$ i.e.  $(x^*_i,y^*_j)\in \Gamma \cap (x_i-h,x_i+h)\times (y_j-h,y_j+h)$. We assume that
\eqref{xiyj}, \eqref{aufpm}, \eqref{g1g2pm} and \eqref{u:approx:ir:key} hold. To simplify the calculation, we also assume that $(x_i,y_j)\in \Op$.

Let us consider that following equation:
{\small{\be
		 \begin{split}\label{stencil:theta:irregular}
			h\nabla \big(u_+ (x_i^*+v_0h,y_j^*&+w_0h)\big)\cdot (\cos(\theta),\sin(\theta))=\sum_{k=-1}^1 \sum_{\ell=-1}^1C_{k,\ell}(h) u(x_i+kh,y_j+\ell h)\\
			&-\sum_{(m,n)\in \ind_{M-1}} C_{f_{+},m,n}(h) f_+^{(m,n)}-\sum_{(m,n)\in \ind_{M-1}} C_{f_{-},m,n}(h)f_-^{(m,n)}\\
			&-\sum_{(m,n)\in \ind_{M+1}} C_{g_1,m,n}(h) g_1^{(m,n)}-\sum_{(m,n)\in \ind_{M}} C_{g_2,m,n}(h) g_2^{(m,n)}+\bo(h^{M+2}),
		\end{split}
		\ee
}}
where  $h\to 0$, $\theta \in [0,2\pi]$, $C_{k,\ell}(h), C_{f_{\pm},m,n}(h)$,
$C_{g_1,m,n}(h)$ and $C_{g_2,m,n}(h)$
are to-be-determined polynomials of $h$ having degrees less than $M+2$.

Similarly to the discussion in \cref{sec:Gradient:Regular}, for $(\cos(\theta),\sin(\theta))=(1,0)$ and (0,1), it can be shown that \eqref{stencil:theta:irregular} has an accuracy of order $M+1$ for the gradient approximation.

According to \eqref{uxy:approx:ir:key},
\[
\begin{split}
h	\nabla \big(u_+ (x_i^*+v_0h,y_j^*+w_0h)\big)\cdot (\cos(\theta),\sin(\theta))
&=\sum_{(m,n)\in \ind_{M+1}^1}
u_+^{(m,n)}I^{+,\theta}_{m,n}(h)\\
&+\sum_{(m,n)\in \ind_{M-1}}
f_+ ^{(m,n)} J^{+,\theta}_{m,n}(h)+\bo(h^{M+2}),
\end{split}
\]
where
\[
\begin{split}
&I^{+,\theta}_{m,n}(h)=h\nabla\big(G^{+}_{m,n}(v_0h,w_0h)\big) \cdot (\cos(\theta),\sin(\theta)),\ \ J^{+,\theta}_{m,n}(h)=h\nabla\big( H^{+}_{m,n}(v_0h,w_0h)\big) \cdot (\cos(\theta),\sin(\theta)).
\end{split}
\]
Similarly to \cref{sec:Solution:Irregular}, we also have:
{\small{
\[
\begin{split}
	&\sum_{k=-1}^1 \sum_{\ell=-1}^1
	C_{k,\ell}(h) u(x_i+kh,y_j+\ell h)-h\nabla \big(u_+ (v_0h+x_i^*,w_0h+y_j^*)\big)\cdot (\cos(\theta),\sin(\theta))\\
	&=\sum_{(k,\ell)\in d_{i,j}^+}
	C_{k,\ell}(h) u(x_i^*+(v_0+k)h,y_j^*+(w_0+\ell) h)
	+\sum_{(k,\ell)\in d_{i,j}^-}
	C_{k,\ell}(h) u(x_i^*+(v_0+k)h,y_j^*+(w_0+\ell) h)\\
	&-h\nabla \big(u_+ (x_i^*+v_0h,y_j^*+w_0h)\big)\cdot (\cos(\theta),\sin(\theta)),
\end{split}
\]
}}
and
\[
\begin{split}
	&	\sum_{k=-1}^1 \sum_{\ell=-1}^1
	C_{k,\ell}(h) u(x_i+kh,y_j+\ell h)-h\nabla \big(u_+ (x_i^*+v_0h,y_j^*+w_0h)\big)\cdot (\cos(\theta),\sin(\theta))\\
	&	=\sum_{(m,n)\in \ind_{M+1}^1}
	u_+^{(m,n)} I_{m,n}(h)+
	\sum_{(m,n)\in \ind_{M-1} } f_+^{(m,n)}J^+_{m,n}(h)\\
	&+\sum_{(m,n)\in \ind_{M-1}} f_-^{(m,n)}J^-_{m,n}(h)
	+\sum_{(m,n)\in \ind_{M+1}} g_1^{(m,n)}J^{g_1}_{m,n}(h)
	+\sum_{(m,n)\in \ind_{M}} g_2^{(m,n)}J^{g_2}_{m,n}(h),
\end{split}
\]
where
%
\begin{align}
	& \hspace{3cm} I_{m,n}(h):=I^+_{m,n}(h)+J^{u_+,T}_{m,n}(h)-I^{+,\theta}_{m,n}(h),\label{IJ:ir1:theta}\\
	&J^{-}_{m,n}(h):=
	J_{m,n}^{-,0}(h)+J^{-,T}_{m,n}(h), \quad  J^{+}_{m,n}(h):=
	J_{m,n}^{+,0}(h)+J^{+,T}_{m,n}(h)-J^{+,\theta}_{m,n}(h),\label{IJ:ir2:theta}
\end{align}
and $I^+_{m,n}(h)$, $J_{m,n}^{\pm,0}(h)$, $J^{u_+,T}_{m,n}(h)$, $J^{\pm,T}_{m,n}(h)$, $J^{g_1}_{m,n}(h)$ and $J^{g_2}_{m,n}(h)$ are defined in \eqref{IJpm} and \eqref{JTno}.
%
Due to the same arguments as the ones provided in \cref{sec:Solution:Irregular}, \eqref{stencil:theta:irregular} is equivalent to:
\begin{align}
	&I_{m,n}(h)=\bo(h^{M+2}),  \qquad h\to 0, \; \mbox{ for all }\; (m,n)\in \ind_{M+1}^1, \label{stencil:regular:u:ir:theta}\\
	&C_{f_\pm,m,n}(h):=J_{m,n}^\pm(h), \quad (m,n)\in \ind_{M-1},\label{coeff:f:ir:gradient:theta}\\
    &C_{g_1,m,n}(h):=J^{g_1}_{m,n}(h),\quad (m,n)\in \ind_{M+1}\quad
\mbox{and}\quad
C_{g_2,m,n}(h)=J^{g_2}_{m,n}(h),\quad (m,n)\in \ind_M.\label{coeff:g:ir:gradient:theta}
\end{align}
The following theorem summarizes the results above that guarantee the third order of accuracy of  the  gradient approximation at irregular points.

\begin{theorem}\label{thm:gradient:irregular}
		Let   $(x_i, y_j)$ be an irregular point and $\big((u_{h})_x\big)_{i,j},\big((u_{h})_y\big)_{i,j}$ be the numerical approximation of the exact gradient $u_x$ and $u_y$ at $(x_i, y_j)$.	 Then
the following compact  approximation to the gradient of the solution of problem (\ref{Qeques1}) at $(x_i,y_j):$
\begin{equation}\label{stencil:ux:irregular}
	\begin{split}
		\big((u_{h})_x\big)_{i,j}&=\sum_{k=-1}^1 \sum_{\ell=-1}^1 \frac{1}{h}C_{k,\ell}(h)(u_{h})_{i+k,j+\ell}-\sum_{(m,n)\in \ind_{1} } \frac{1}{h}f_+^{(m,n)}J^+_{m,n}(h)-\sum_{(m,n)\in \ind_{1}} \frac{1}{h}f_-^{(m,n)}J^-_{m,n}(h)\\
		&-\sum_{(m,n)\in \ind_{3}} \frac{1}{h}g_1^{(m,n)}J^{g_1}_{m,n}(h)-\sum_{(m,n)\in \ind_{2}} \frac{1}{h}g_2^{(m,n)}J^{g_2}_{m,n}(h) \ \mbox{ with }\   (\cos(\theta),\sin(\theta))=(1,0),
	\end{split}
\end{equation}
\begin{equation}\label{stencil:uy:irregular}
	\begin{split}
		\big((u_{h})_y\big)_{i,j}&=\sum_{k=-1}^1 \sum_{\ell=-1}^1 \frac{1}{h}C_{k,\ell}(h)(u_{h})_{i+k,j+\ell}-\sum_{(m,n)\in \ind_{1} } \frac{1}{h}f_+^{(m,n)}J^+_{m,n}(h)-\sum_{(m,n)\in \ind_{1}} \frac{1}{h}f_-^{(m,n)}J^-_{m,n}(h)\\
		&-\sum_{(m,n)\in \ind_{3}} \frac{1}{h}g_1^{(m,n)}J^{g_1}_{m,n}(h)-\sum_{(m,n)\in \ind_{2}} \frac{1}{h}g_2^{(m,n)}J^{g_2}_{m,n}(h) \ \mbox{ with } \ (\cos(\theta),\sin(\theta))=(0,1),
	\end{split}
\end{equation}
	achieves third  order of accuracy for the gradient approximation $(u_{h})_x$ and $(u_{h})_y$ at the irregular point $(x_i,y_j)$,
	where $(u_{h})_{i,j}$ is the numerical solution at $(x_i,y_j)$ from \cref{sec:Solution}, $\{C_{k,\ell}(h)\}_{k,\ell=-1,0,1}$ is the nontrivial solution of  \eqref{stencil:regular:u:ir:theta} with $M=2$, $(\cos(\theta),\sin(\theta))=(1,0)$ or $(0,1)$, $J^\pm_{m,n}, (m,n)\in \ind_1$, $J^{g_1}_{m,n}, (m,n)\in \ind_3$, and $J^{g_2}_{m,n}, (m,n)\in \ind_2$ are given in  \eqref{IJ:ir2:theta} and \eqref{JTno}.
\end{theorem}

\section{Numerical experiments}\label{sec:Numeri}

Let $\Omega=(l_1,l_2)\times(l_3,l_4)$ with
$l_4-l_3=N_0(l_2-l_1)$ for some positive integer $N_0$. For a given $J\in \NN$, we define
$h:=(l_2-l_1)/N_1$ with $N_1:=2^J$ and let
$x_i=l_1+ih$ and
$y_j=l_3+jh$ for $i=1,2,\dots,N_1-1$ and $j=1,2,\dots,N_2-1$ with $N_2:=N_0N_1$ and $N:=(N_1-1)(N_2-1)$.
%
Consider the following sets of grid points:
\begin{align*}
&\ind_{R}:=\{(i,j): \; 1\le i \le N_1-1,
\; 1\le j \le N_2-1 \mbox{ and } (x_i,y_j)\mbox{ is a regular point} \},\\
%
&\ind_{I}:=\{(i,j): \; 1\le i \le N_1-1,
\; 1\le j \le N_2-1 \mbox{ and } (x_i,y_j)\mbox{ is an irregular point} \},\\
%
&\ind_{\Omega}:=\{(i,j): \; 1\le i \le N_1-1 \mbox{ and }
\; 1\le j \le N_2-1  \}.
\end{align*}
Let
$u$ be the exact solution of \eqref{Qeques1} and $(u_{h})_{i,j}$ be its numerical approximation at $(x_i, y_j)$ on a grid with a  mesh size $h$.
Consider the following approximation of the $L^2$ norm of a given function $f$ :
\[
\|f\|^2_{2,\Omega,h}:=\int_{\Omega}|f(x,y)|^2 dxdy\approx h^2\sum_{(i,j)\in \ind_{\Omega}} |f(x_i,y_j)|^2 .
\]
If the exact solution is available, the accuracy of the scheme is verified by the relative error
$\frac{\|u_{h}-u\|_{2,\ind}}{\|u\|_{2,\ind,h}}$, where
\[
\begin{split}
&\|u_{h}-u\|_{2,\ind}^2:=h^2\sum_{(i,j)\in \ind} \left((u_h)_{i,j}-u(x_i,y_j)\right)^2, \qquad \|u\|_{2,\ind,h}^2:=h^2\sum_{(i,j)\in \ind} \left(u(x_i,y_j)\right)^2,\\
\end{split}
\]
and compute the order of convergence as follows:
\[
\mbox{order}=\log_2 \left(\tfrac{\|u_{h}-u\|_{2,\ind}/\|u\|_{2,\ind,h}}{\|u_{h/2}-u\|_{2,\ind}/\|u\|_{2,\ind,h/2}} \right),
\]
with $\ind$=$\ind_{R}$, $\ind_{I}$, or $\ind_{\Omega}$.
Otherwise, we quantify the error by
 $\|u_{h}-u_{h/2}\|_{2,\ind}$, where:
 \[
 \|u_{h}-u_{h/2}\|_{2,\ind}^2:= h^2\sum_{(i,j)\in \ind} \left((u_{h})_{i,j}-(u_{h/2})_{2i,2j}\right)^2, 
 \]
 and compute the order of convergence as follows:
 \[
\mbox{order}=\log_2 \left(\tfrac{\|u_{h}-u_{h/2}\|_{2,\ind}}{\|u_{h/2}-u_{h/4}\|_{2,\ind}} \right),
 \]
with $\ind$=$\ind_{\Omega}$.
Let
$(u_x(x,y),u_y(x,y))$ be the exact gradient of the solution of problem \eqref{Qeques1} and $\left(((u_{h})_x)_{i,j},((u_{h})_y)_{i,j}\right)$ be its numerical approximation at $(x_i, y_j)$ using the mesh size $h$.
If the exact solution $u$ is available, the convergence rate of the numerical approximation of the gradient is verified by the relative error
$\frac{|u_{h}-u|_{H^1,\ind}}{|u|_{H^1,\ind,h}}$, where
\begin{align*}
&|u_{h}-u|_{H^1,\ind}^2:= h^2\sum_{(i,j)\in \ind} \left(\big((u_h)_x\big)_{i,j}-u_x(x_i,y_j)\right)^2
+\left(\big((u_h)_y\big)_{i,j}-u_y(x_i,y_j)\right)^2,
\\
&|u|_{H^1,\ind,h}^2:=h^2 \sum_{(i,j)\in \ind} \left(u_x(x_i,y_j)\right)^2+ \left(u_y(x_i,y_j)\right)^2, \quad \mbox{order}=\log_2 \left(\tfrac{|u_{h}-u|_{H^1,\ind}/|u|_{H^1,\ind,h}}
{|u_{h/2}-u|_{H^1,\ind}/|u|_{H^1,\ind,h/2}} \right),
\end{align*}
with $\ind$=$\ind_{R}$, $\ind_{I}$ and $\ind_{\Omega}$.
If it is not, we  quantify the error by
 $|u_{h}-u_{h/2}|_{H^1,\ind}$, where
%
%
\[
\begin{split}
&|u_{h}-u_{h/2}|_{H^1,\ind}^2:= h^2\sum_{(i,j)\in \ind} \Big(\big((u_h)_x\big)_{i,j}-\big((u_{h/2})_x\big)_{2i,2j}
\Big)^2+\Big(\big((u_h)_y\big)_{i,j}-\big((u_{h/2})_y\big)_{2i,2j}\Big)^2,\\
&\mbox{order}=\log_2 \left(\tfrac{|u_{h}-u_{h/2}|_{H^1,\ind}}{|u_{h/2}-u_{h/4}|_{H^1,\ind}} \right),
\end{split}
\]
with $\ind$=$\ind_{\Omega}$.
 Since the flux $(au_x,au_y)$ represents the velocity of the fluid flow through a porous medium, we also provide the relative error   $\frac{|u_{h}-u|_{V,\ind}}{|u|_{V,\ind,h}}$ for the velocity
  if the exact solution $u$ is available, where
   \[
   |u_{h}-u|_{V,\ind}^2:= h^2\sum_{(i,j)\in \ind} a^2(x_i,y_j)\Big(\Big(\big((u_h)_x\big)_{i,j}
   -u_x(x_i,y_j)\Big)^2+\Big(\big((u_h)_y\big)_{i,j}
   -u_y(x_i,y_j)\Big)^2\Big),
   \]
   {\small{
   \[|u|_{V,\ind,h}^2:=h^2 \sum_{(i,j)\in \ind} a^2(x_i,y_j)\left(\left(u_x(x_i,y_j)\right)^2+ \left(u_y(x_i,y_j)\right)^2\right), \ \mbox{order}=\log_2 \left(\tfrac{|u_{h}-u|_{V,\ind}/|u|_{V,\ind,h}}
   {|u_{h/2}-u|_{V,\ind}/|u|_{V,\ind,h/2}} \right),\]
   }}
with $\ind$=$\ind_{R}$, $\ind_{I}$ and $\ind_{\Omega}$.
If it is not, we  quantify the error by
   $|u_{h}-u_{h/2}|_{V,\ind}$, where
   %
   {\small{
\begin{align*}
   	&|u_{h}-u_{h/2}|_{V,\ind}^2:= h^2\sum_{(i,j)\in \ind} a^2(x_i,y_j)\Big(\Big(\big((u_h)_x\big)_{i,j}-\big((u_{h/2})_x\big)_{2i,2j}\Big)^2
   +\Big(\big((u_h)_y\big)_{i,j}-\big((u_{h/2})_y\big)_{2i,2j}\Big)^2\Big),\\
   	&\mbox{order}=\log_2 \left(\tfrac{|u_{h}-u_{h/2}|_{V,\ind}}{|u_{h/2}-u_{h/4}|_{V,\ind}} \right),
\end{align*}
}}
    with $\ind$=$\ind_{\Omega}$.
   In addition, $\kappa$ denotes the condition number of the coefficient matrix.

\subsection{Numerical examples with $u$ known and $\Gamma \cap \partial \Omega=\emptyset$}
In this subsection, we provide numerical results of five test problems with an available exact solution $u$ of \eqref{Qeques1}. 

\begin{example}\label{Drafex2}
	\normalfont
	Let $\Omega=(-3,3)^2$ and
	the interface curve be given by
	$\Gamma:=\{(x,y)\in \Omega \; :\; \psi(x,y)=0\}$ with
	$\psi (x,y)=x^4+2y^4-2$. Note that $\Gamma \cap \partial \Omega=\emptyset$, the coefficient $a$ and	the exact solution $u$ of \eqref{Qeques1} are given by
	\begin{align*}
		 &a_{+}=a\chi_{\Op}=\frac{2+\cos(x)\cos(y)}{10},
		\qquad a_{-}=a\chi_{\Om}=10(2+\cos(x)\cos(y)),\\
		 &u_{+}=u\chi_{\Op}=10\sin(3.5x)(x^4+2y^4-2),
		\qquad u_{-}=u\chi_{\Om}=\frac{\sin(3.5x)(x^4+2y^4-2)}{10}+100.
	\end{align*}
	All the functions $f,g_1,g_2,g$ in \eqref{Qeques1} can be obtained by plugging the above coefficient and exact solution into \eqref{Qeques1}. In particular,
	$g_1=-100$ and $g_2=0$.
	The numerical results are presented in \cref{table:QSp2} and \cref{fig:figure2}.
\end{example}

\begin{table}[h!]
	\caption{\tiny{Performance in \cref{Drafex2}  of the proposed high order compact finite difference scheme in \cref{thm:regular,thm:gradient:regular,thm:irregular,thm:gradient:irregular} on uniform Cartesian meshes with $h=2^{-J}\times6$. $\kappa$ is the condition number of the coefficient matrix.}}
	\centering
	\setlength{\tabcolsep}{2mm}{
		\begin{tabular}{c|c|c|c|c|c|c|c}
			\hline
$J$
& $\frac{\|u_{h}-u\|_{2,\ind_{\Omega}}}{\|u\|_{2,\ind_{\Omega},h}}$

&order &$\frac{|u_{h}-u|_{H^1,\ind_{\Omega}}}{|u|_{H^1,\ind_{\Omega},h}}$

&order &  $\frac{|u_{h}-u|_{V,\ind_{\Omega}}}{|u|_{V,\ind_{\Omega},h}}$

&order &  $\kappa$ \\
\hline
			3    &2.4313E+00    &0    &2.8246E+00    &0    &3.3445E+02    &0    &1.6573E+04\\
			4    &1.0232E-01    &4.571    &5.8051E-02    &5.605    &2.0212E-01    &10.692    &4.9143E+06\\
			5    &5.1329E-03    &4.317    &3.9490E-03    &3.878    &1.2196E-02    &4.051    &1.2735E+05\\
			6    &2.3932E-04    &4.423    &2.7211E-04    &3.859    &1.3618E-03    &3.163    &8.4325E+05\\
			7    &1.9677E-05    &3.604    &2.6162E-05    &3.379    &1.4113E-04    &3.270    &4.8504E+06\\
			8    &1.0251E-06    &4.263    &1.8239E-06    &3.842    &1.4893E-05    &3.244    &9.0675E+06\\
			\hline
			

			$J$
			& $\frac{\|u_{h}-u\|_{2,\ind_{R}}}{\|u\|_{2,\ind_{R},h}}$
			
			&order &$\frac{|u_{h}-u|_{H^1,\ind_{R}}}{|u|_{H^1,\ind_{R},h}}$

			&order &  $\frac{|u_{h}-u|_{V,\ind_{R}}}{|u|_{V,\ind_{R},h}}$
			
			&order &  $\kappa$  \\
			\hline
			3    &8.5123E-01    &0    &1.3474E+00    &0    &3.4272E+01    &0    &1.6573E+04\\
			4    &7.7318E-02    &3.461    &3.2533E-02    &5.372    &1.4823E-01    &7.853    &4.9143E+06\\
			5    &4.4903E-03    &4.106    &1.9071E-03    &4.092    &9.8595E-03    &3.910    &1.2735E+05\\
			6    &2.2137E-04    &4.342    &1.2998E-04    &3.875    &1.1574E-03    &3.091    &8.4325E+05\\
			7    &1.8978E-05    &3.544    &9.6595E-06    &3.750    &1.3093E-04    &3.144    &4.8504E+06\\
			8    &1.0049E-06    &4.239    &5.9421E-07    &4.023    &1.4082E-05    &3.217    &9.0675E+06\\
			\hline
			$J$
			& $\frac{\|u_{h}-u\|_{2,\ind_{I}}}{\|u\|_{2,\ind_{I},h}}$
			
			&order &$\frac{|u_{h}-u|_{H^1,\ind_{I}}}{|u|_{H^1,\ind_{I},h}}$

			&order &  $\frac{|u_{h}-u|_{V,\ind_{I}}}{|u|_{V,\ind_{I},h}}$
			
			&order   &  $\kappa$ \\
			\hline
			3    &1.2501E+01    &0    &1.5013E+01    &0    &1.7014E+03    &0    &1.6573E+04\\
			4    &7.1859E-01    &4.121    &2.2993E+00    &2.707    &5.5292E+00    &8.265    &4.9143E+06\\
			5    &4.4345E-02    &4.018    &2.9401E-01    &2.967    &5.2016E-01    &3.410    &1.2735E+05\\
			6    &2.5440E-03    &4.124    &4.3955E-02    &2.742    &1.0321E-01    &2.333    &8.4325E+05\\
			7    &2.1169E-04    &3.587    &6.9868E-03    &2.653    &1.0925E-02    &3.240    &4.8504E+06\\
			8    &1.1939E-05    &4.148    &8.0020E-04    &3.126    &1.5374E-03    &2.829    &9.0675E+06\\
			\hline
			
	\end{tabular}}
	\label{table:QSp2}
\end{table}

\begin{remark} \normalfont
	(i)	For $u_h$, our proposed scheme  achieves third order at irregular points and fourth order at regular points respectively, while note that $u_{h}$ is solved globally. Thus, from \cref{table:QSp2}, we observe that the numerical orders for $\frac{\|u_{h}-u\|_{2,\ind_{\Omega}}}{\|u\|_{2,\ind_{\Omega},h}}$ ,$\frac{\|u_{h}-u\|_{2,\ind_{R}}}{\|u\|_{2,\ind_{R},h}}$  and $\frac{\|u_{h}-u\|_{2,\ind_{I}}}{\|u\|_{2,\ind_{I},h}}$ are all concentrated around $4$.\\
(ii) For $\nabla u_h$, our proposed scheme also achieves third order at irregular points and fourth order at regular points and $\nabla u_{h}$ is obtained locally.  Thus we observe that the numerical orders for $\frac{|u_{h}-u|_{H^1,\ind_{\Omega}}}{|u|_{H^1,\ind_{\Omega},h}}$ and $\frac{|u_{h}-u|_{H^1,\ind_{R}}}{|u|_{H^1,\ind_{R},h}}$ are both concentrated around $4$, while  the numerical orders for $\frac{|u_{h}-u|_{H^1,\ind_{I}}}{|u|_{H^1,\ind_{I},h}}$ are concentrated around $3$.
\end{remark}

\begin{figure}[htbp]
	\centering
	\begin{subfigure}[b]{0.3\textwidth}
		 \includegraphics[width=5.7cm,height=4.cm]{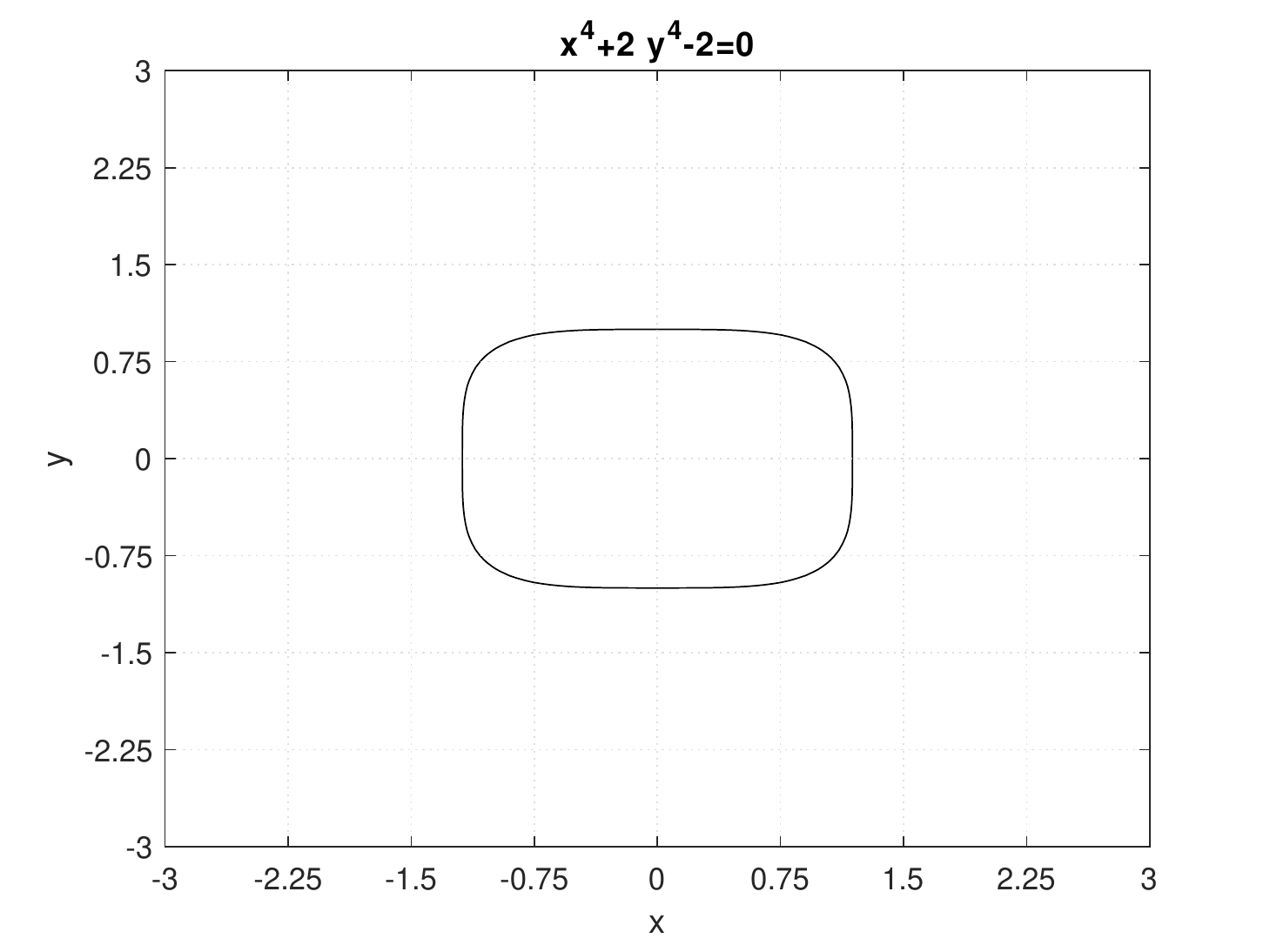}
	\end{subfigure}
	\begin{subfigure}[b]{0.3\textwidth}
		 \includegraphics[width=5.7cm,height=4.5cm]{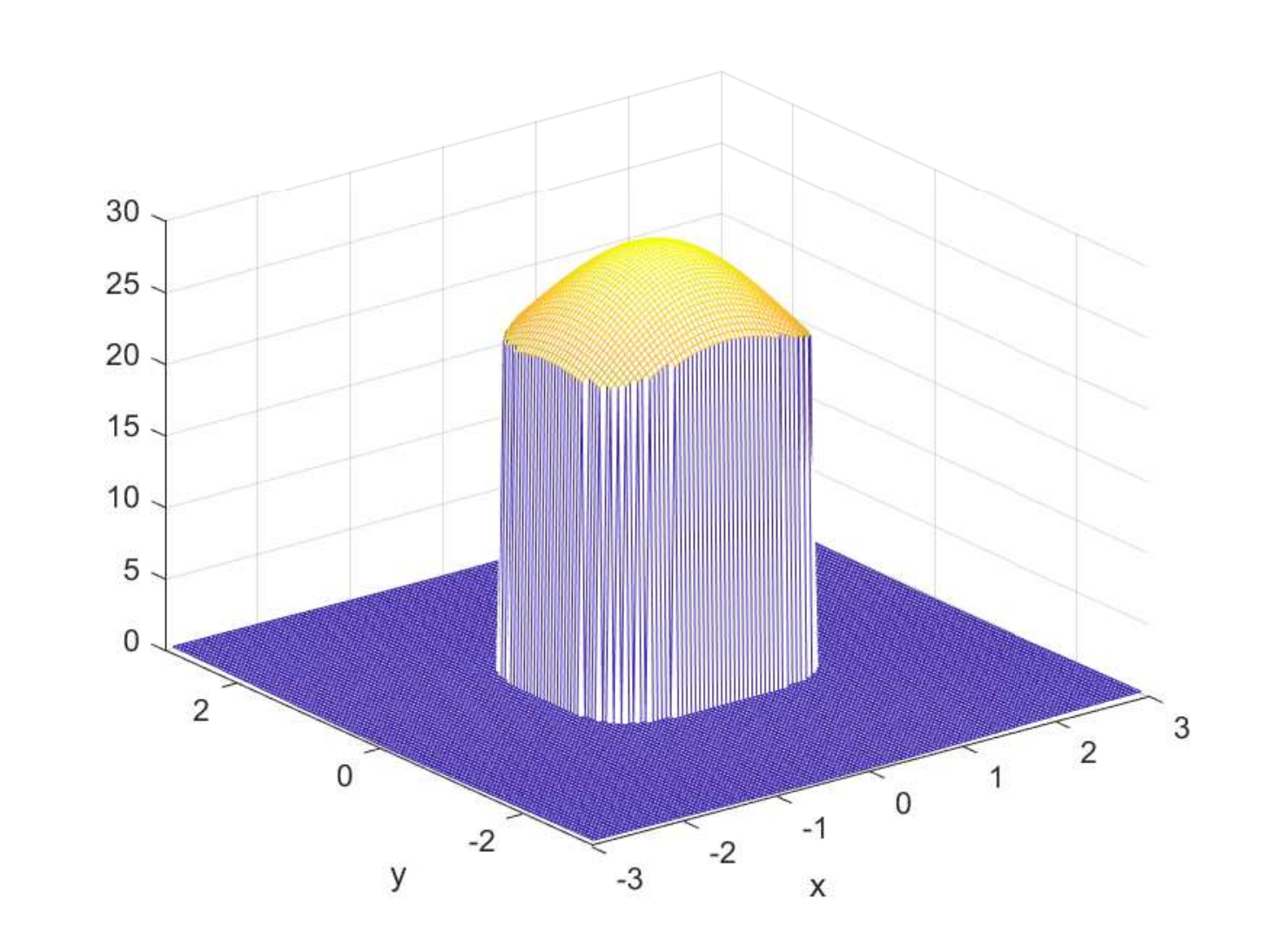}
	\end{subfigure}
	\begin{subfigure}[b]{0.3\textwidth}
		 \includegraphics[width=5.7cm,height=4.5cm]{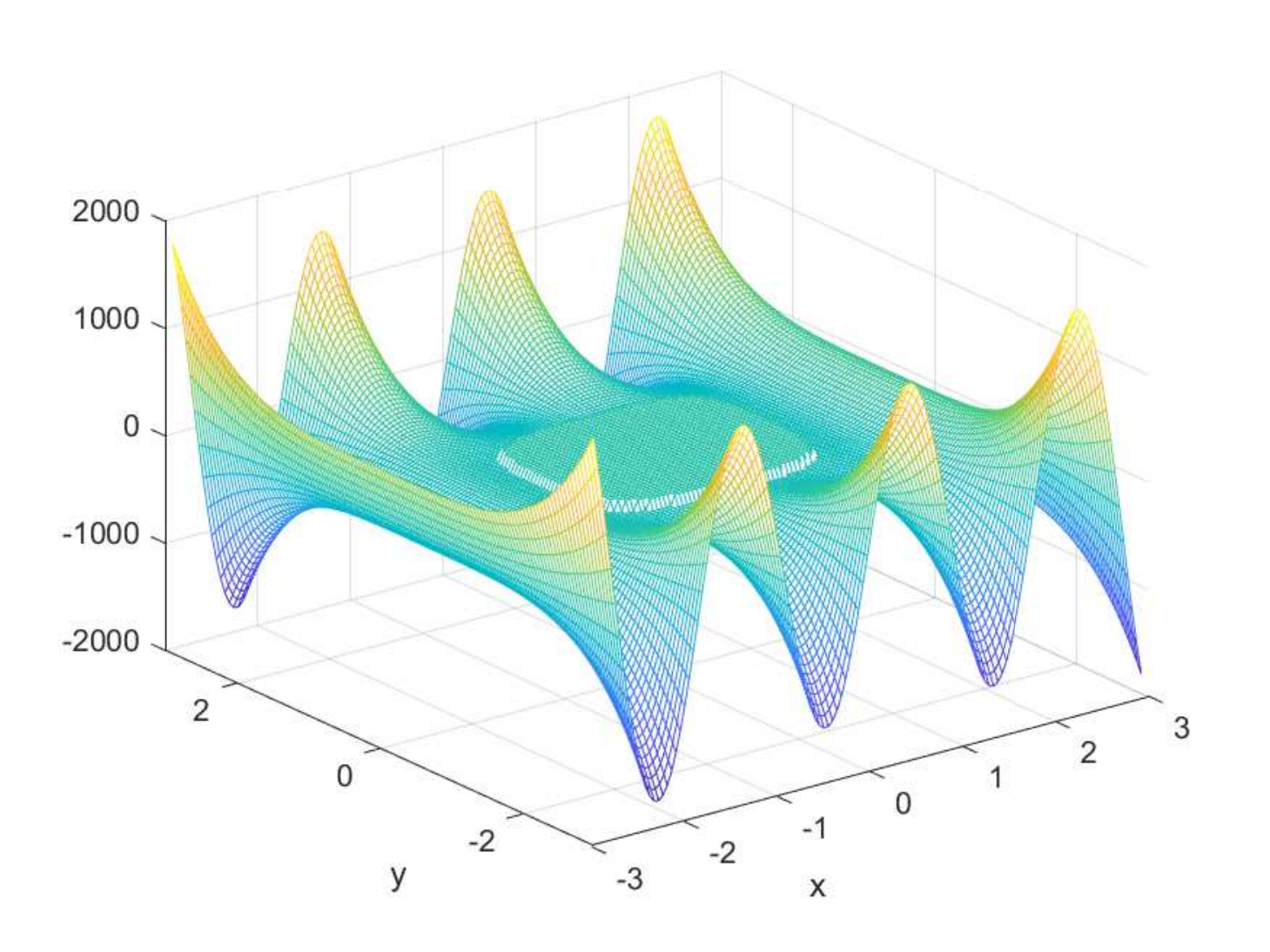}
	\end{subfigure}
	\begin{subfigure}[b]{0.3\textwidth}
		 \includegraphics[width=5.7cm,height=4.5cm]{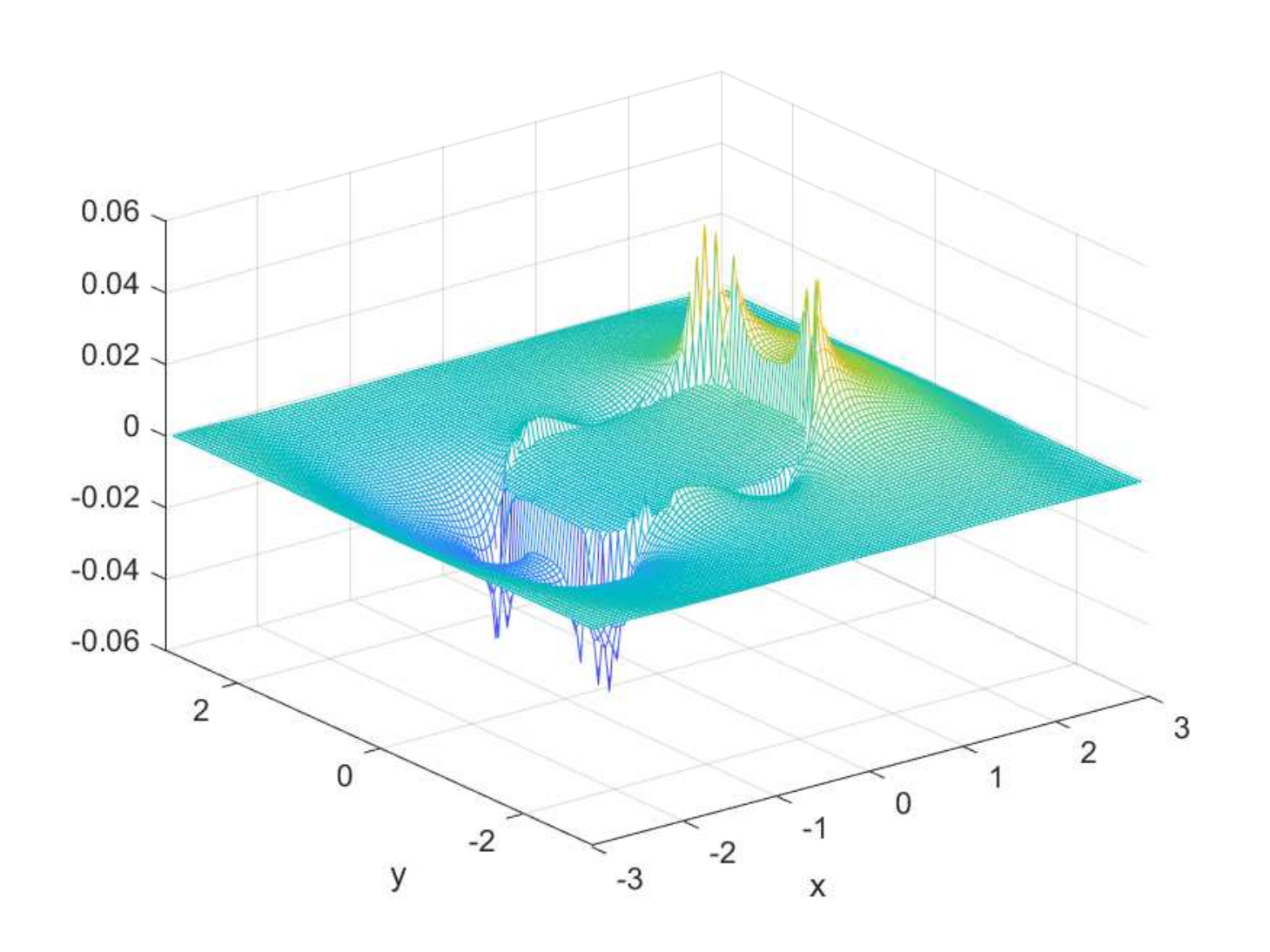}
	\end{subfigure}
	\begin{subfigure}[b]{0.3\textwidth}
		 \includegraphics[width=5.7cm,height=4.5cm]{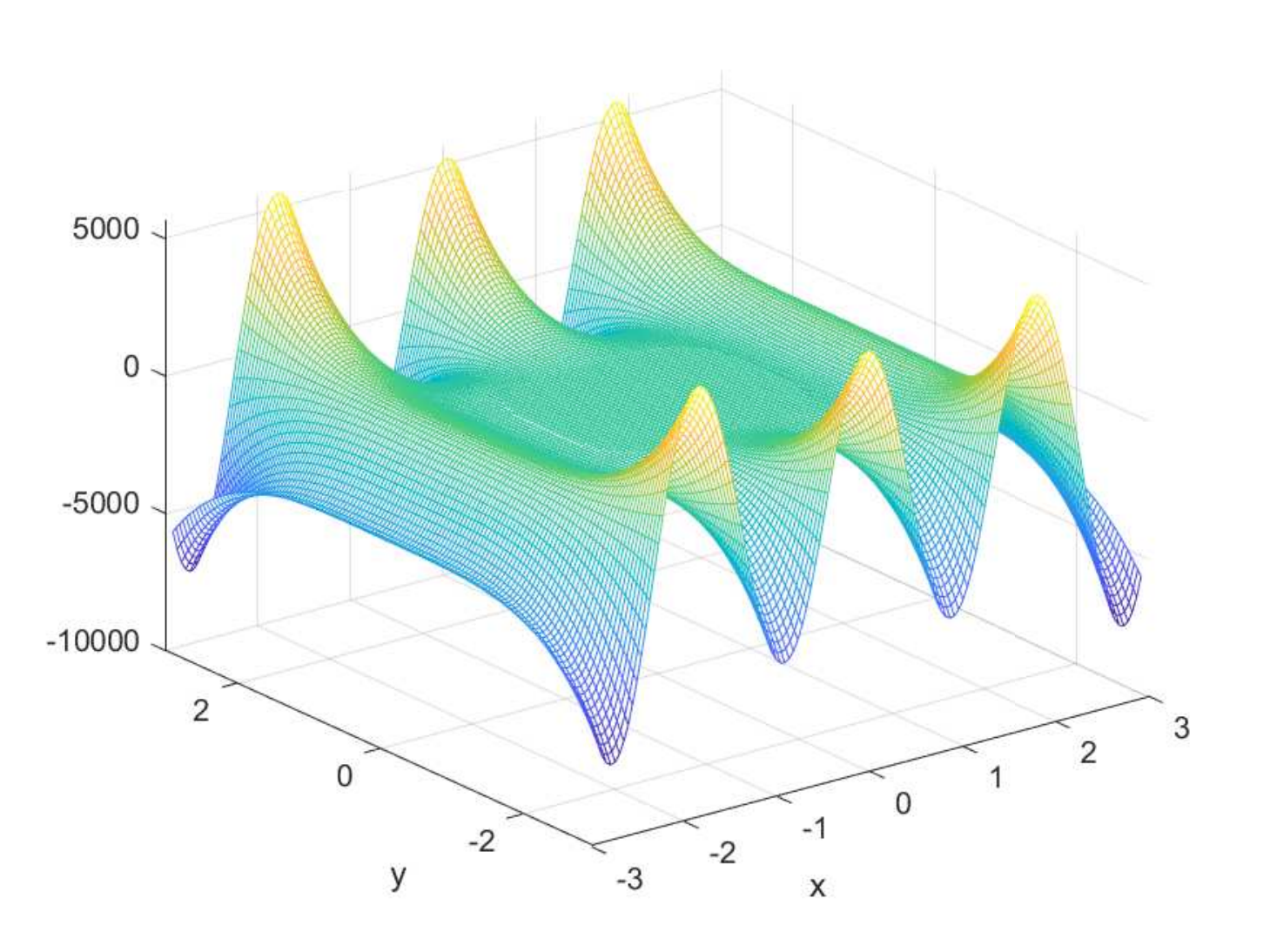}
	\end{subfigure}
	\begin{subfigure}[b]{0.3\textwidth}
		 \includegraphics[width=5.7cm,height=4.5cm]{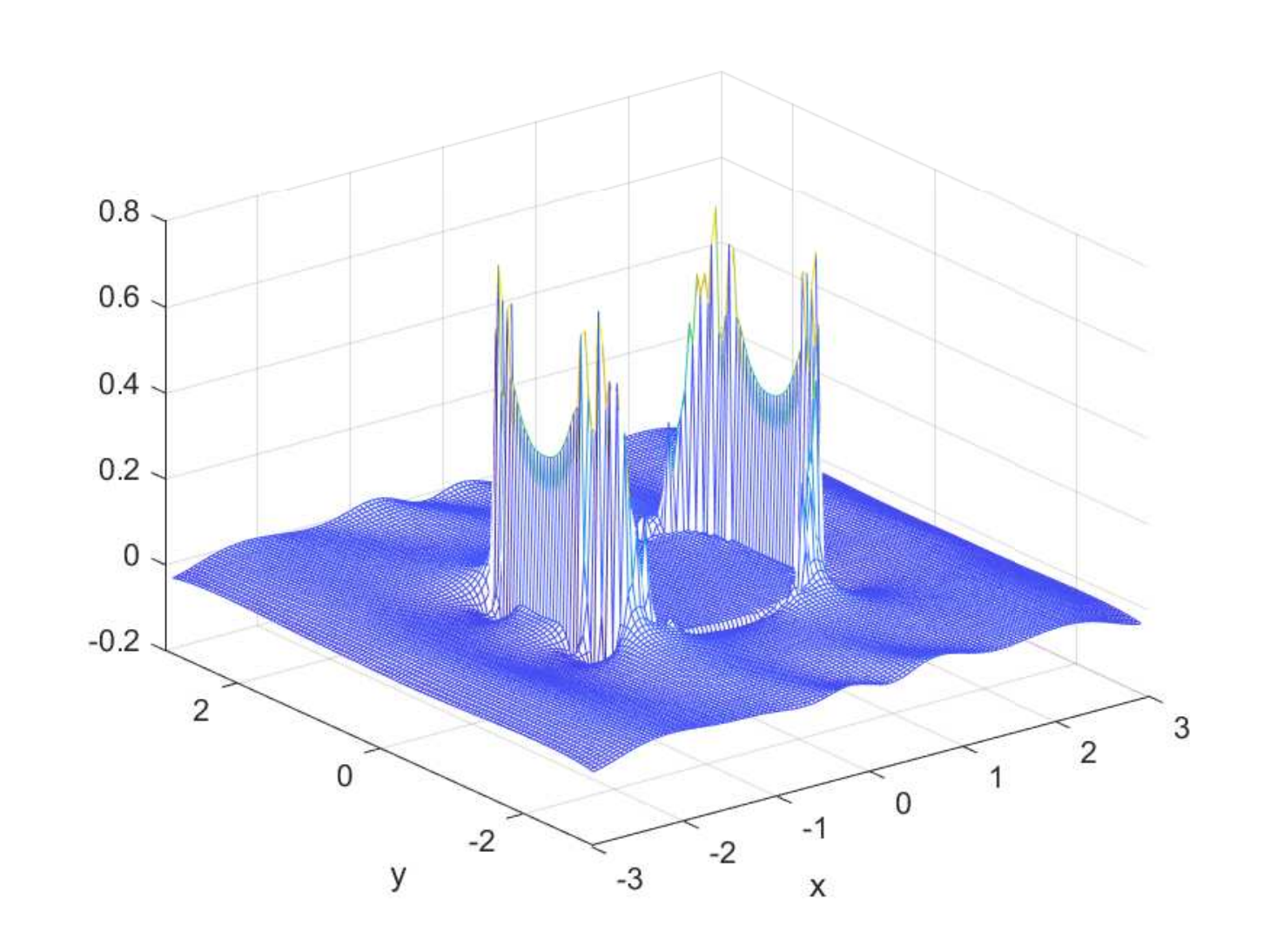}
	\end{subfigure}
	\caption
	{\tiny{Top row for \cref{Drafex2}: the interface curve $\Gamma$ (left), the coefficient $a(x,y)$ (middle)	and
		the numerical solution $u_h$ (right) with $h=2^{-7}\times 6$. Bottom row for \cref{Drafex2}: the error $u_h-u$ (left), the numerical $(u_h)_x$ (middle) and the error $(u_h)_x-u_x$ (right) with $h=2^{-7}\times 6$.}}
	\label{fig:figure2}
\end{figure}
\begin{example}\label{Drafex1}
\normalfont
Let $\Omega=(-\pi,\pi)^2$ and
the interface curve be given by
$\Gamma:=\{(x,y)\in \Omega \; :\; \psi(x,y)=0\}$ with
$\psi (x,y)=x^2+y^2-2$. Note that $\Gamma \cap \partial \Omega=\emptyset$,
the coefficient $a$ and	the exact solution $u$ of \eqref{Qeques1} are given by
\begin{align*}
	 &a_{+}=a\chi_{\Op}=\frac{2+\sin(x)\sin(y)}{100},
	\qquad a_{-}=a\chi_{\Om}=10(2+\sin(x)\sin(y)),\\
	 &u_{+}=u\chi_{\Op}=100\sin(-2x)(x^2+y^2-2),
	\qquad u_{-}=u\chi_{\Om}=\frac{\sin(-2x)(x^2+y^2-2)}{10}-100.
\end{align*}
All the functions $f,g_1,g_2,g$ in \eqref{Qeques1} can be obtained by plugging the above coefficient and exact solution into \eqref{Qeques1}. In particular,
$g_1=100$ and $g_2=0$.
The numerical results are presented in \cref{table:QSp1} and \cref{fig:figure1}.
\end{example}
\begin{table}[htbp]
	\caption{\tiny{Performance in \cref{Drafex1}  of the proposed high order compact finite difference scheme in \cref{thm:regular,thm:gradient:regular,thm:irregular,thm:gradient:irregular} on uniform Cartesian meshes with $h=2^{-J}\times2\pi$. $\kappa$ is the condition number of the coefficient matrix.}}
	\centering
	\setlength{\tabcolsep}{2mm}{
		\begin{tabular}{c|c|c|c|c|c|c|c}
			\hline
			$J$
			& $\frac{\|u_{h}-u\|_{2,\ind_{\Omega}}}{\|u\|_{2,\ind_{\Omega},h}}$
			
			&order &$\frac{|u_{h}-u|_{H^1,\ind_{\Omega}}}{|u|_{H^1,\ind_{\Omega},h}}$

			&order &  $\frac{|u_{h}-u|_{V,\ind_{\Omega}}}{|u|_{V,\ind_{\Omega},h}}$
			
			&order &  $\kappa$ \\
			\hline
3   &1.3309E+01   &0   &5.9842E+00   &0   &9.9219E+01   &0   &9.2851E+06\\
4   &3.5796E-02   &8.538   &4.5479E-02   &7.040   &1.6452E-01   &9.236   &2.7593E+06\\
6   &1.1989E-04   &4.111   &3.7322E-04   &3.465   &2.3716E-03   &3.058   &5.6730E+06\\
7   &5.6264E-06   &4.413   &2.9690E-05   &3.652   &1.6035E-04   &3.887   &2.9228E+07\\
8   &3.7420E-07   &3.910   &2.4650E-06   &3.590   &1.6760E-05   &3.258   &6.5794E+07\\
			\hline

	\end{tabular}}
	\label{table:QSp1}
\end{table}
\begin{figure}[htbp]
	\centering
	\begin{subfigure}[b]{0.3\textwidth}
		 \includegraphics[width=5.7cm,height=4.cm]{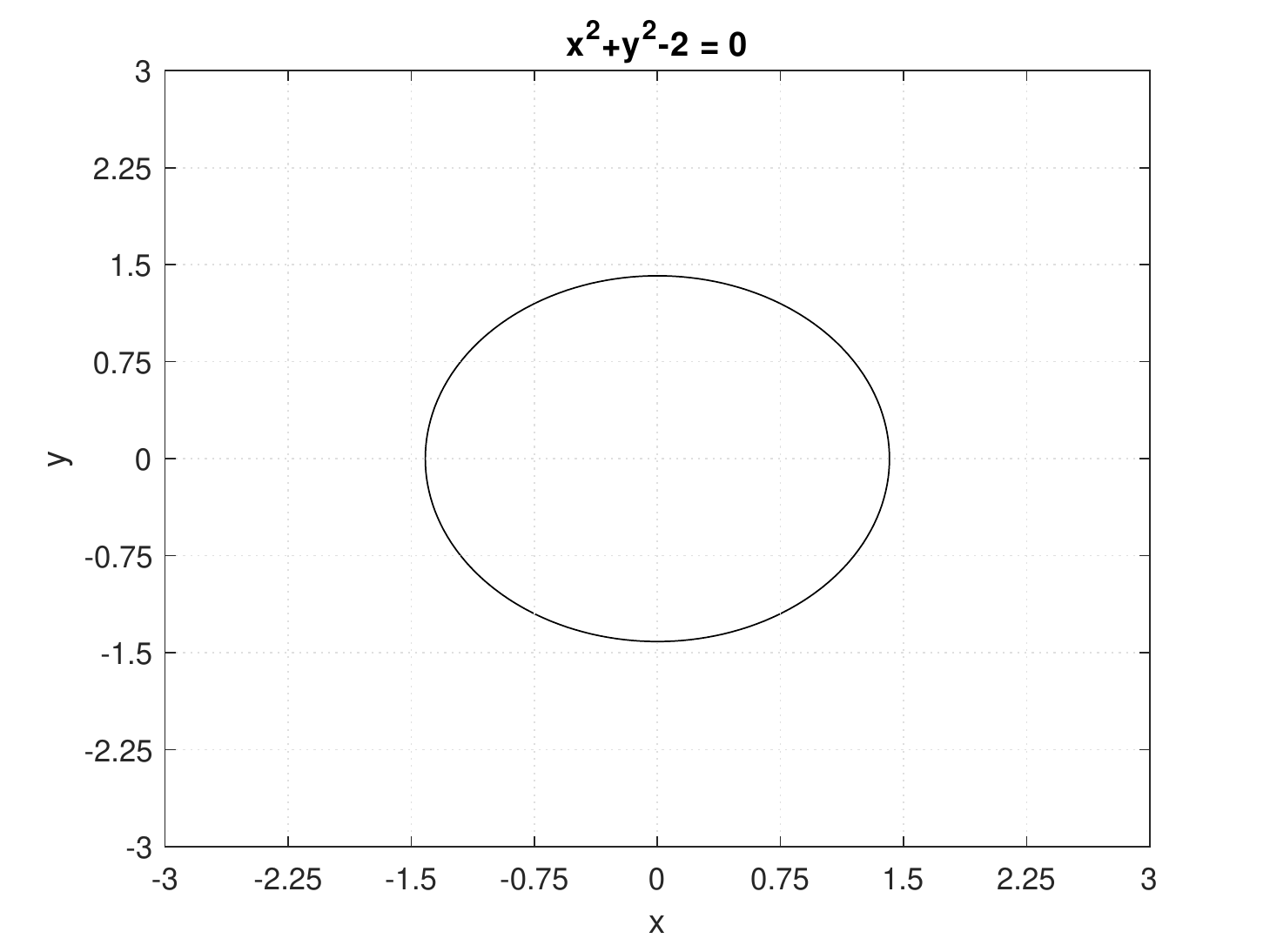}
	\end{subfigure}
	\begin{subfigure}[b]{0.3\textwidth}
		 \includegraphics[width=5.7cm,height=4.5cm]{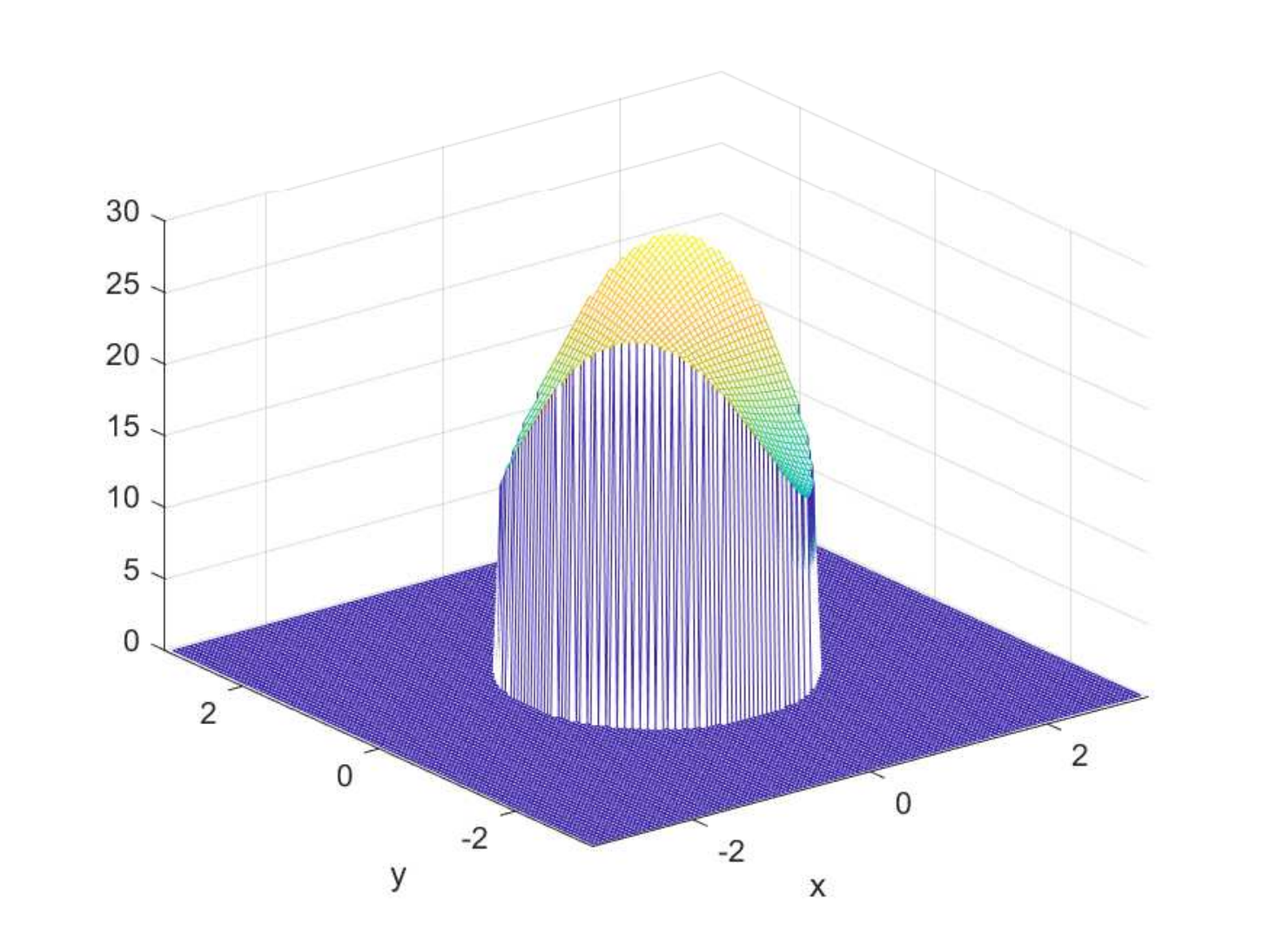}
	\end{subfigure}
	\begin{subfigure}[b]{0.3\textwidth}
		 \includegraphics[width=5.7cm,height=4.5cm]{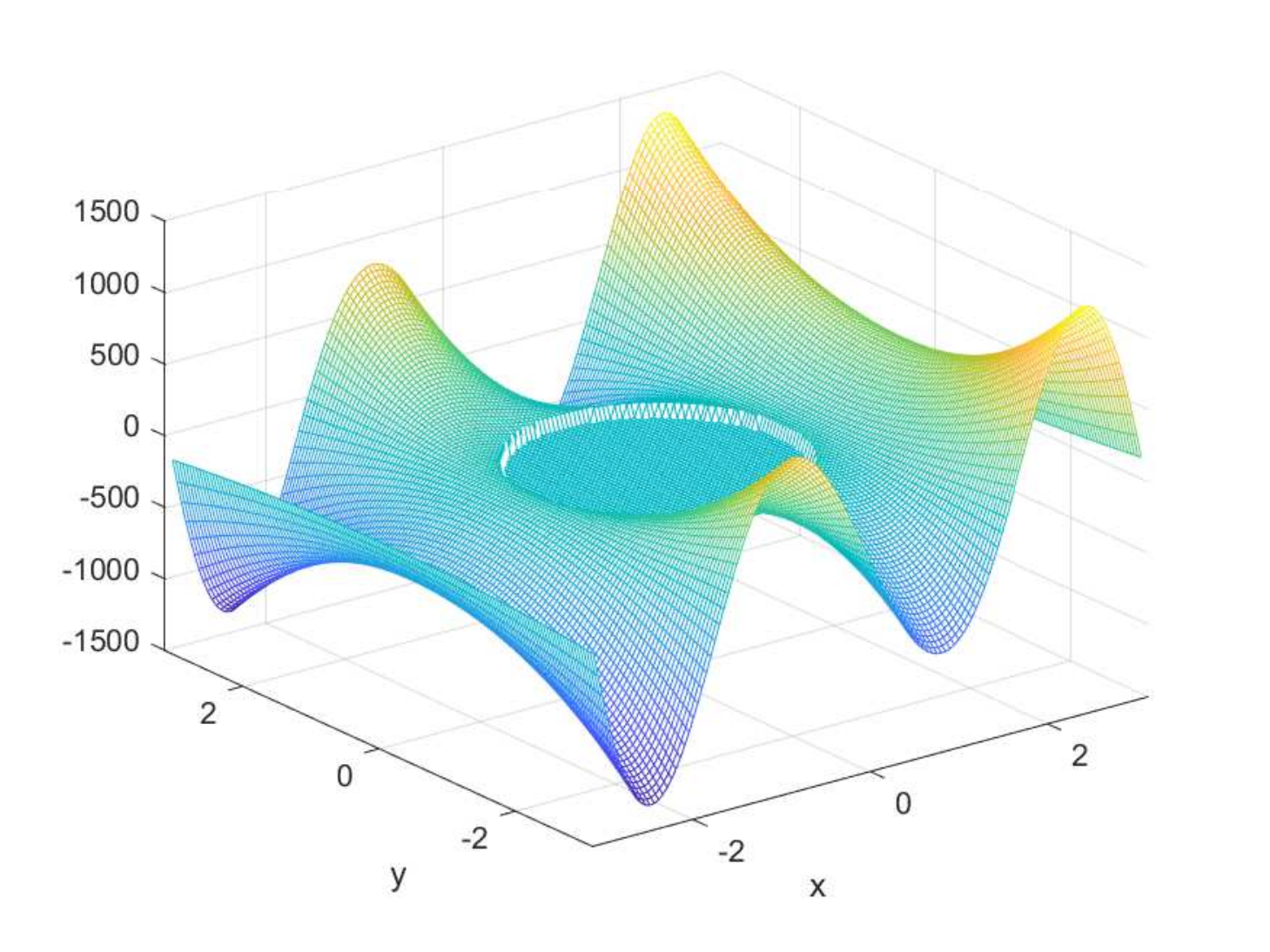}
	\end{subfigure}
	\begin{subfigure}[b]{0.3\textwidth}
		 \includegraphics[width=5.7cm,height=4.5cm]{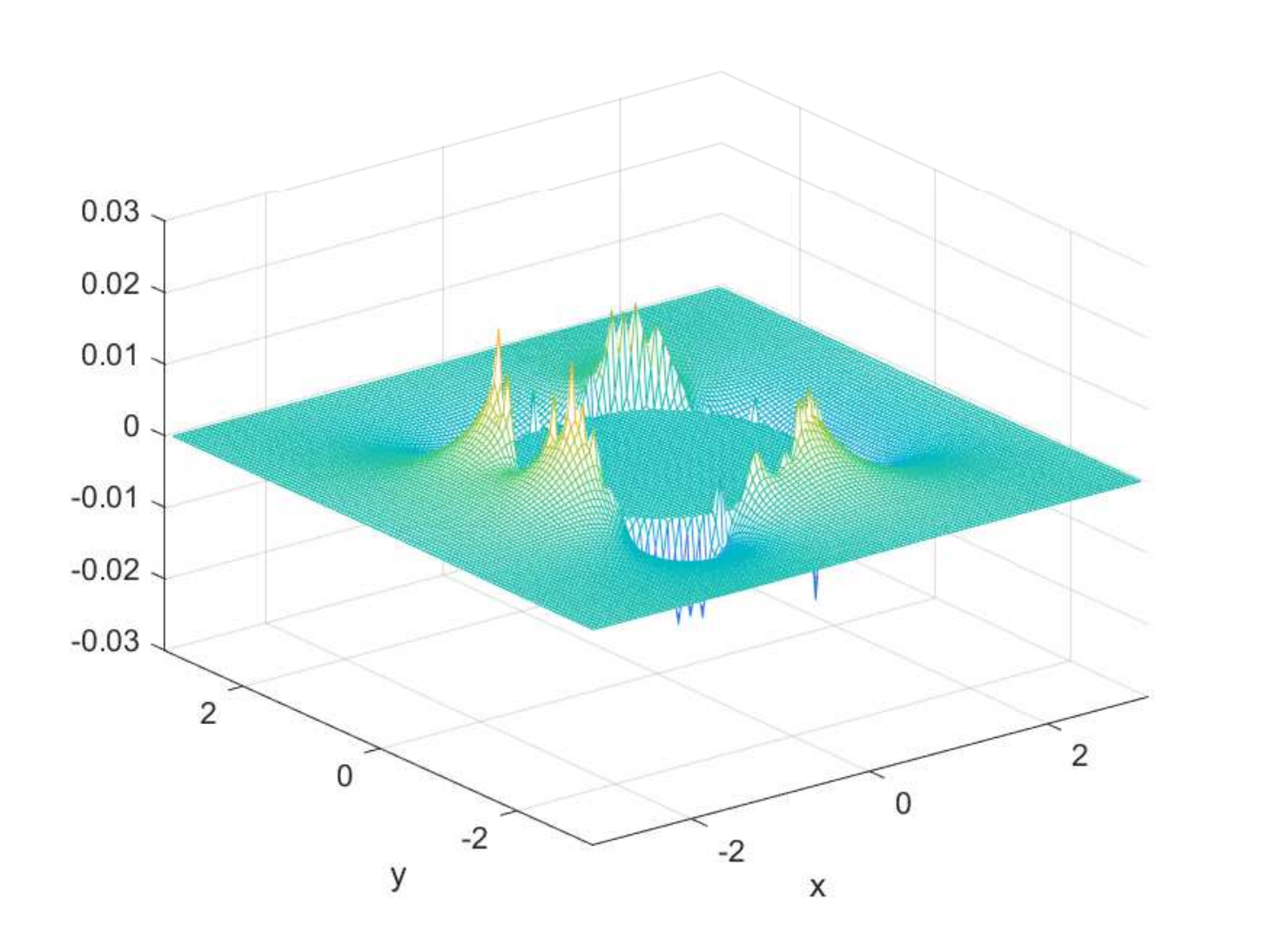}
	\end{subfigure}
	\begin{subfigure}[b]{0.3\textwidth}
		 \includegraphics[width=5.7cm,height=4.5cm]{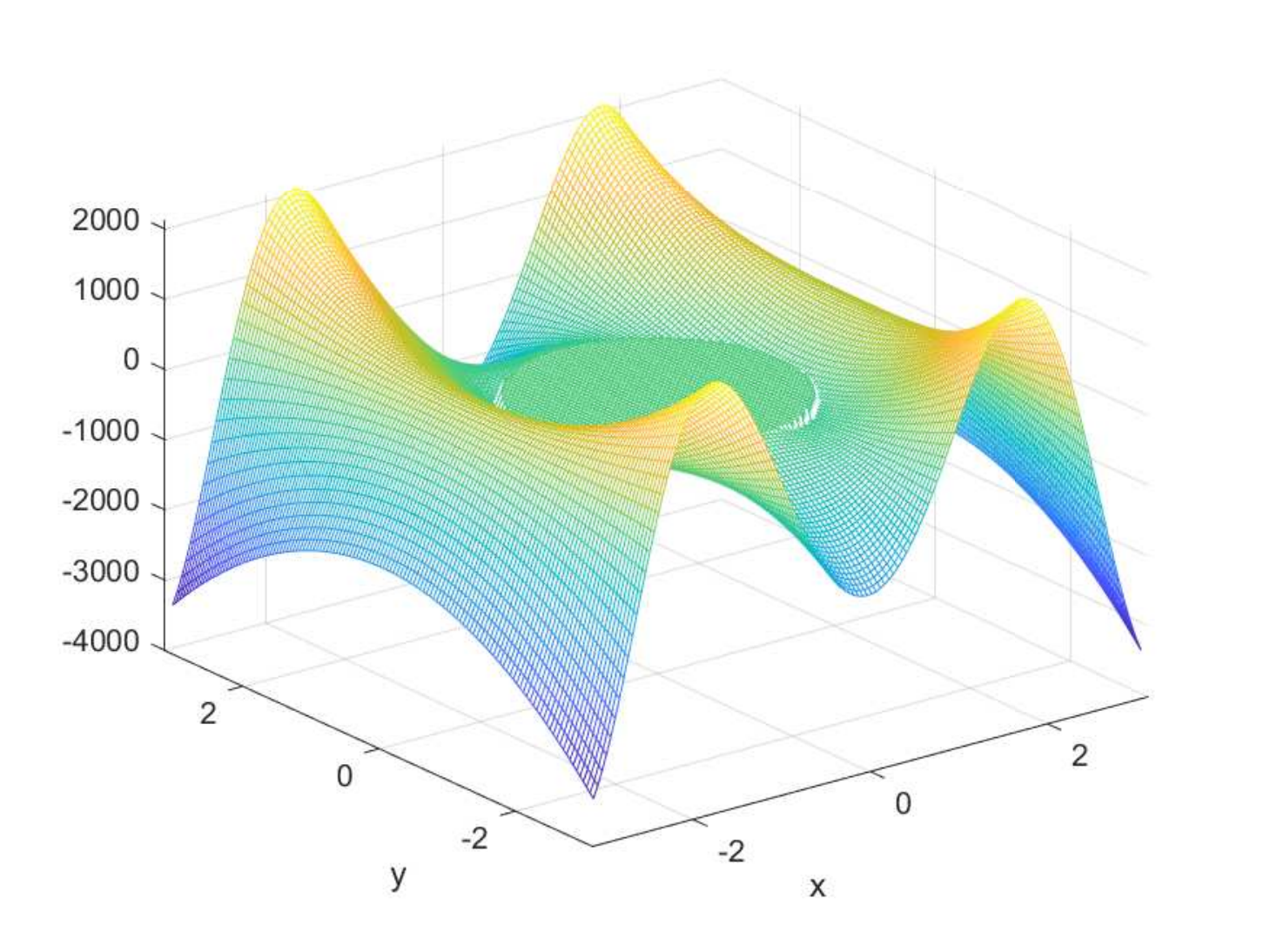}
	\end{subfigure}
	\begin{subfigure}[b]{0.3\textwidth}
		 \includegraphics[width=5.7cm,height=4.5cm]{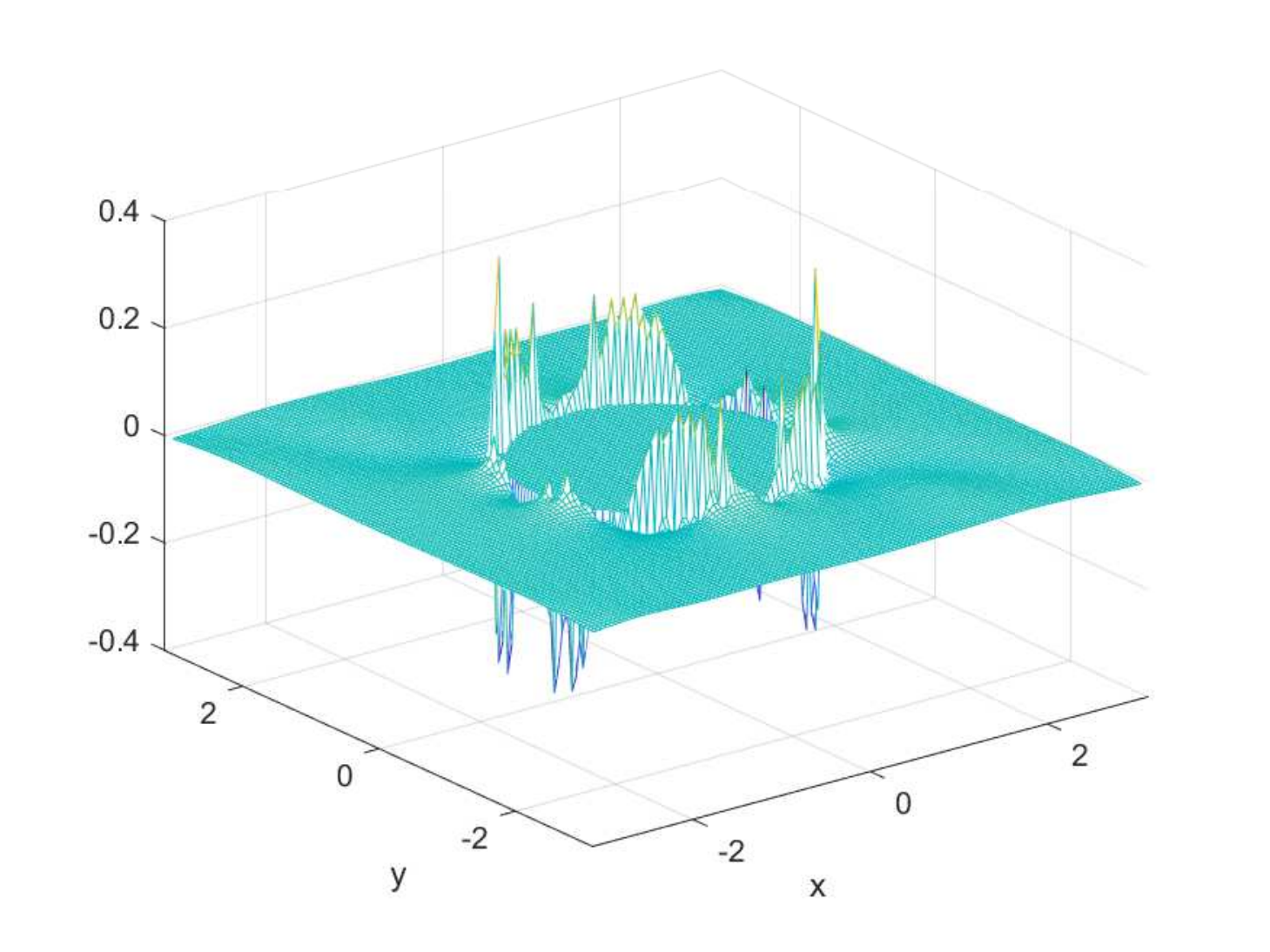}
	\end{subfigure}
	\caption
	{\tiny{Top row for \cref{Drafex1}: the interface curve $\Gamma$ (left), the coefficient $a(x,y)$ (middle)	and
		the numerical solution $u_h$ (right) with $h=2^{-7}\times 2\pi$. Bottom row for \cref{Drafex1}: the error $u_h-u$ (left), the numerical $(u_h)_x$ (middle) and the error $(u_h)_x-u_x$ (right) with $h=2^{-7}\times 2\pi$.}}
	\label{fig:figure1}
\end{figure}
\begin{example}\label{Drafex3}
	\normalfont
	Let $\Omega=(-\pi,\pi)^2$ and
	the interface curve be given by
	$\Gamma:=\{(x,y)\in \Omega \; :\; \psi(x,y)=0\}$ with
	$\psi (x,y)=y^2-2x^2+x^4-1$. Note that $\Gamma \cap \partial \Omega=\emptyset$,
the coefficient $a$ and	the exact solution $u$ of \eqref{Qeques1} are given by
	\begin{align*}
		 &a_{+}=a\chi_{\Op}=10(10+\sin(x+y)),
		\qquad a_{-}=a\chi_{\Om}=\frac{10+\sin(x+y)}{1000},\\
		 &u_{+}=u\chi_{\Op}=\frac{\sin(2y)(y^2-2x^2+x^4-1)}{10},
		\qquad u_{-}=u\chi_{\Om}=1000\sin(2y)(y^2-2x^2+x^4-1)+100.
	\end{align*}
	All the functions $f,g_1,g_2,g$ in \eqref{Qeques1} can be obtained by plugging the above coefficient and exact solution into \eqref{Qeques1}. In particular,
	$g_1=-100$ and $g_2=0$.
	The numerical results are presented in \cref{table:QSp3}  and \cref{fig:figure3}.	
\end{example}
\begin{table}[htbp]
	\caption{\tiny{Performance in \cref{Drafex3}  of the proposed  high order compact finite difference scheme in \cref{thm:regular,thm:gradient:regular,thm:irregular,thm:gradient:irregular} on uniform Cartesian meshes with $h=2^{-J}\times2\pi$. $\kappa$ is the condition number of the coefficient matrix.}}
	\centering
	\setlength{\tabcolsep}{2mm}{
		\begin{tabular}{c|c|c|c|c|c|c|c}
			\hline
$J$
& $\frac{\|u_{h}-u\|_{2,\ind_{\Omega}}}{\|u\|_{2,\ind_{\Omega},h}}$

&order &$\frac{|u_{h}-u|_{H^1,\ind_{\Omega}}}{|u|_{H^1,\ind_{\Omega},h}}$

&order &  $\frac{|u_{h}-u|_{V,\ind_{\Omega}}}{|u|_{V,\ind_{\Omega},h}}$

&order &  $\kappa$ \\
\hline
3   &6.1884E+01   &0   &3.9779E+01   &0   &6.7918E+00   &0   &9.9817E+05\\
4   &4.9790E-01   &6.958   &1.0408E+00   &5.256   &2.7858E-01   &4.608   &2.6762E+05\\
5   &3.3387E-02   &3.899   &2.0705E-01   &2.330   &3.9791E-02   &2.808   &3.8920E+04\\
6   &1.5988E-03   &4.384   &7.9473E-03   &4.703   &3.6630E-03   &3.441   &8.2465E+04\\
7   &1.0671E-04   &3.905   &5.1411E-04   &3.950   &2.9131E-04   &3.652   &1.7768E+05\\
8   &7.5627E-06   &3.819   &4.1588E-04   &0.306   &3.2258E-05   &3.175   &4.1282E+05\\
			\hline

	\end{tabular}}
	\label{table:QSp3}
\end{table}
\begin{figure}[htbp]
	\centering
	\begin{subfigure}[b]{0.3\textwidth}
		 \includegraphics[width=5.7cm,height=4.cm]{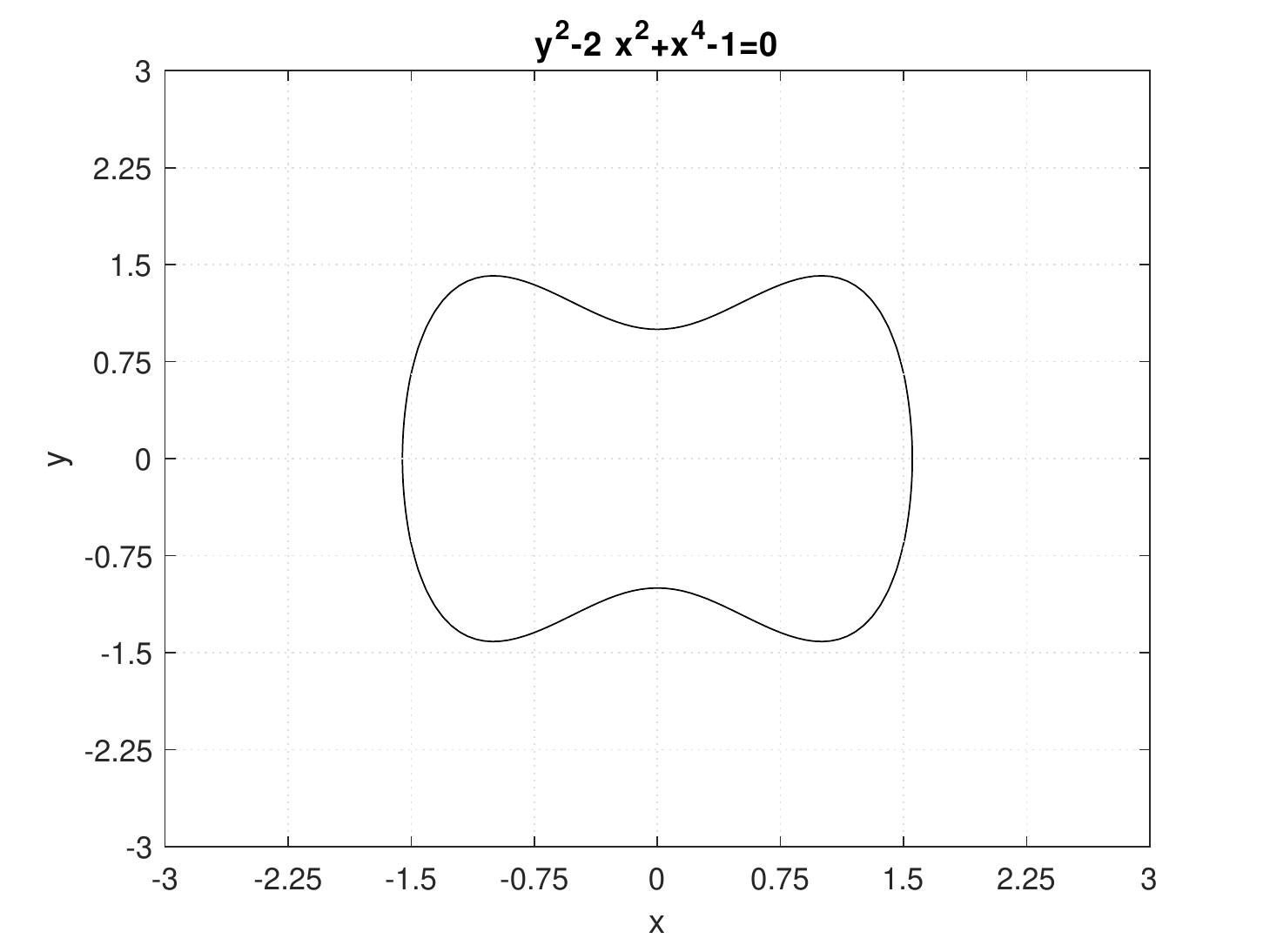}
	\end{subfigure}
	\begin{subfigure}[b]{0.3\textwidth}
		 \includegraphics[width=5.7cm,height=4.5cm]{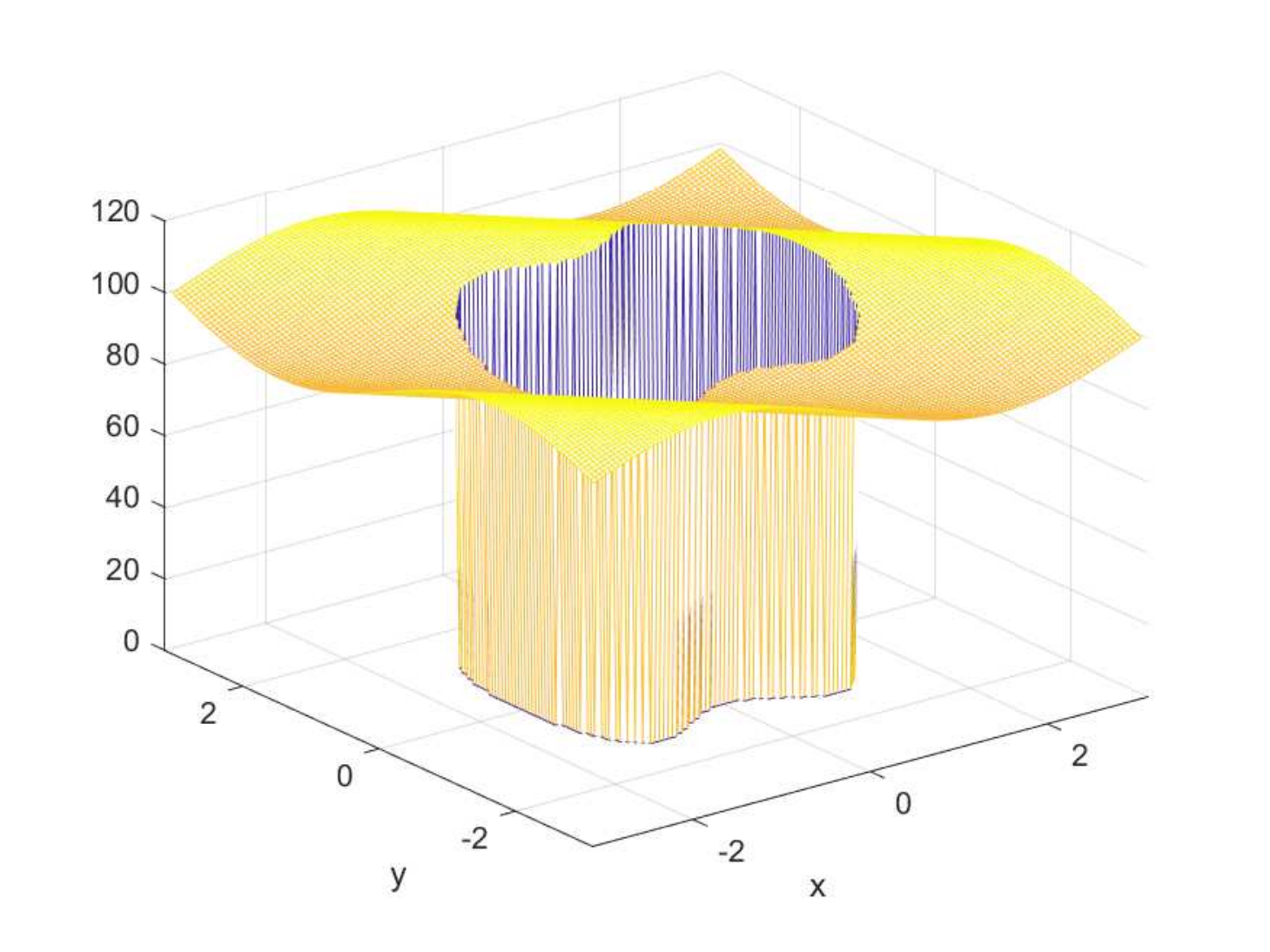}
	\end{subfigure}
	\begin{subfigure}[b]{0.3\textwidth}
		 \includegraphics[width=5.7cm,height=4.5cm]{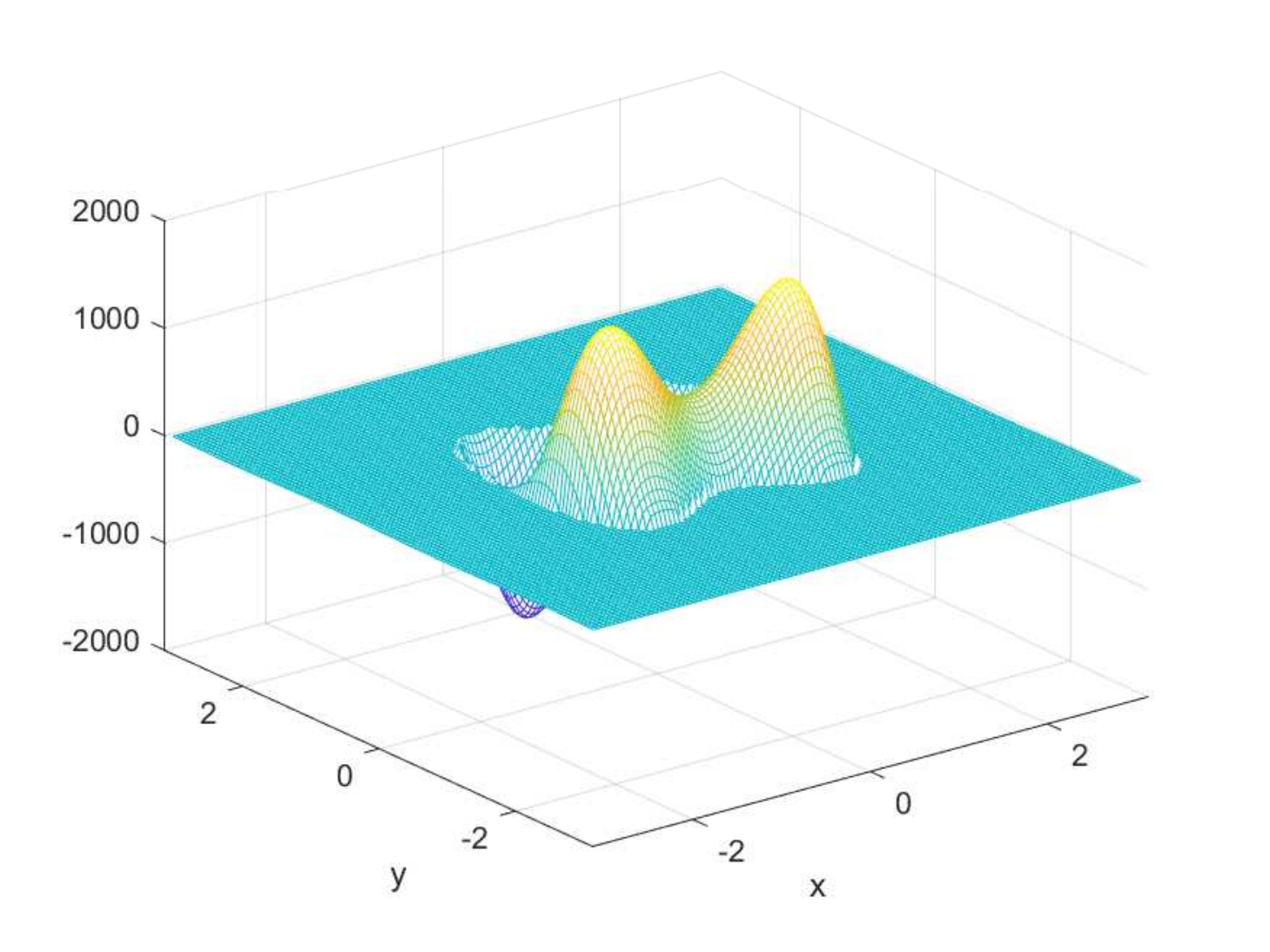}
	\end{subfigure}
	\begin{subfigure}[b]{0.3\textwidth}
		 \includegraphics[width=5.7cm,height=4.5cm]{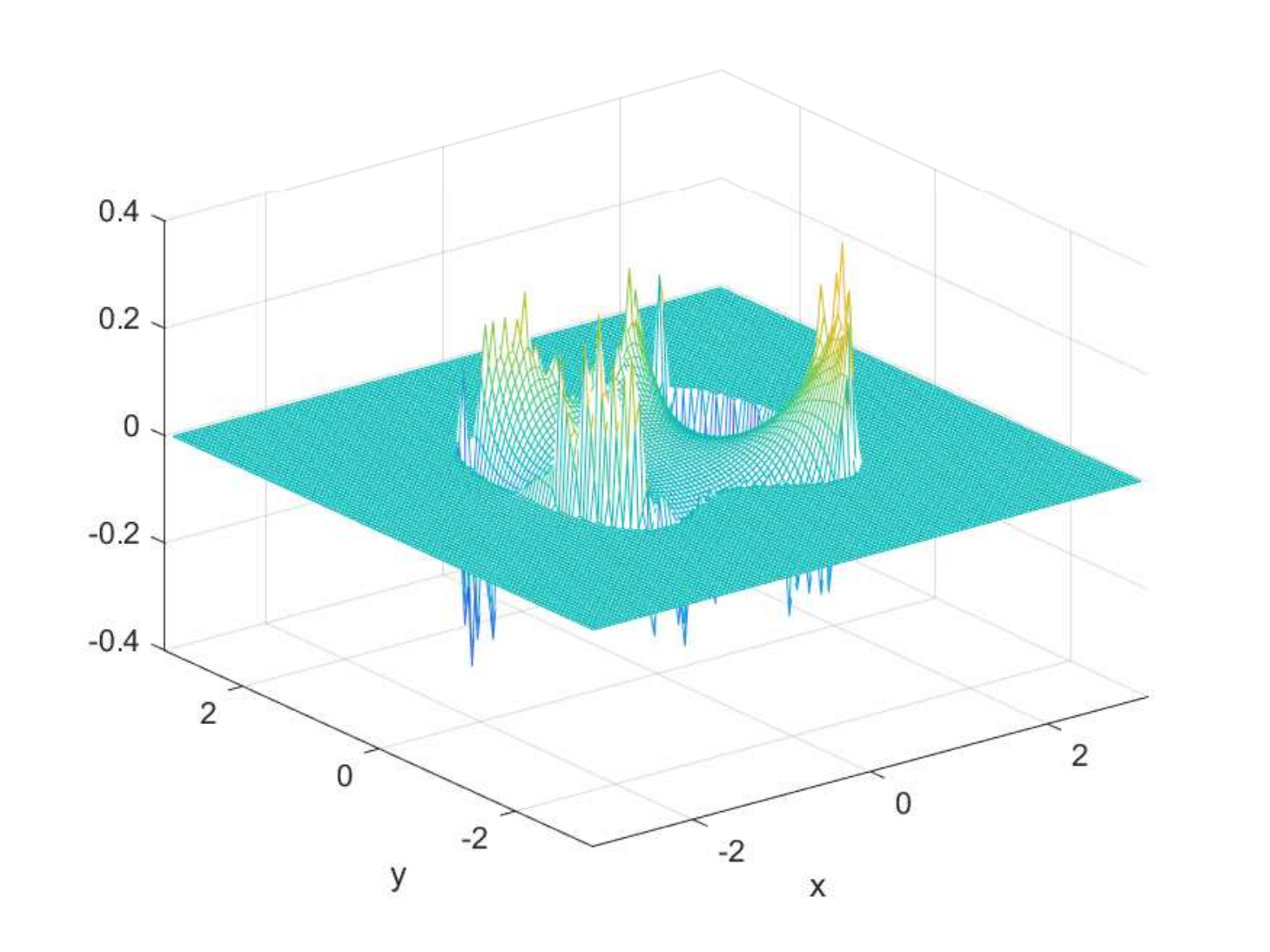}
	\end{subfigure}
	\begin{subfigure}[b]{0.3\textwidth}
		 \includegraphics[width=5.7cm,height=4.5cm]{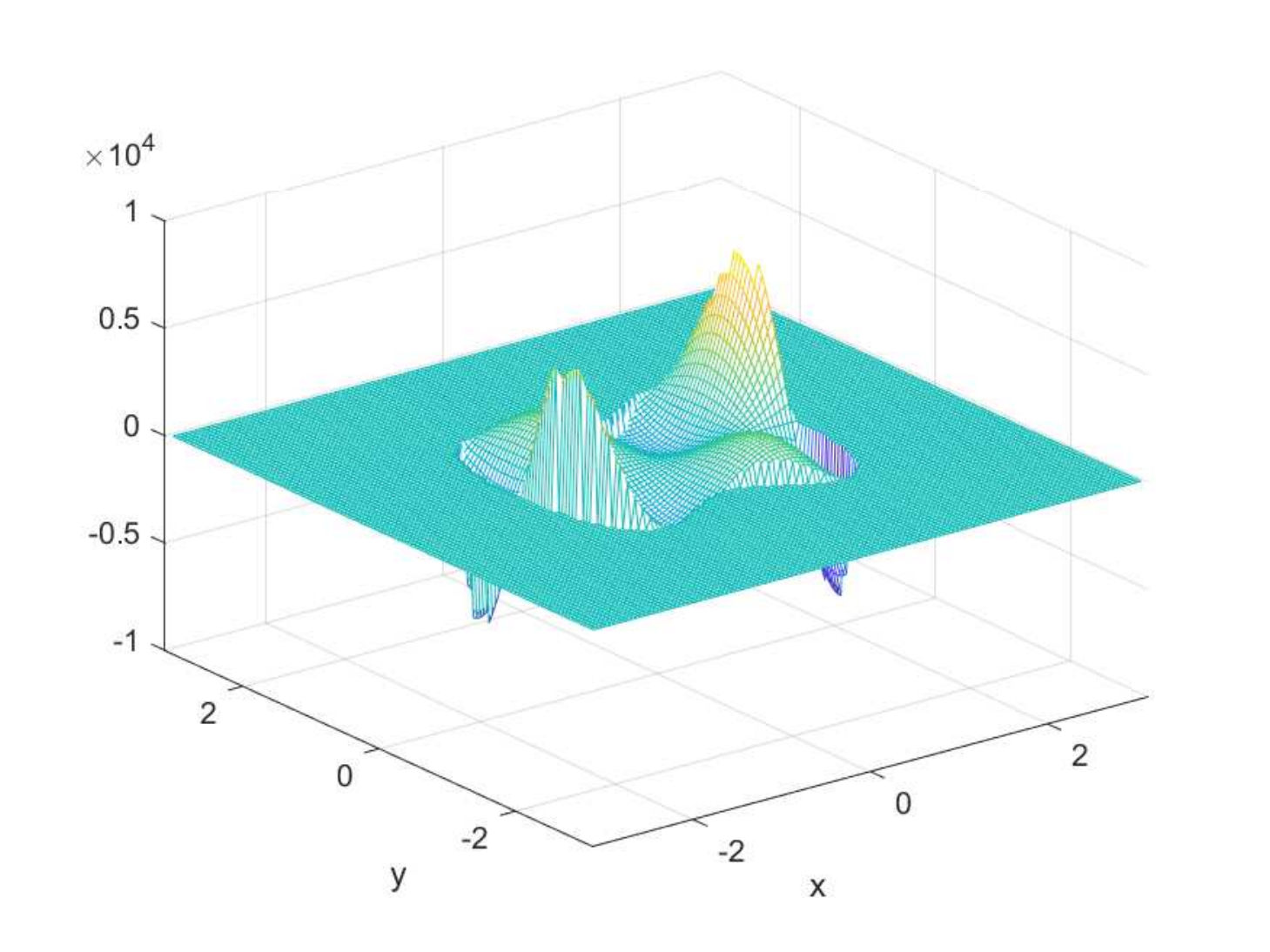}
	\end{subfigure}
	\begin{subfigure}[b]{0.3\textwidth}
		 \includegraphics[width=5.7cm,height=4.5cm]{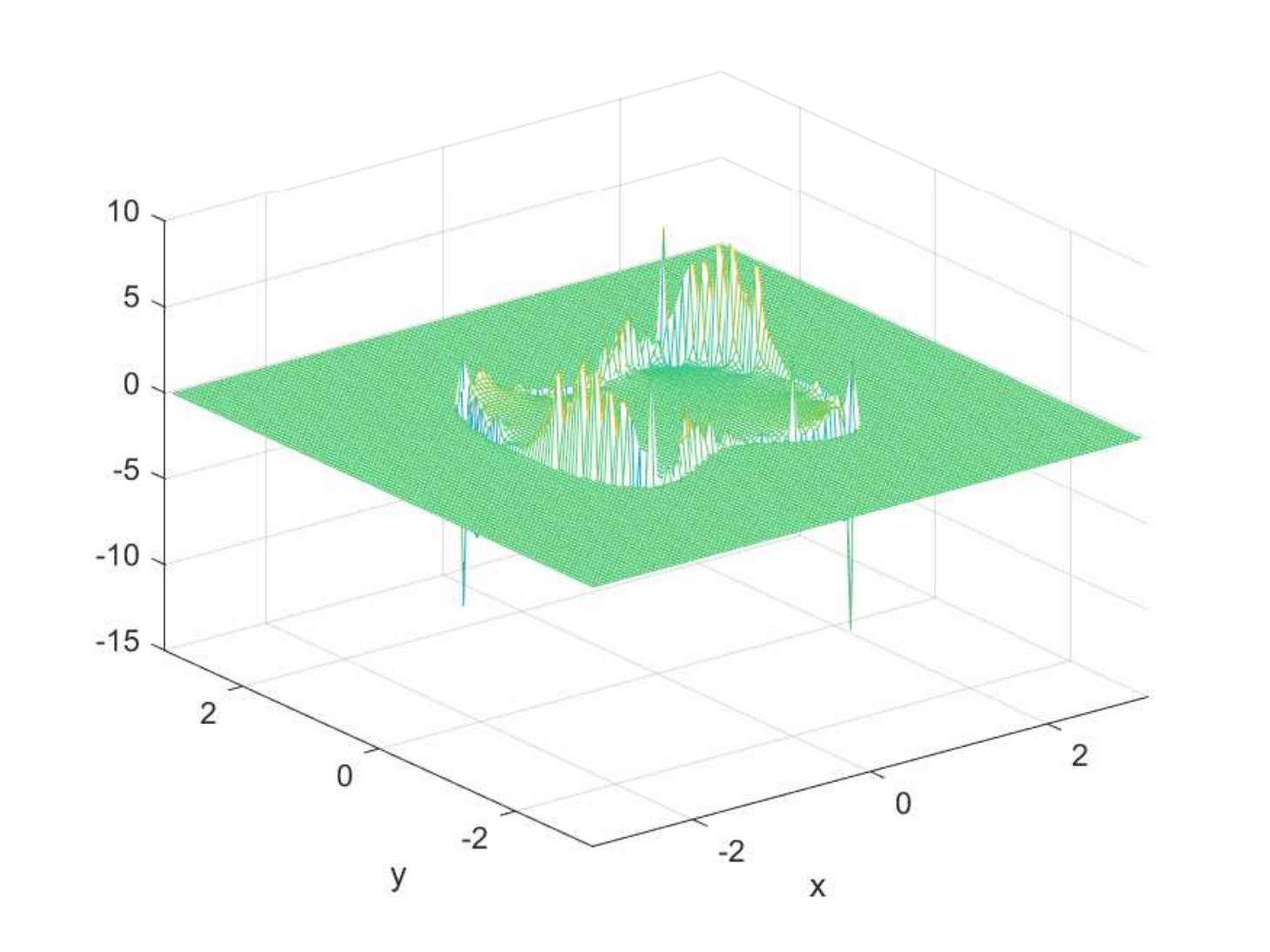}
	\end{subfigure}
	\caption
	{\tiny{Top row for \cref{Drafex3}: the interface curve $\Gamma$ (left), the coefficient $a(x,y)$ (middle)	and
		the numerical solution $u_h$ (right) with $h=2^{-7}\times 2\pi$. Bottom row for \cref{Drafex3}: the error $u_h-u$ (left), the numerical $(u_h)_x$ (middle) and the error $(u_h)_x-u_x$ (right) with $h=2^{-7}\times 2\pi$.}}
	\label{fig:figure3}
\end{figure}
\begin{example}\label{Drafex4}
	\normalfont
	Let $\Omega=(-2.5,2.5)^2$ and
	the interface curve be given by
	$\Gamma:=\{(x,y)\in \Omega \; :\; \psi(x,y)=0\}$ with
	$\psi (x,y)=2x^4+y^2-1/2$. Note that $\Gamma \cap \partial \Omega=\emptyset$,
the coefficient $a$ and	the exact solution $u$ of \eqref{Qeques1} are given by
	\begin{align*}
		&a_{+}=a\chi_{\Op}=10(\exp(x-y)),
		\qquad a_{-}=a\chi_{\Om}=\frac{\exp(x-y)}{1000},\\
		 &u_{+}=u\chi_{\Op}=\frac{\cos(4x)(2x^4+y^2-1/2)}{10},
		\qquad u_{-}=u\chi_{\Om}=1000\cos(4x)(2x^4+y^2-1/2)+100.
	\end{align*}
	All the functions $f,g_1,g_2,g$ in \eqref{Qeques1} can be obtained by plugging the above coefficient and exact solution into \eqref{Qeques1}. In particular,
	$g_1=-100$ and $g_2=0$.
	The numerical results are presented in \cref{table:QSp4}  and \cref{fig:figure4}.
\end{example}
\begin{table}[htbp]
	\caption{\tiny{Performance in \cref{Drafex4}  of the proposed  high order compact finite difference scheme in \cref{thm:regular,thm:gradient:regular,thm:irregular,thm:gradient:irregular} on uniform Cartesian meshes with $h=2^{-J}\times 5$. $\kappa$ is the condition number of the coefficient matrix.}}
	\centering
	\setlength{\tabcolsep}{2mm}{
		\begin{tabular}{c|c|c|c|c|c|c|c}
			\hline
$J$
& $\frac{\|u_{h}-u\|_{2,\ind_{\Omega}}}{\|u\|_{2,\ind_{\Omega},h}}$

&order &$\frac{|u_{h}-u|_{H^1,\ind_{\Omega}}}{|u|_{H^1,\ind_{\Omega},h}}$

&order &  $\frac{|u_{h}-u|_{V,\ind_{\Omega}}}{|u|_{V,\ind_{\Omega},h}}$

&order &  $\kappa$ \\
\hline
3   &2.1935E+02   &0   &4.7152E+04   &0   &2.7271E+01   &0   &1.1680E+08\\
4   &2.5293E+00   &6.438   &1.7464E+00   &14.721   &7.9277E-02   &8.426   &1.6969E+04\\
5   &4.0984E-01   &2.626   &1.0324E+00   &0.758   &5.1915E-03   &3.933   &6.9287E+04\\
6   &3.6886E-02   &3.474   &4.9797E-02   &4.374   &3.0372E-04   &4.095   &5.8211E+04\\
7   &2.2165E-03   &4.057   &4.3521E-03   &3.516   &1.7765E-05   &4.096   &4.8458E+04\\
8   &1.3130E-04   &4.077   &3.8261E-04   &3.508   &1.7755E-06   &3.323   &1.4692E+05\\
			\hline

	\end{tabular}}
	\label{table:QSp4}
\end{table}
\begin{figure}[htbp]
	\centering
	\begin{subfigure}[b]{0.3\textwidth}
		 \includegraphics[width=5.7cm,height=4.cm]{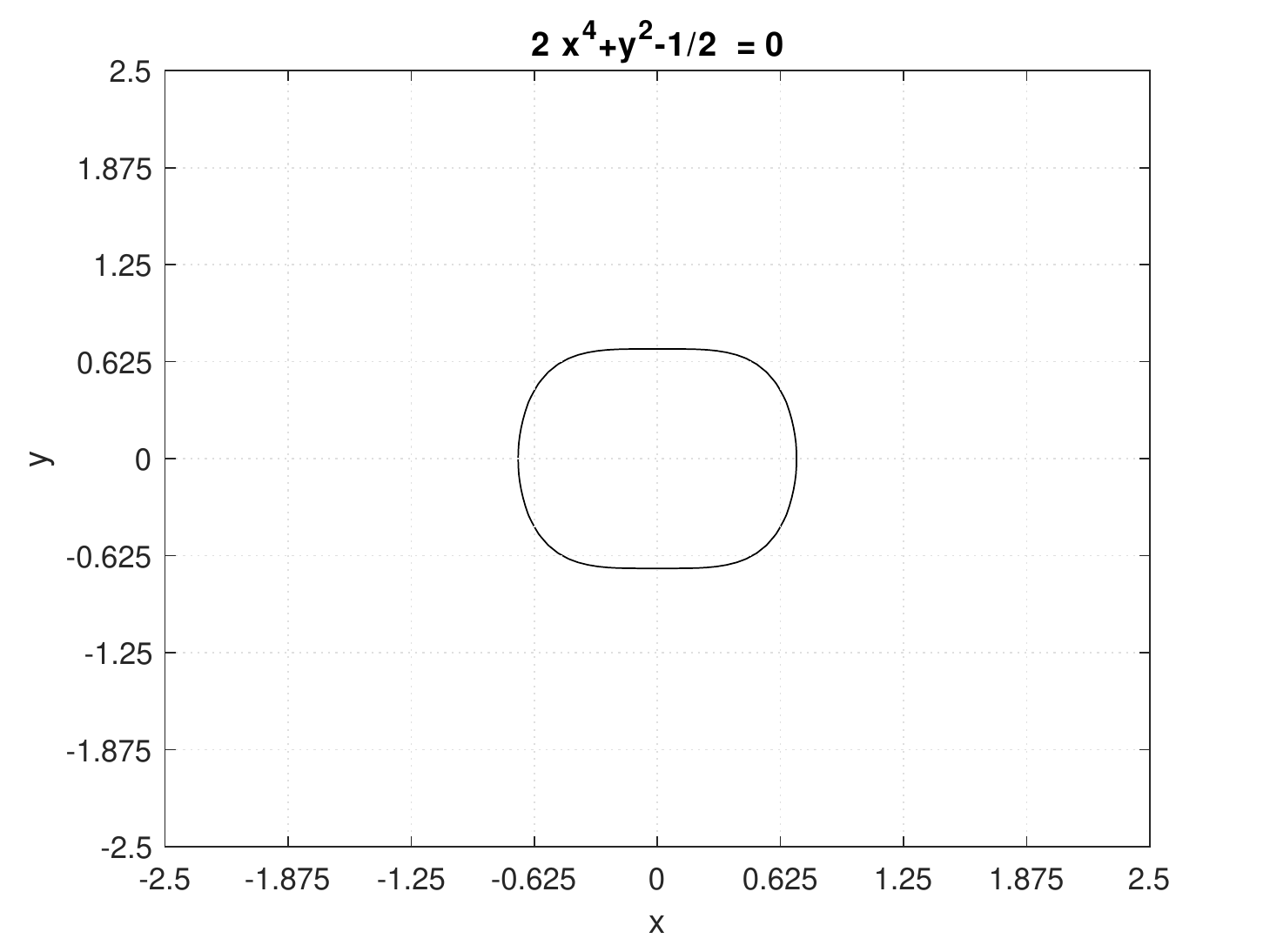}
	\end{subfigure}
	\begin{subfigure}[b]{0.3\textwidth}
		 \includegraphics[width=5.7cm,height=4.5cm]{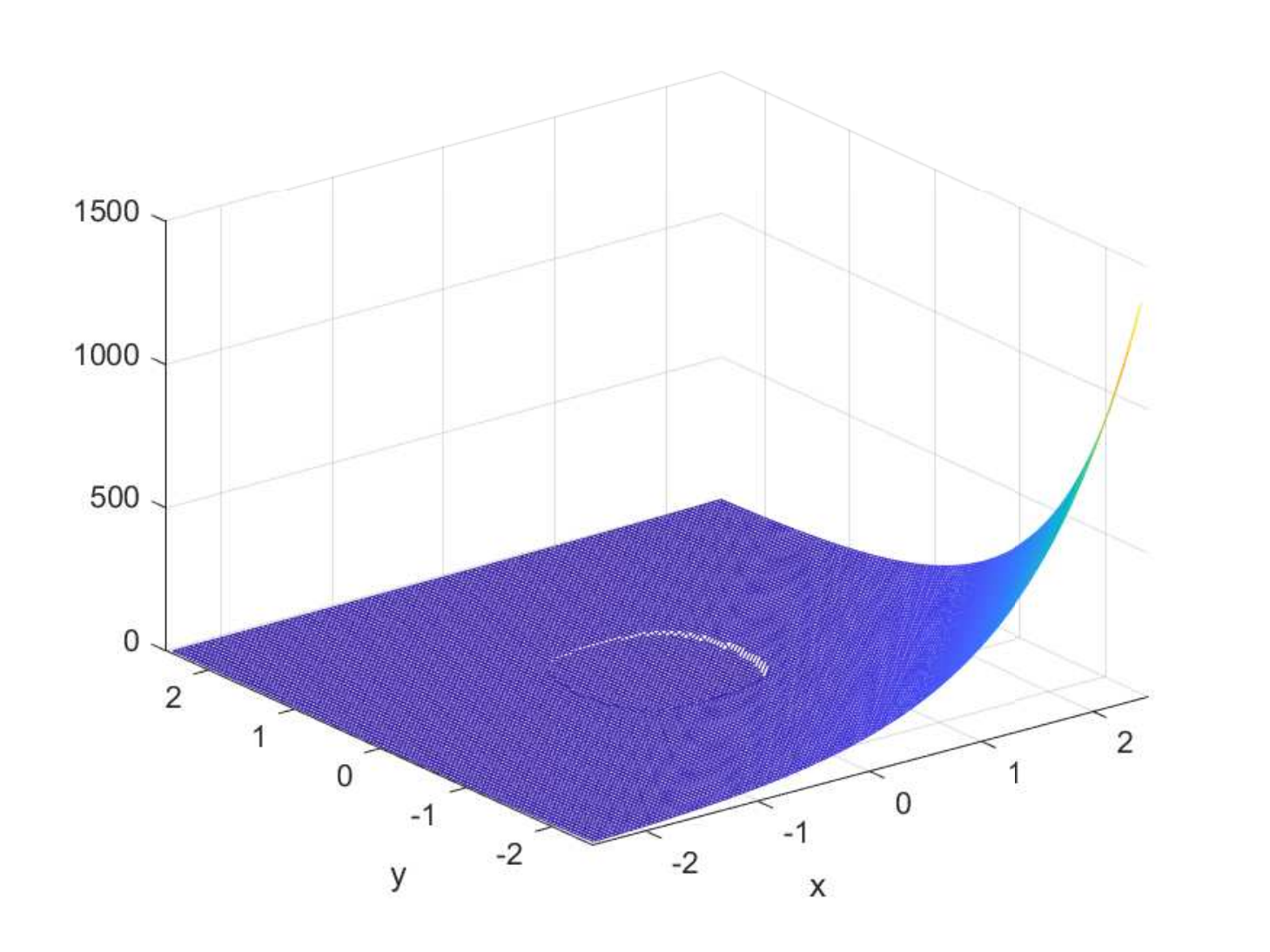}
	\end{subfigure}
	\begin{subfigure}[b]{0.3\textwidth}
		 \includegraphics[width=5.7cm,height=4.5cm]{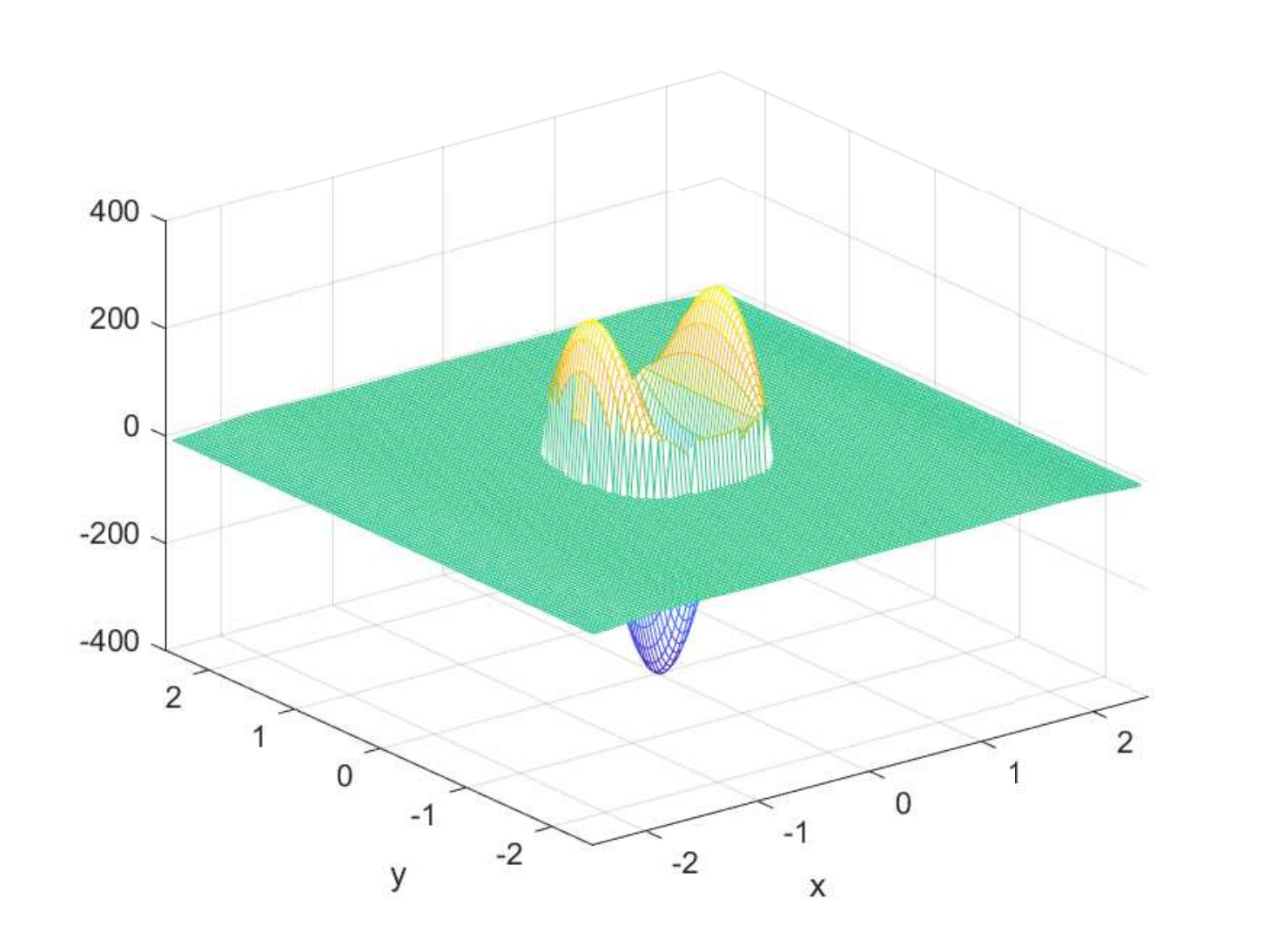}
	\end{subfigure}
	\begin{subfigure}[b]{0.3\textwidth}
		 \includegraphics[width=5.7cm,height=4.5cm]{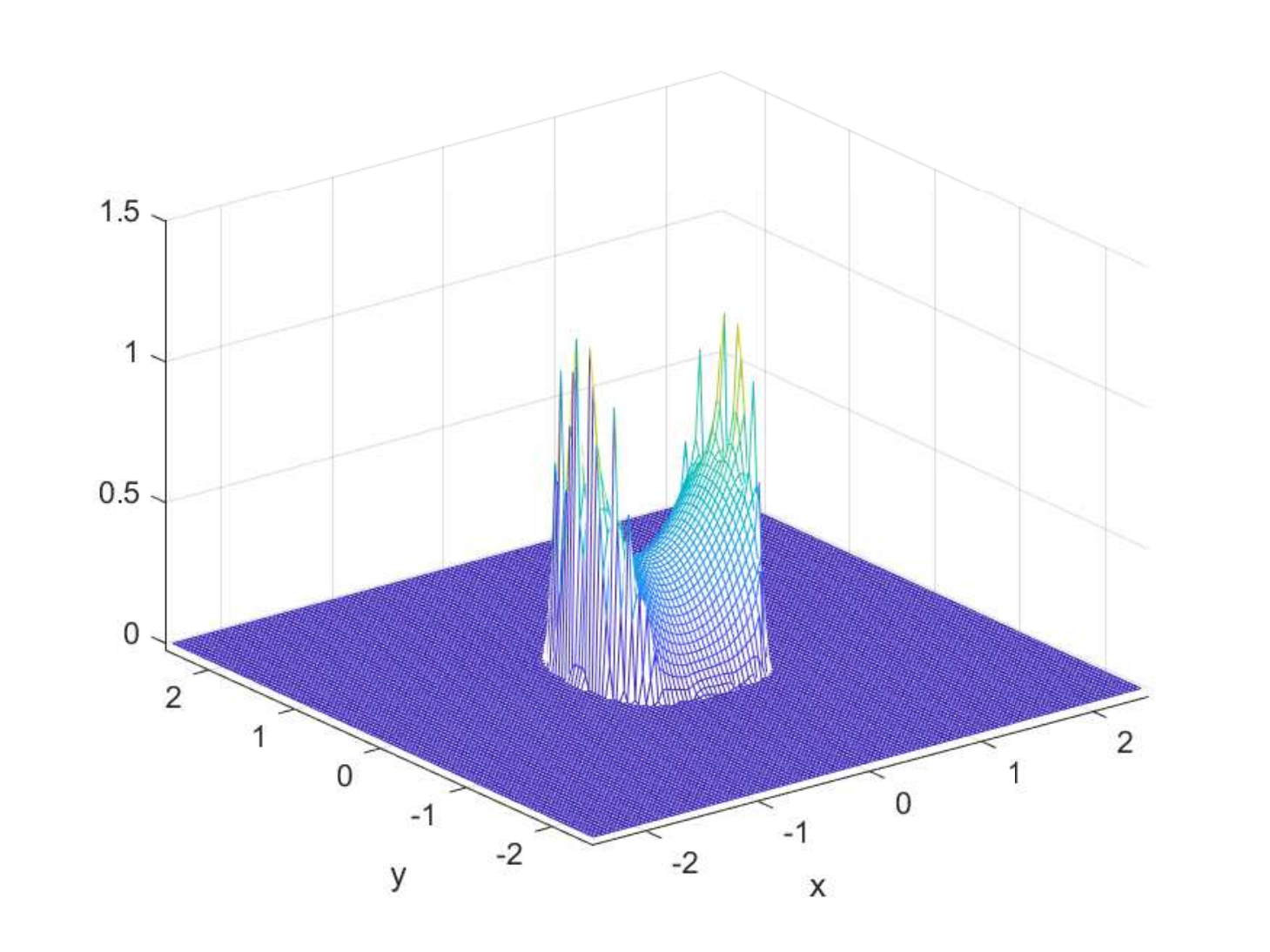}
	\end{subfigure}
	\begin{subfigure}[b]{0.3\textwidth}
		 \includegraphics[width=5.7cm,height=4.5cm]{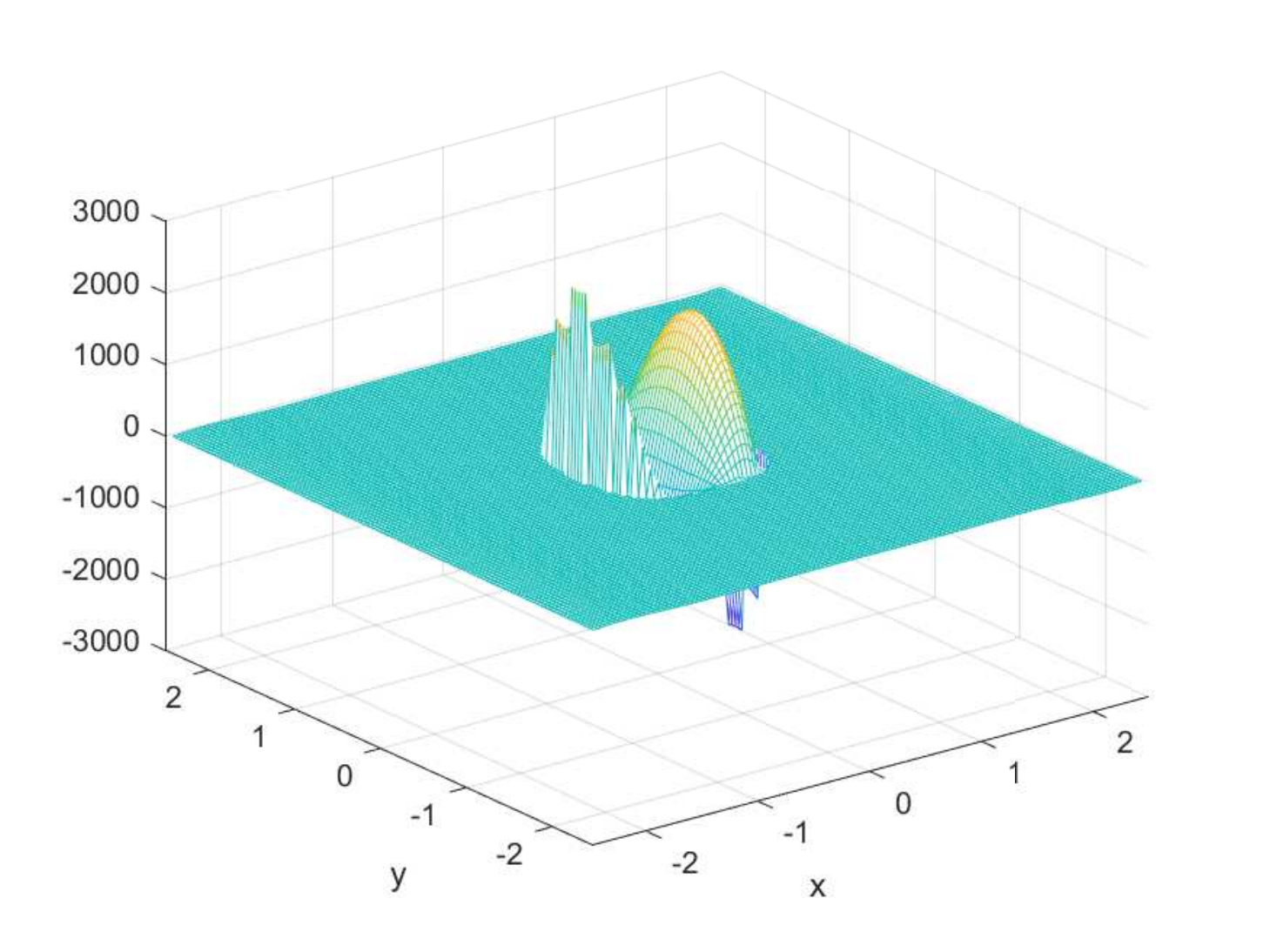}
	\end{subfigure}
	\begin{subfigure}[b]{0.3\textwidth}
		 \includegraphics[width=5.7cm,height=4.5cm]{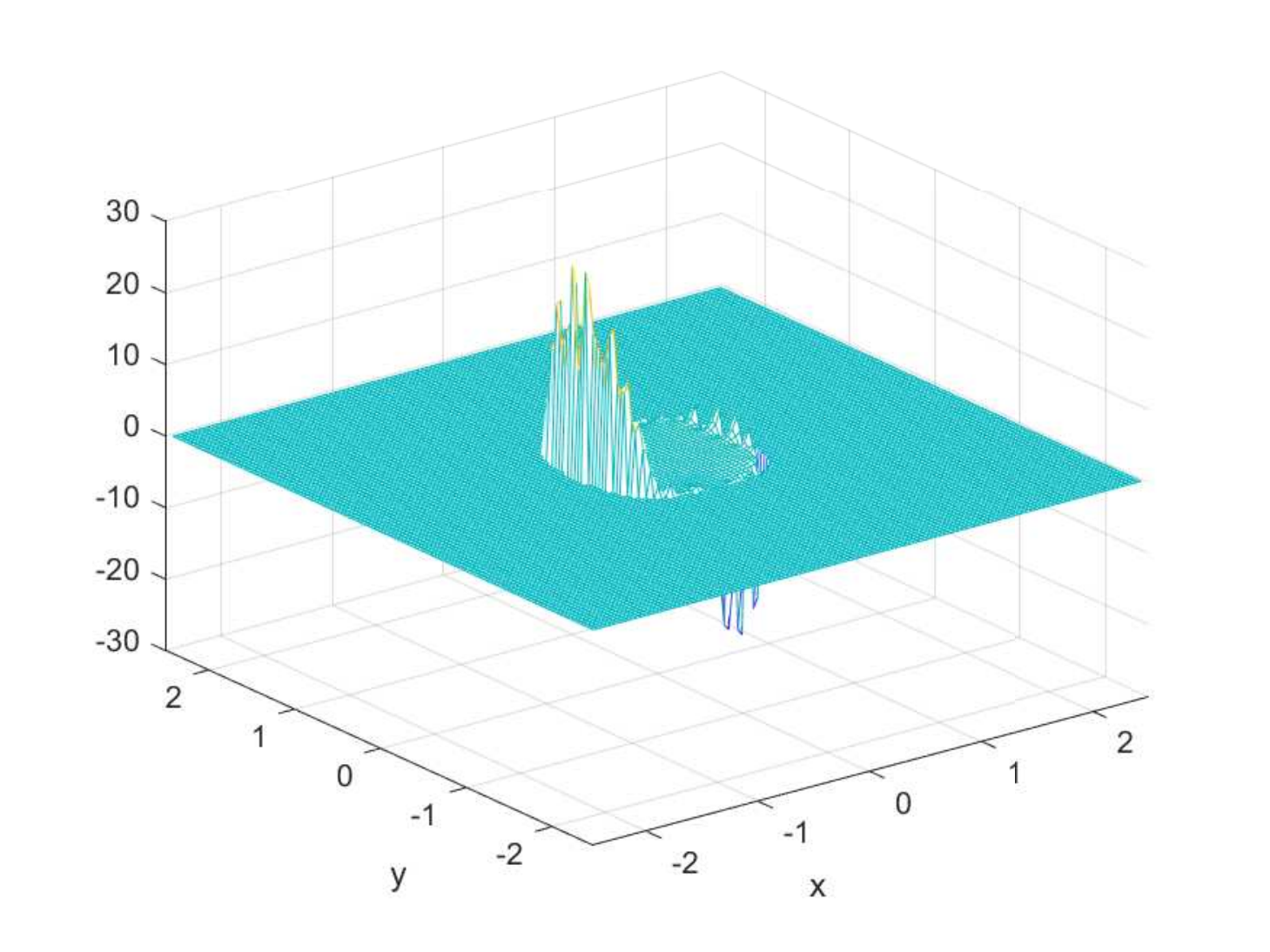}
	\end{subfigure}
	\caption
	{\tiny{Top row for \cref{Drafex4}: the interface curve $\Gamma$ (left), the coefficient $a(x,y)$ (middle)	and
		the numerical solution $u_h$ (right) with $h=2^{-7}\times 5$. Bottom row for \cref{Drafex4}: the error $u_h-u$ (left), the numerical $(u_h)_x$ (middle) and the error $(u_h)_x-u_x$ (right) with $h=2^{-7}\times 5$.}}
	\label{fig:figure4}
\end{figure}
\begin{example}\label{Drafex5}
	\normalfont
	Let $\Omega=(-\frac{2\pi}{3},\frac{2\pi}{3})^2$ and
	the interface curve be given by
	$\Gamma:=\{(x,y)\in \Omega \; :\; \psi(x,y)=0\}$ with
	$\psi (x,y)=y^2+\frac{2x^2}{x^2+1}-1$. Note that $\Gamma \cap \partial \Omega=\emptyset$,
the coefficient $a$ and	the exact solution $u$ of \eqref{Qeques1} are given by
	\begin{align*}
		 &a_{+}=a\chi_{\Op}=100(2+\cos(x)\sin(y)),
		\qquad a_{-}=a\chi_{\Om}=\frac{2+\cos(x)\sin(y)}{10},\\
		 &u_{+}=u\chi_{\Op}=\frac{\cos(4x)(y^2(x^2+1)+x^2-1)}{100},
		\qquad u_{-}=u\chi_{\Om}=10\cos(4x)(y^2(x^2+1)+x^2-1)+100.
	\end{align*}
	All the functions $f,g_1,g_2,g$ in \eqref{Qeques1} can be obtained by plugging the above coefficient and exact solution into \eqref{Qeques1}. In particular,
	$g_1=-100$ and $g_2=0$.
	The numerical results are presented in \cref{table:QSp5}  and \cref{fig:figure5}.	
\end{example}
\begin{table}[htbp]
	\caption{\tiny{Performance in \cref{Drafex5}  of the proposed  high order compact finite difference scheme in \cref{thm:regular,thm:gradient:regular,thm:irregular,thm:gradient:irregular} on uniform Cartesian meshes with $h=2^{-J}\times \frac{4\pi}{3}$. $\kappa$ is the condition number of the coefficient matrix.}}
	\centering
	\setlength{\tabcolsep}{2mm}{
		\begin{tabular}{c|c|c|c|c|c|c|c}
			\hline
$J$
& $\frac{\|u_{h}-u\|_{2,\ind_{\Omega}}}{\|u\|_{2,\ind_{\Omega},h}}$

&order &$\frac{|u_{h}-u|_{H^1,\ind_{\Omega}}}{|u|_{H^1,\ind_{\Omega},h}}$

&order &  $\frac{|u_{h}-u|_{V,\ind_{\Omega}}}{|u|_{V,\ind_{\Omega},h}}$

&order &  $\kappa$ \\
\hline
3   &1.5378E-01   &0   &3.8692E+00   &0   &2.0796E+00   &0   &2.2711E+04\\
4   &9.3482E-02   &0.718   &1.3359E+00   &1.534   &1.9525E-01   &3.413   &1.7331E+04\\
5   &6.1310E-03   &3.930   &1.3409E-01   &3.317   &2.1612E-02   &3.175   &1.6113E+04\\
6   &2.9209E-04   &4.392   &1.0776E-02   &3.637   &3.5076E-03   &2.623   &3.8570E+04\\
7   &1.4985E-05   &4.285   &8.2753E-04   &3.703   &3.0513E-04   &3.523   &2.9413E+04\\
8   &1.3087E-06   &3.517   &9.3813E-05   &3.141   &3.4585E-05   &3.141   &6.0381E+04\\
			\hline
	
	\end{tabular}}
	\label{table:QSp5}
\end{table}
\begin{figure}[htbp]
	\centering
	\begin{subfigure}[b]{0.3\textwidth}
		 \includegraphics[width=5.7cm,height=4cm]{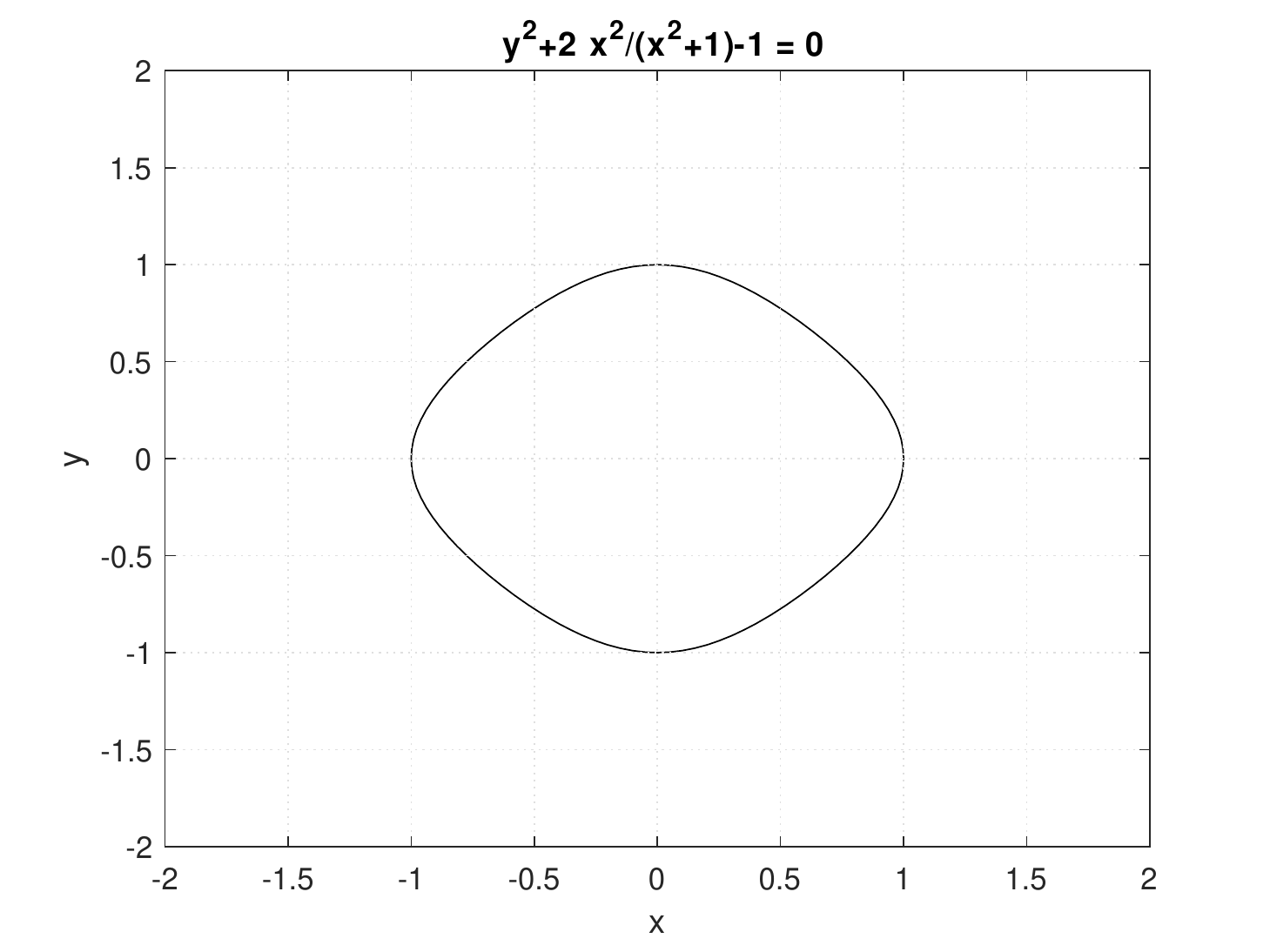}
	\end{subfigure}
	\begin{subfigure}[b]{0.3\textwidth}
		 \includegraphics[width=5.7cm,height=4.5cm]{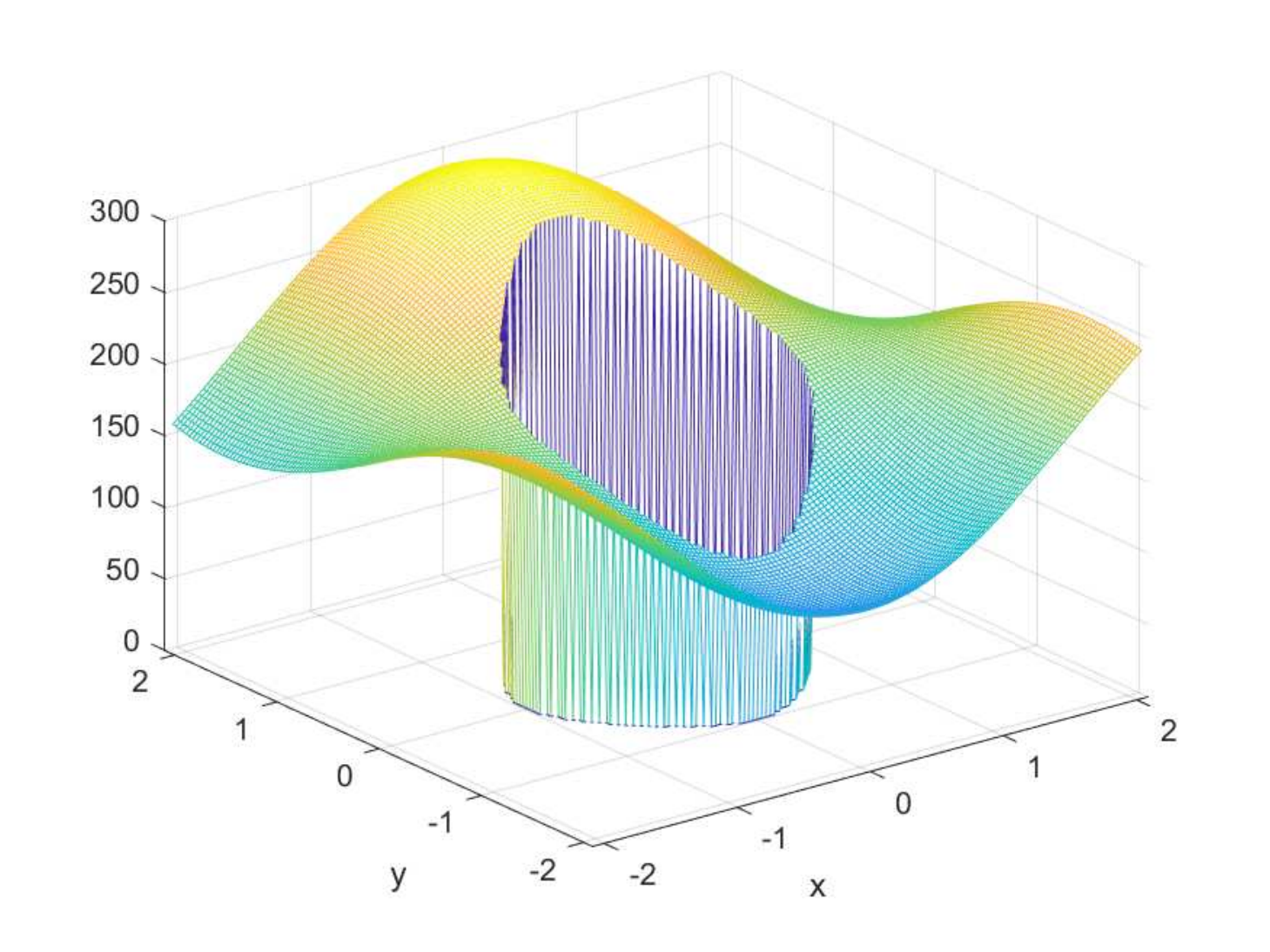}
	\end{subfigure}
	\begin{subfigure}[b]{0.3\textwidth}
		 \includegraphics[width=5.7cm,height=4.5cm]{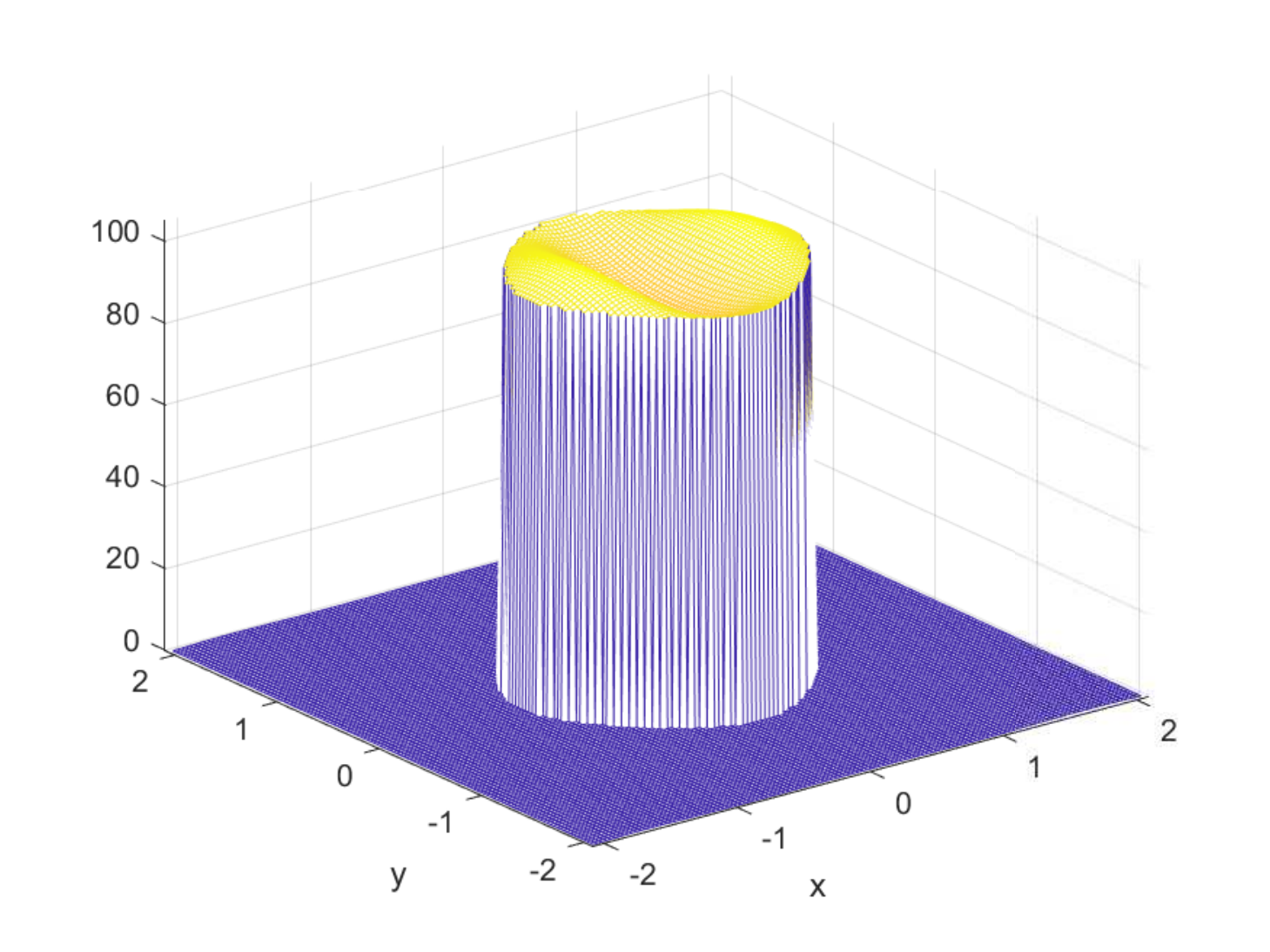}
	\end{subfigure}
	\begin{subfigure}[b]{0.3\textwidth}
		 \includegraphics[width=5.7cm,height=4.5cm]{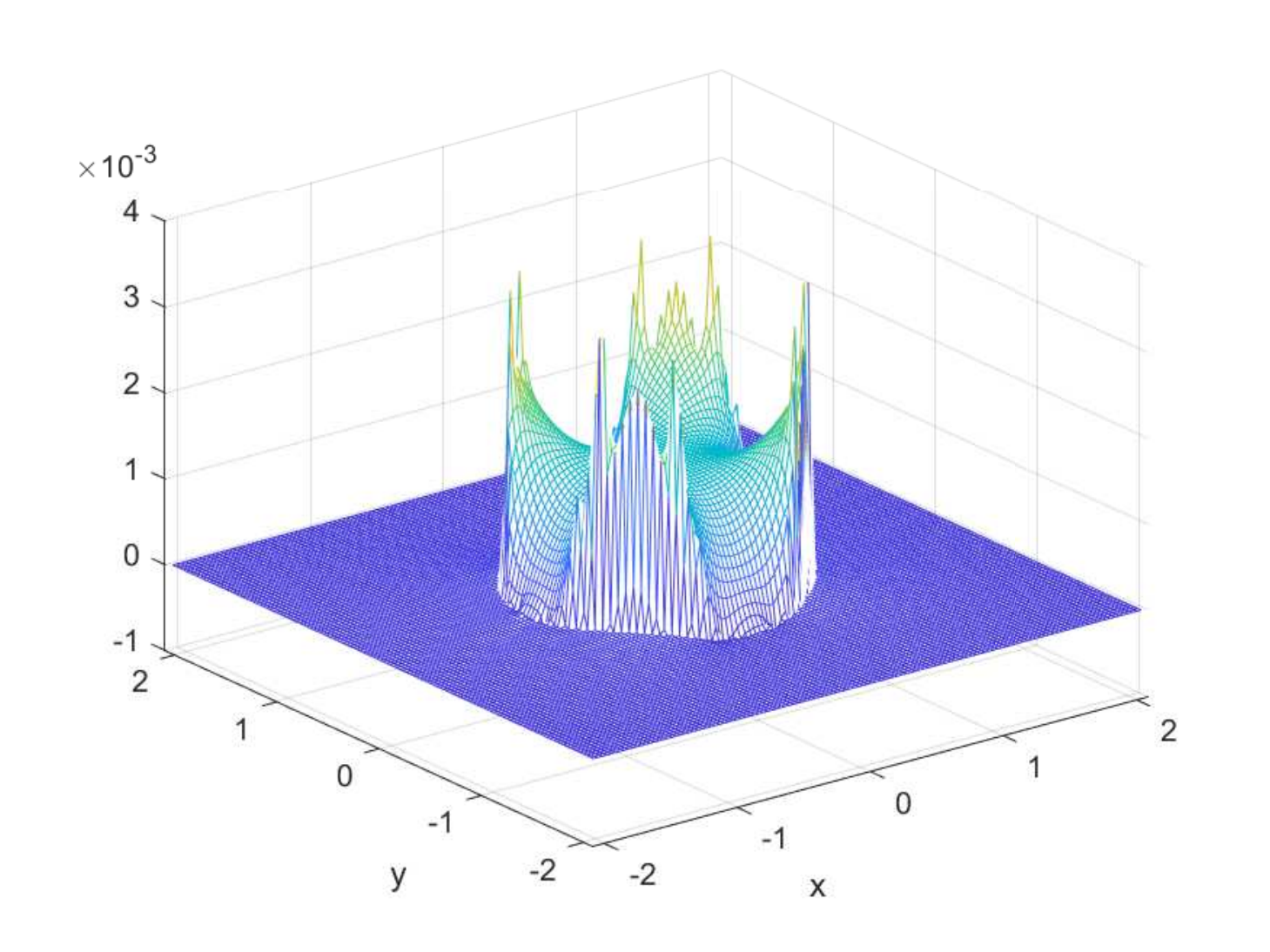}
	\end{subfigure}
	\begin{subfigure}[b]{0.3\textwidth}
		 \includegraphics[width=5.7cm,height=4.5cm]{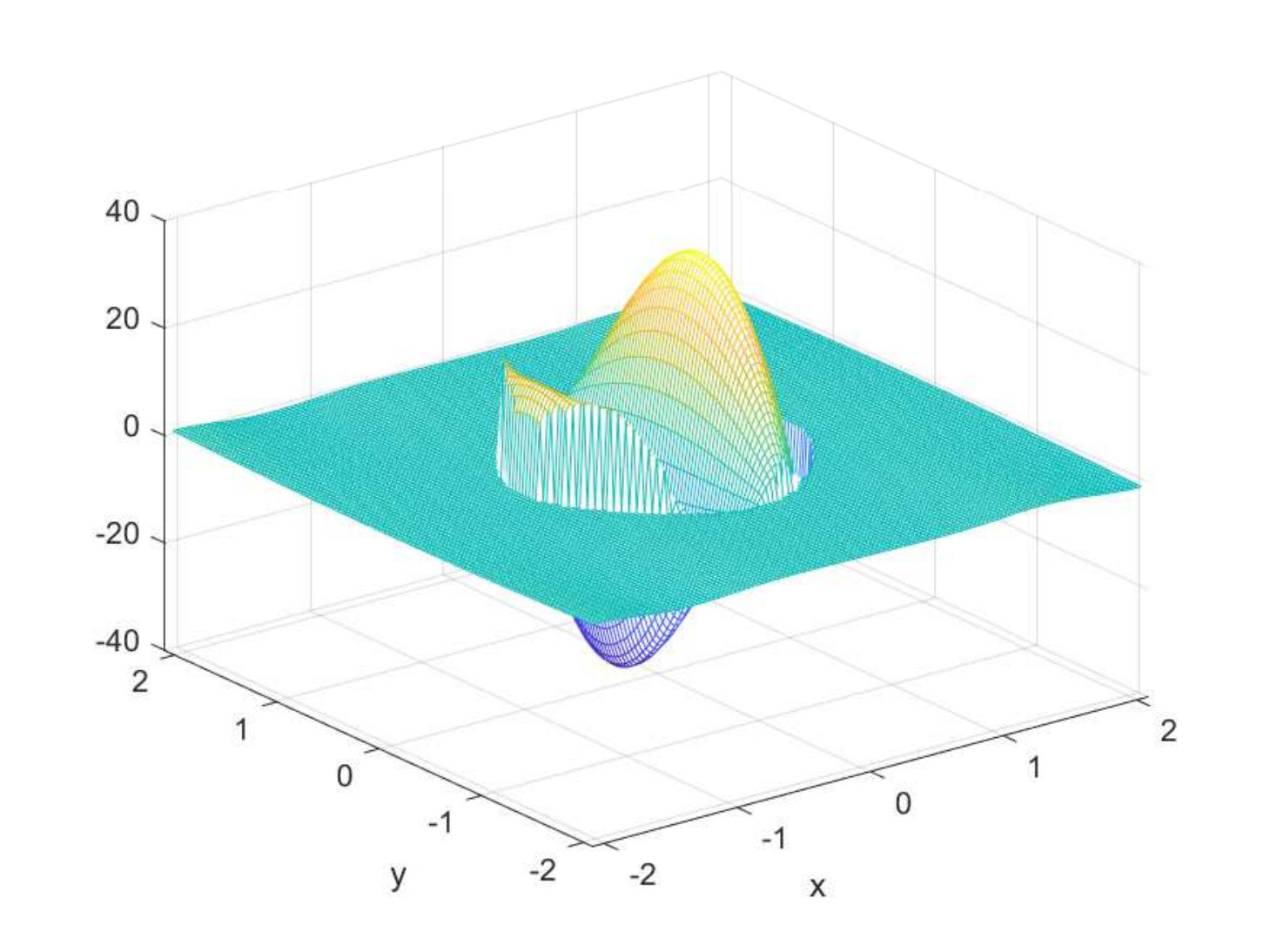}
	\end{subfigure}
	\begin{subfigure}[b]{0.3\textwidth}
		 \includegraphics[width=5.7cm,height=4.5cm]{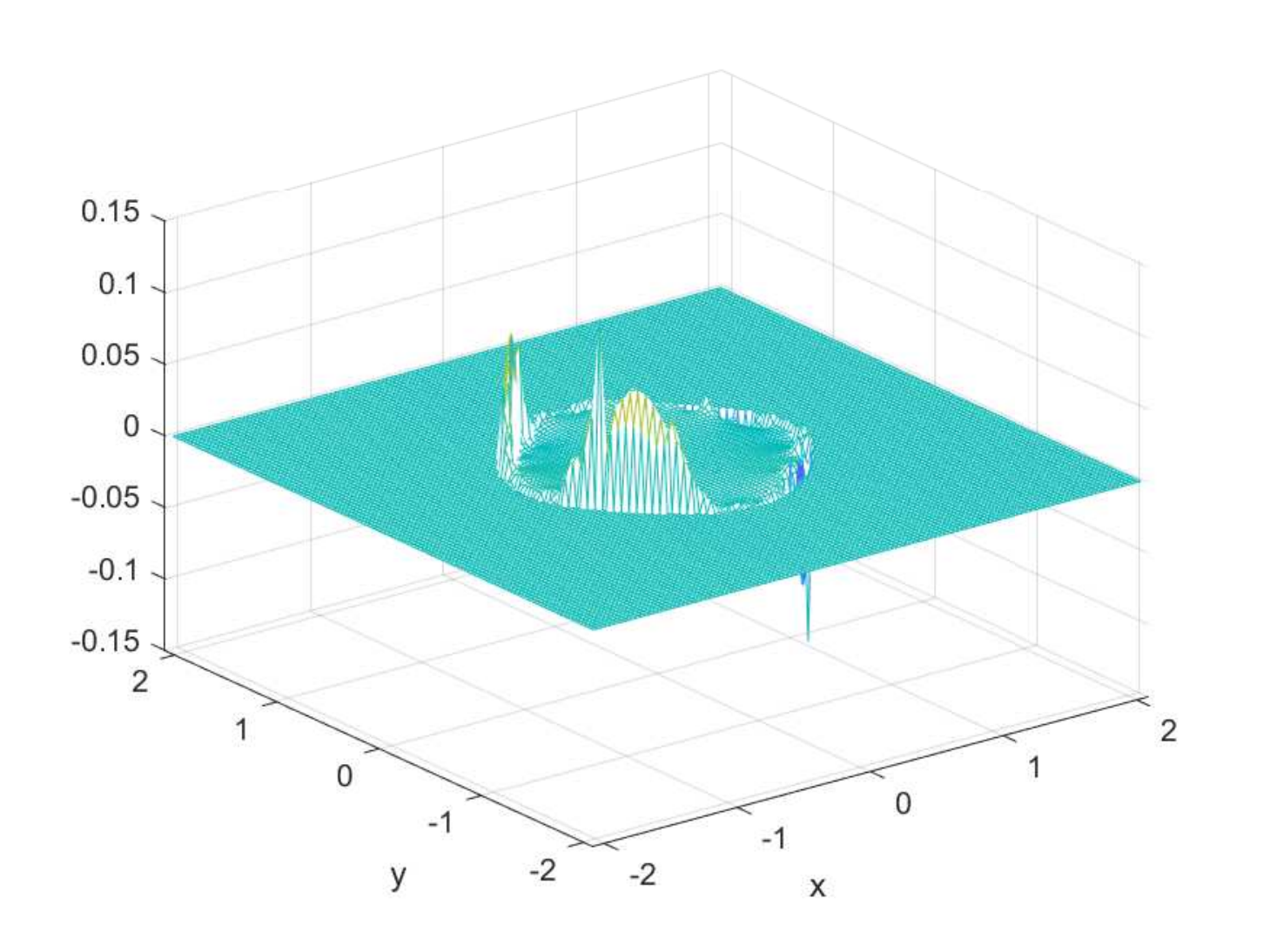}
	\end{subfigure}
	\caption
	{\tiny{Top row for \cref{Drafex5}: the interface curve $\Gamma$ (left), the coefficient $a(x,y)$ (middle)	and
		the numerical solution $u_h$ (right) with $h=2^{-7}\times \frac{4\pi}{3}$. Bottom row for \cref{Drafex5}: the error $u_h-u$ (left), the numerical $(u_h)_x$ (middle) and the error $(u_h)_x-u_x$ (right) with $h=2^{-7}\times \frac{4\pi}{3}$.}}
	\label{fig:figure5}
\end{figure}
\subsection{Numerical examples with $u$ unknown and $\Gamma \cap \partial \Omega=\emptyset$\label{NumeSect2}}
In this subsection, we provide 5 numerical experiments such that the exact solution $u$ of \eqref{Qeques1} is unknown.

\begin{example}\label{Drafex7}
	\normalfont
	Let $\Omega=(-\pi,\pi)^2$ and
	the interface curve be given by
	$\Gamma:=\{(x,y)\in \Omega \; :\; \psi(x,y)=0\}$ with
	$\psi (x,y)=x^4+2y^4-2$. Note that $\Gamma \cap \partial \Omega=\emptyset$ and
    \eqref{Qeques1} is given by
	\begin{align*}
		 &a_{+}=a\chi_{\Op}=100(2+\sin(x)\cos(y)),
		\qquad a_{-}=a\chi_{\Om}=\frac{2+\sin(x)\cos(y)}{10},\\
		 &f_{+}=f\chi_{\Op}=\sin(2x)\sin(2y),
		\qquad f_{-}=f\chi_{\Om}=\cos(2x)\cos(2y), \\
		&g_1=\exp(x-y)-10, \qquad g_2=\cos(x+y),
		\qquad g=0.
	\end{align*}
	The numerical results are provided in \cref{table:QSp7}  and \cref{fig:figure7}.		
\end{example}
\begin{table}[htbp]
	\caption{\tiny{Performance in \cref{Drafex7}  of the proposed  high order compact finite difference scheme in \cref{thm:regular,thm:gradient:regular,thm:irregular,thm:gradient:irregular} on uniform Cartesian meshes with $h=2^{-J}\times 2\pi$. $\kappa$ is the condition number of the coefficient matrix.}}
	\centering
	\setlength{\tabcolsep}{2mm}{
		\begin{tabular}{c|c|c|c|c|c|c|c}
			\hline
			$J$
& $\|u_{h}-u_{h/2}\|_{2,\ind_{\Omega}}$

&order &$|u_{h}-u_{h/2}|_{H^1,\ind_{\Omega}}$

&order &  $|u_{h}-u_{h/2}|_{V,\ind_{\Omega}}$

&order &  $\kappa$ \\
\hline
3    &2.8284E-01    &0    &1.4650E+01    &0    &1.2366E+01    &0    &7.8789E+02\\
4    &5.3709E-02    &2.397    &3.6196E-01    &5.339    &2.6415E+00    &2.227    &2.7310E+03\\
5    &6.6858E-03    &3.006    &7.6924E-02    &2.234    &8.3040E-01    &1.669    &8.7219E+03\\
6    &3.9281E-04    &4.089    &5.4377E-03    &3.822    &6.2964E-02    &3.721    &4.5222E+04\\
7    &2.0733E-05    &4.244    &3.2159E-04    &4.080    &5.2101E-03    &3.595    &2.2395E+04\\
\hline
			
	\end{tabular}}
	\label{table:QSp7}
\end{table}
\begin{figure}[htbp]
	\centering
	\begin{subfigure}[b]{0.3\textwidth}
		 \includegraphics[width=5.7cm,height=4.cm]{AA2.pdf}
	\end{subfigure}
	\begin{subfigure}[b]{0.3\textwidth}
		 \includegraphics[width=5.7cm,height=4.5cm]{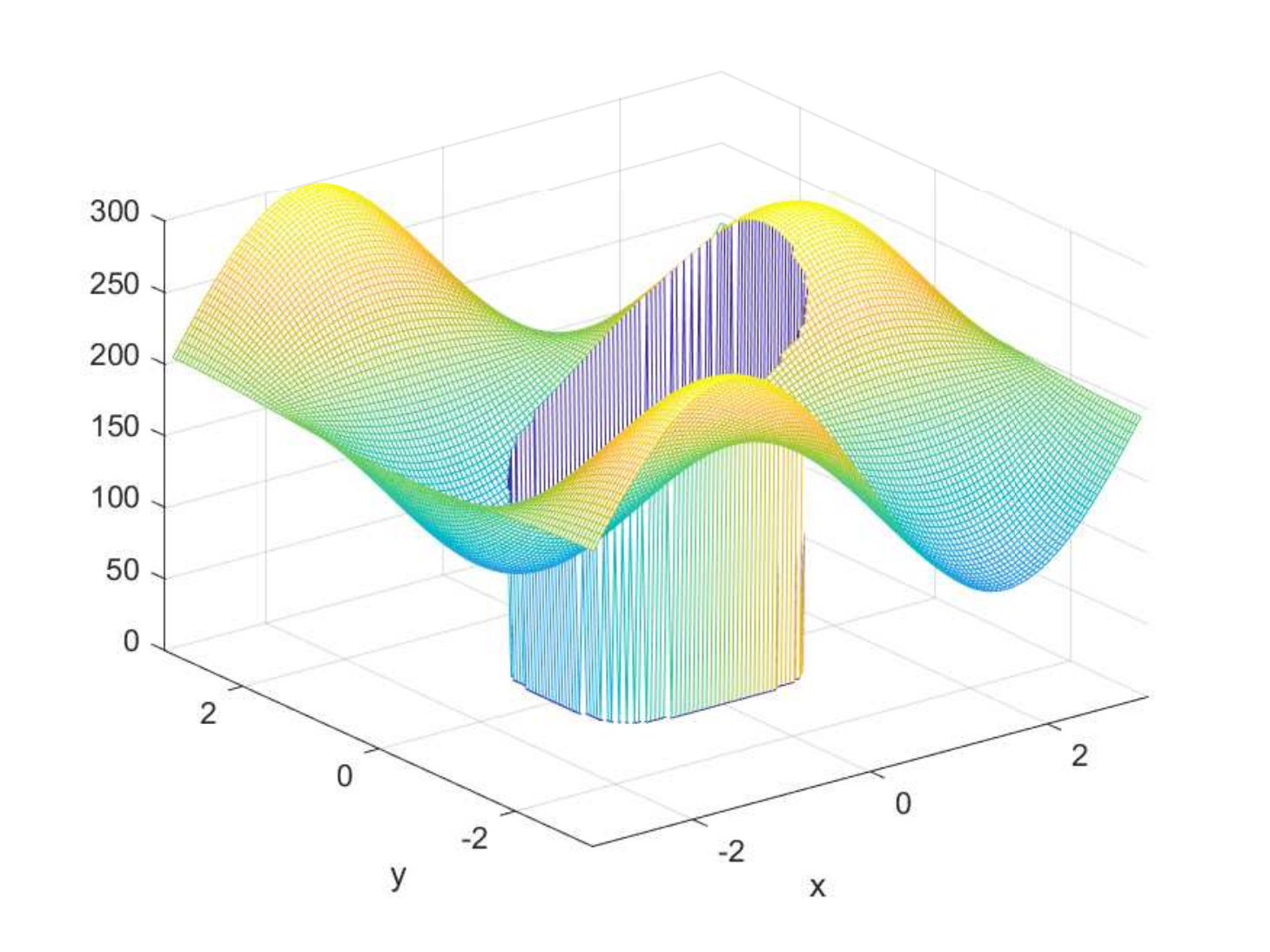}
	\end{subfigure}
	\begin{subfigure}[b]{0.3\textwidth}
		 \includegraphics[width=5.7cm,height=4.5cm]{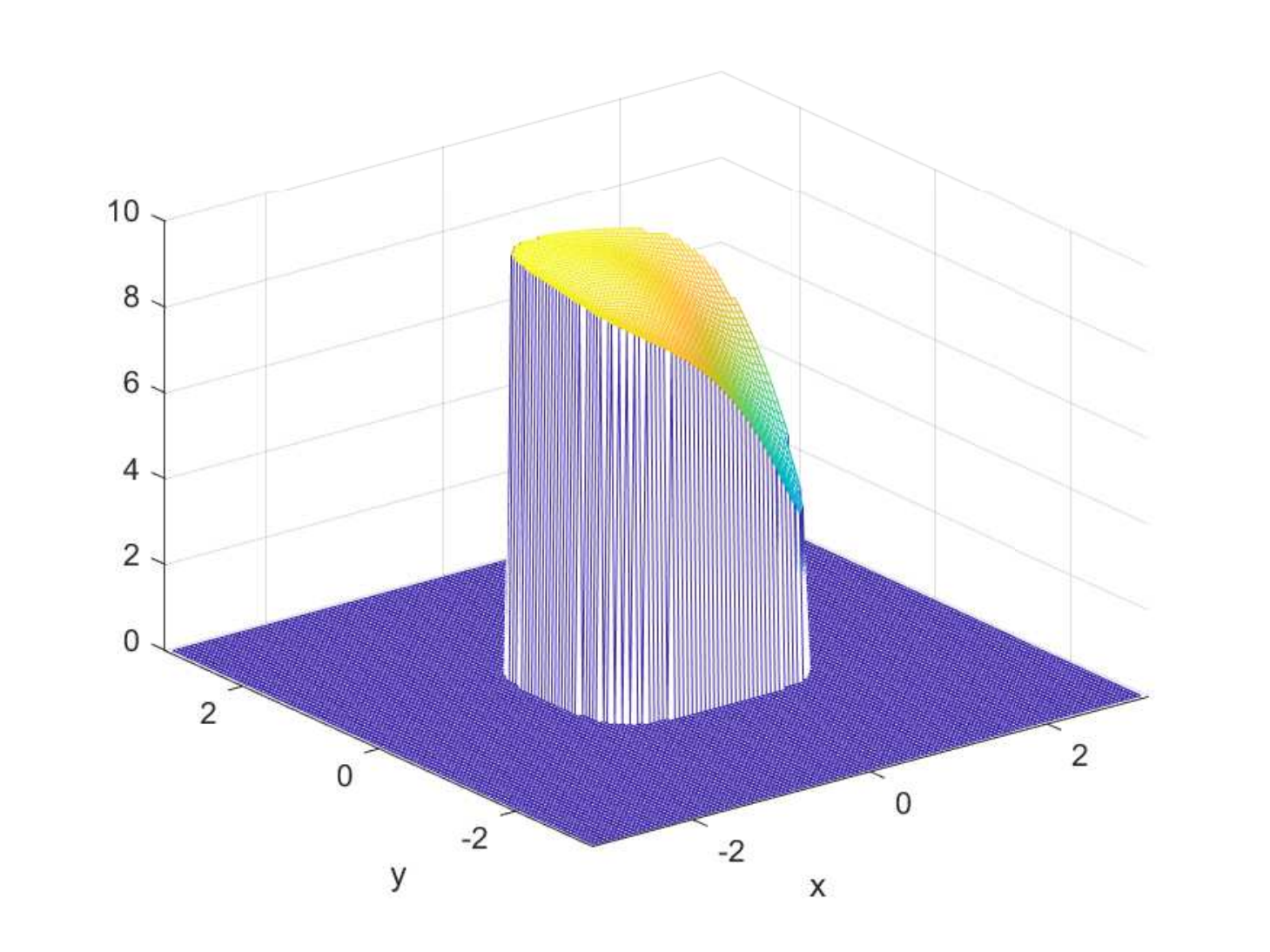}
	\end{subfigure}
	\begin{subfigure}[b]{0.3\textwidth}
		 \includegraphics[width=5.7cm,height=4.5cm]{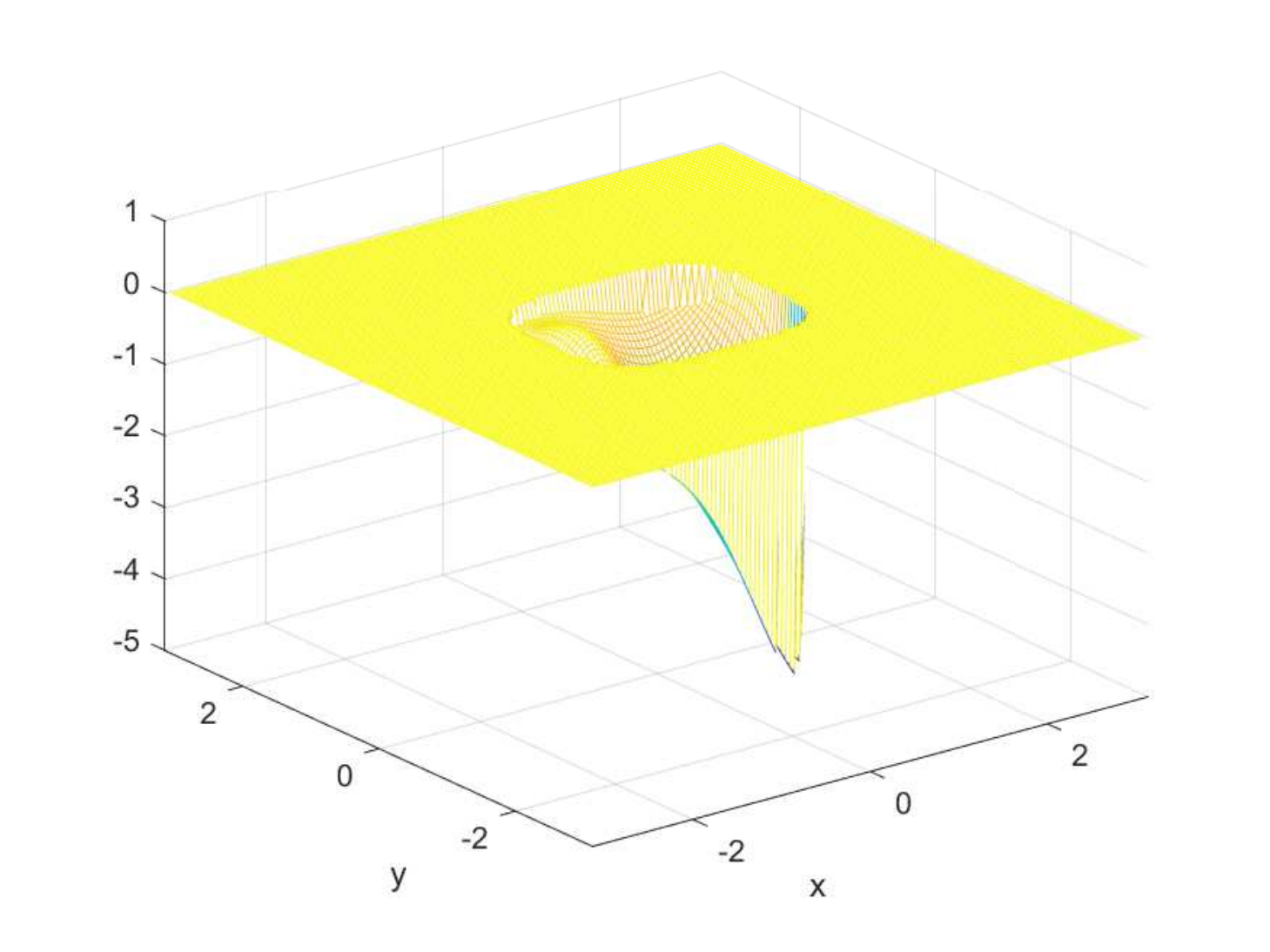}
	\end{subfigure}
	\begin{subfigure}[b]{0.3\textwidth}
	 \includegraphics[width=5.7cm,height=4.5cm]{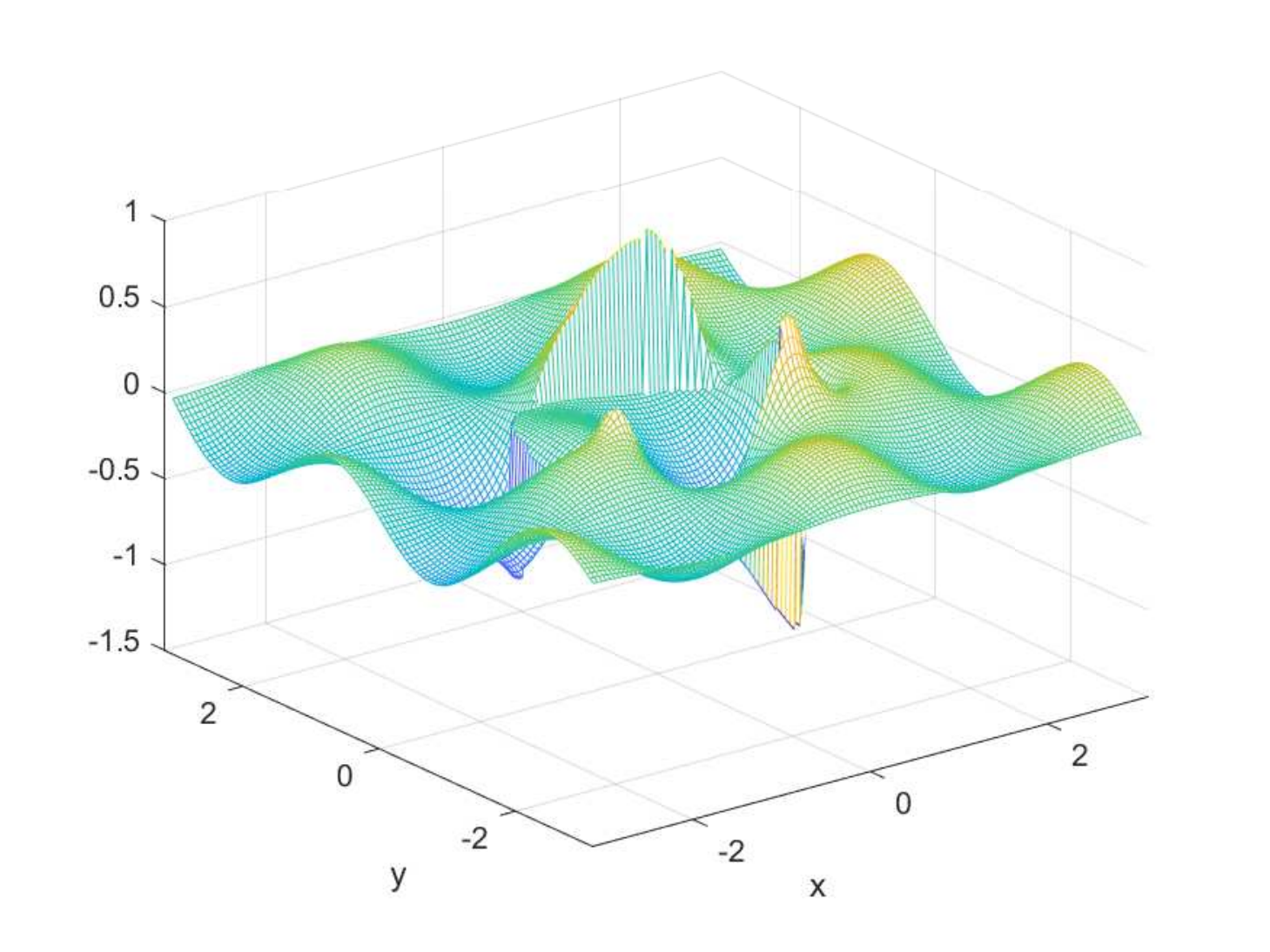}
\end{subfigure}
\begin{subfigure}[b]{0.3\textwidth}
	 \includegraphics[width=5.7cm,height=4.5cm]{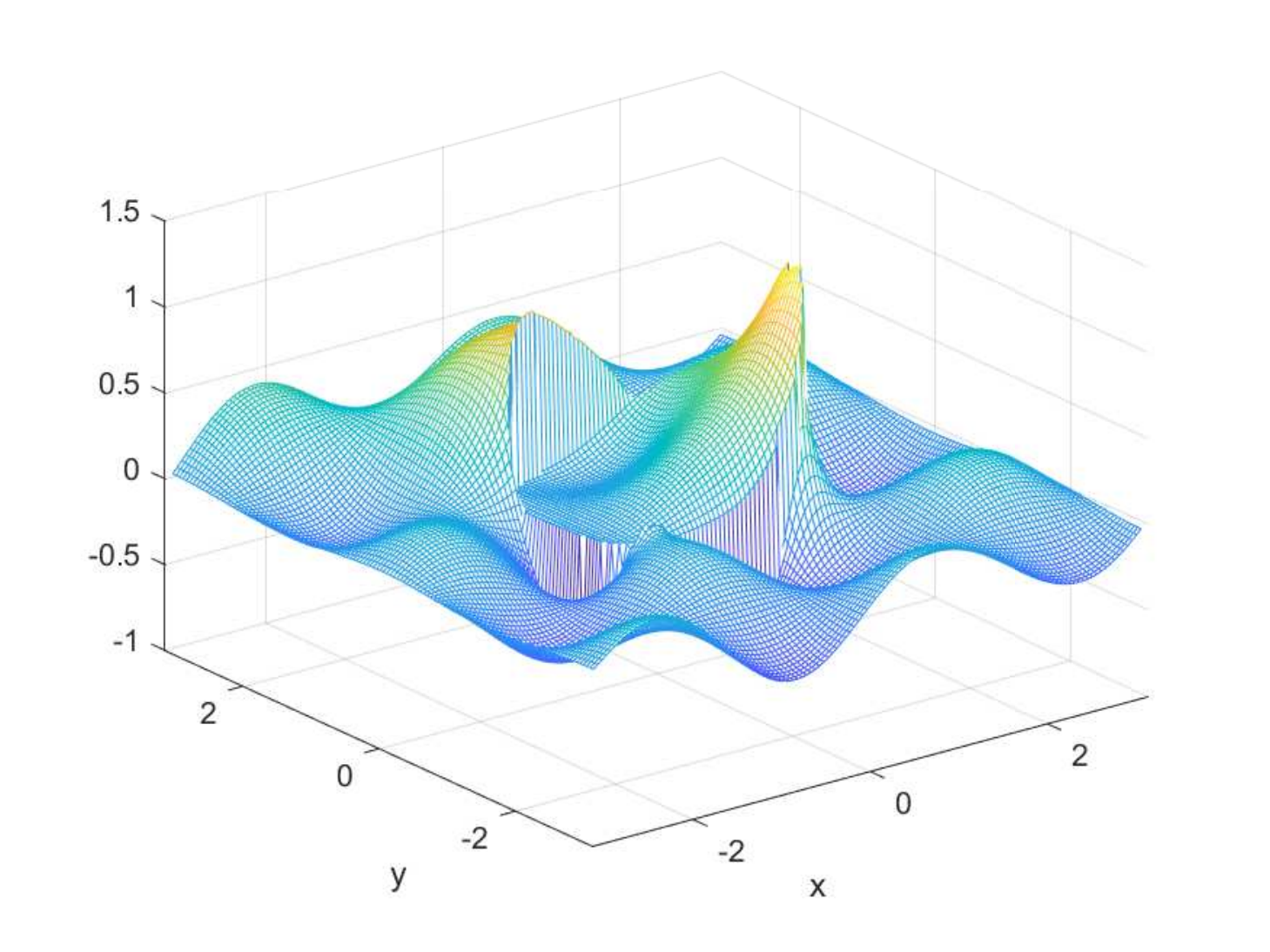}
\end{subfigure}
	\caption
	{\tiny{Top row for \cref{Drafex7}: the interface curve $\Gamma$ (left), the coefficient $a(x,y)$ (middle) and the numerical solution $u_h$ (right) with $h=2^{-7}\times 2\pi$. Bottom row for \cref{Drafex7}:  the numerical  $(u_h)_x$ (left), the numerical  $a(x,y)\times (u_h)_x$ (middle) and  the numerical  $a(x,y)\times(u_h)_y$ (right) with $h=2^{-7}\times 2\pi$.}}
	\label{fig:figure7}
\end{figure}
\begin{example}\label{Drafex8}
	\normalfont
	Let $\Omega=(-\pi,\pi)^2$ and
	the interface curve be given by
	$\Gamma:=\{(x,y)\in \Omega \; :\; \psi(x,y)=0\}$ with
	$\psi (x,y)=y^2-2x^2+x^4-1$. Note that $\Gamma \cap \partial \Omega=\emptyset$ and
    \eqref{Qeques1} is given by
	\begin{align*}
		 &a_{+}=a\chi_{\Op}=\frac{10}{2+\cos(x+y)},
		\qquad a_{-}=a\chi_{\Om}=\frac{2+\sin(x+y)}{100},\\
		&f_{+}=f\chi_{\Op}=\sin(2x)\sin(y),
		\qquad f_{-}=f\chi_{\Om}=\exp(x+y)\sin(x), \\
		&g_1=\cos(x-y)-1, \qquad g_2=\sin(x-y),
		\qquad g=0.
	\end{align*}
	The numerical results are provided in \cref{table:QSp8}  and \cref{fig:figure8}.		
\end{example}
\begin{table}[htbp]
	\caption{\tiny{Performance in \cref{Drafex8}  of the proposed  high order compact finite difference scheme in \cref{thm:regular,thm:gradient:regular,thm:irregular,thm:gradient:irregular} on uniform Cartesian meshes with $h=2^{-J}\times 2\pi$. $\kappa$ is the condition number of the coefficient matrix.}}
	\centering
	\setlength{\tabcolsep}{2mm}{
		\begin{tabular}{c|c|c|c|c|c|c|c}
			\hline
			$J$
& $\|u_{h}-u_{h/2}\|_{2,\ind_{\Omega}}$

&order &$|u_{h}-u_{h/2}|_{H^1,\ind_{\Omega}}$

&order &  $|u_{h}-u_{h/2}|_{V,\ind_{\Omega}}$

&order &  $\kappa$ \\
\hline
3    &2.4355E+03    &0    &5.4333E+03    &0    &1.9856E+03    &0    &8.2409E+05\\
4    &1.0396E+01    &7.872    &3.3857E+02    &4.004    &9.7875E+00    &7.664    &4.3788E+03\\
5    &5.4539E-01    &4.253    &1.1565E+01    &4.872    &6.0239E-01    &4.022    &6.5855E+03\\
6    &3.3619E-02    &4.020    &2.7956E-01    &5.370    &7.6827E-02    &2.971    &4.3061E+03\\
7    &2.1505E-03    &3.967    &2.1704E-02    &3.687    &5.4389E-03    &3.820    &1.3806E+04\\
\hline
			
	\end{tabular}}
	\label{table:QSp8}
\end{table}
\begin{figure}[htbp]
	\centering
	\begin{subfigure}[b]{0.3\textwidth}
		 \includegraphics[width=5.7cm,height=4.cm]{AA3.pdf}
	\end{subfigure}
	\begin{subfigure}[b]{0.3\textwidth}
		 \includegraphics[width=5.7cm,height=4.5cm]{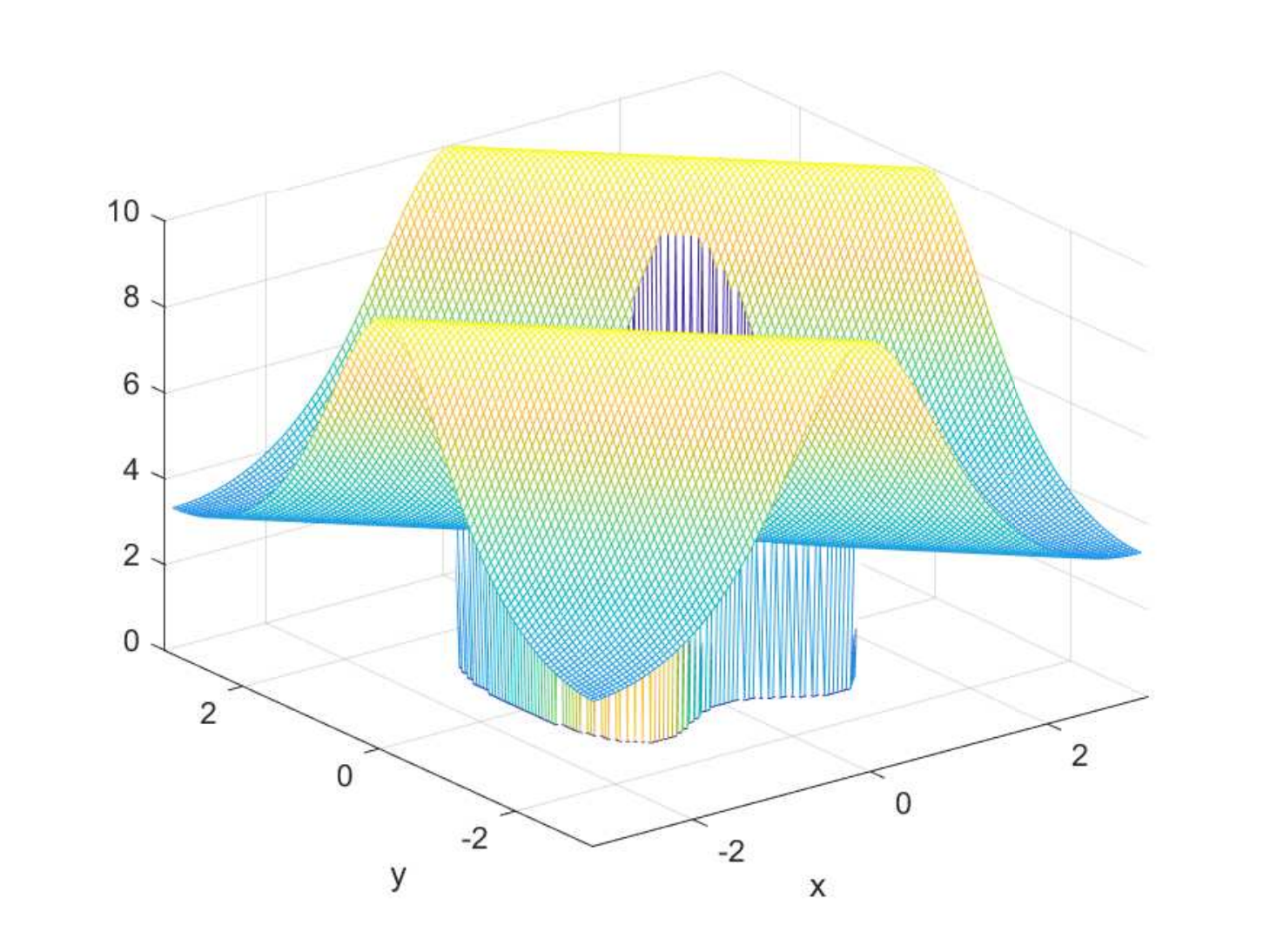}
	\end{subfigure}
	\begin{subfigure}[b]{0.3\textwidth}
		 \includegraphics[width=5.7cm,height=4.5cm]{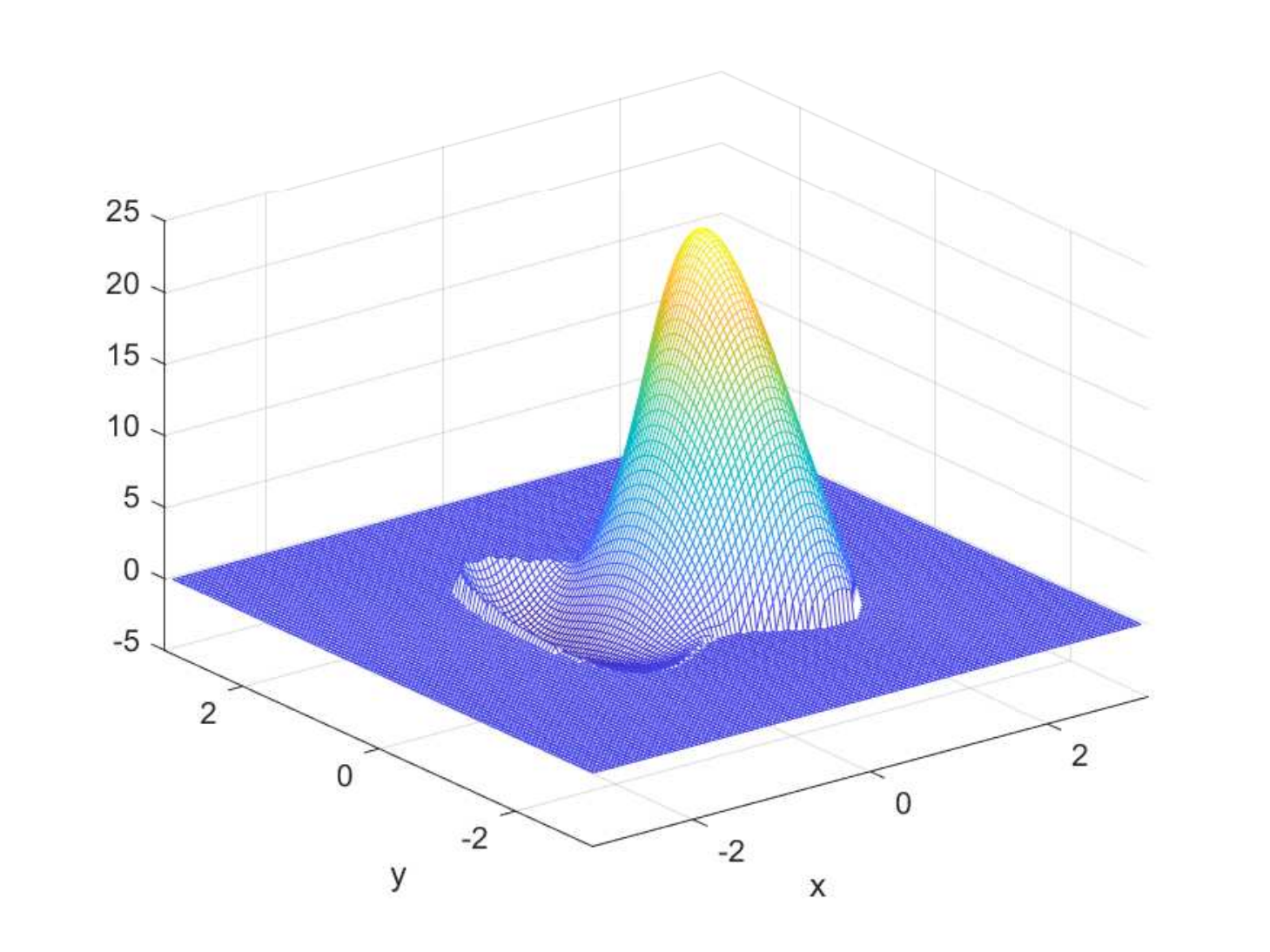}
	\end{subfigure}
	\begin{subfigure}[b]{0.3\textwidth}
		 \includegraphics[width=5.7cm,height=4.5cm]{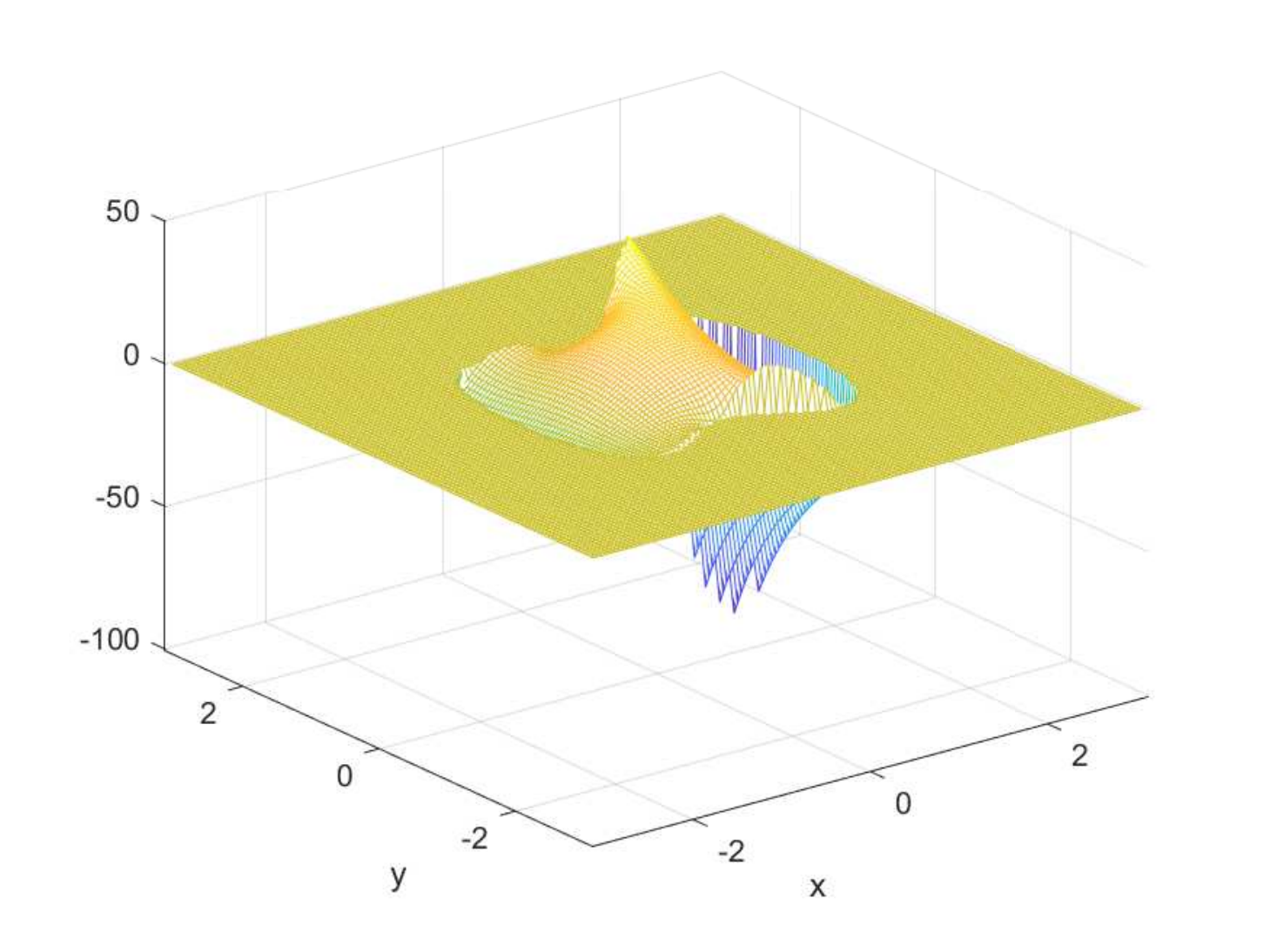}
	\end{subfigure}
	\begin{subfigure}[b]{0.3\textwidth}
	 \includegraphics[width=5.7cm,height=4.5cm]{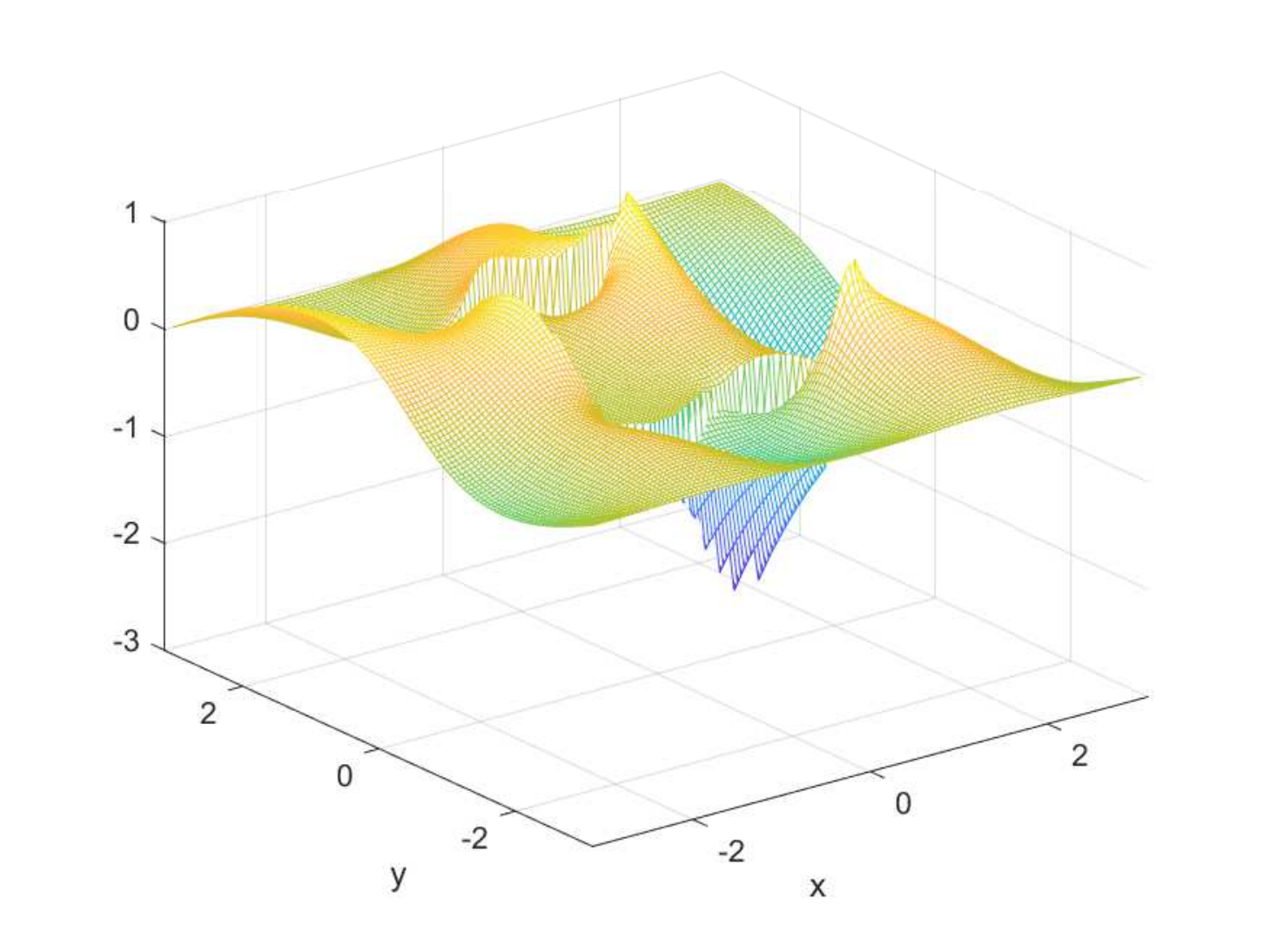}
\end{subfigure}
\begin{subfigure}[b]{0.3\textwidth}
	 \includegraphics[width=5.7cm,height=4.5cm]{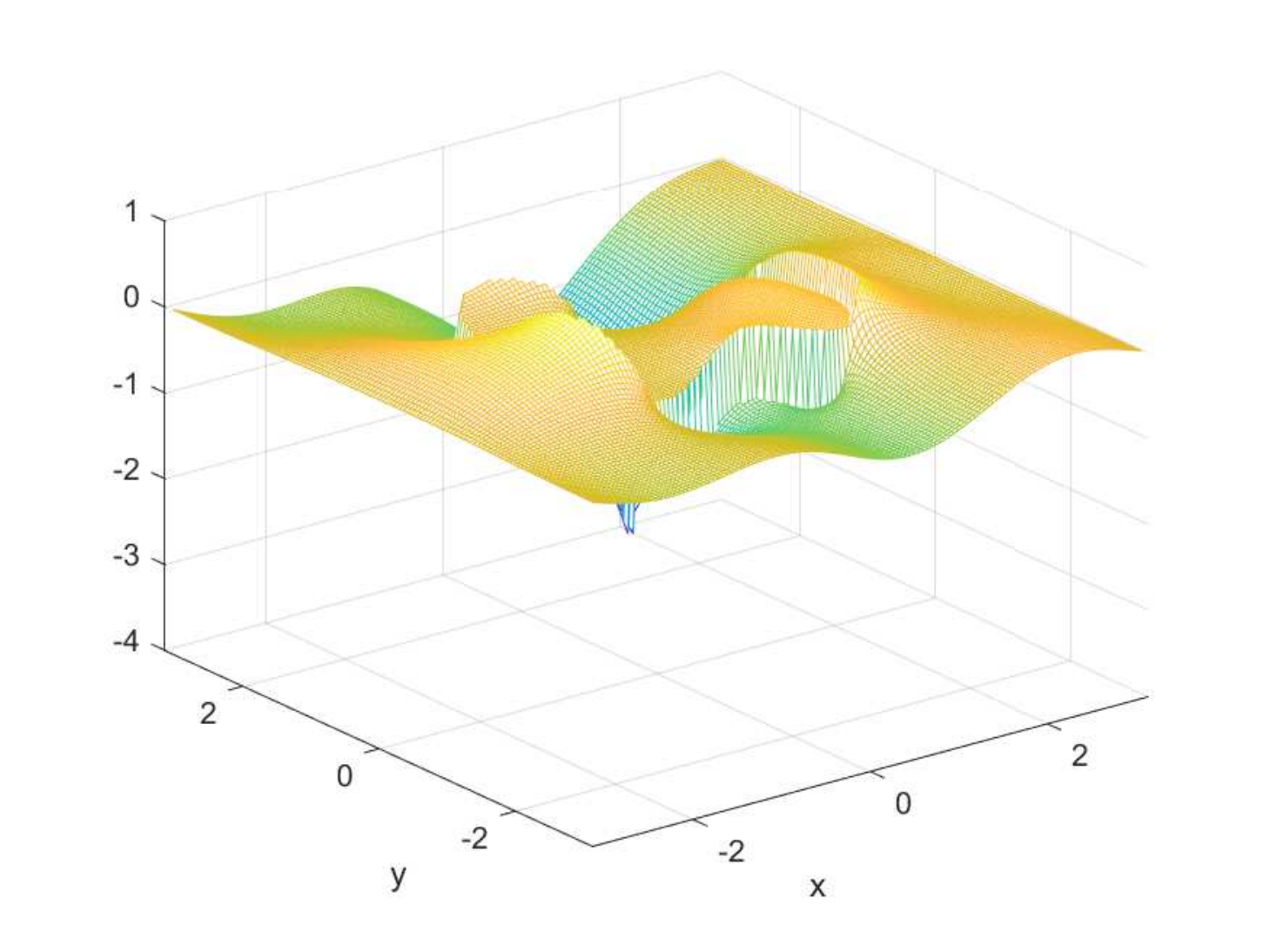}
\end{subfigure}
	\caption
	{\tiny{Top row for \cref{Drafex8}: the interface curve $\Gamma$ (left), the coefficient $a(x,y)$ (middle) and the numerical solution $u_h$ (right) with $h=2^{-7}\times 2\pi$. Bottom row for \cref{Drafex8}:  the numerical  $(u_h)_x$ (left), the numerical  $a(x,y)\times (u_h)_x$ (middle) and  the numerical  $a(x,y)\times(u_h)_y$ (right) with $h=2^{-7}\times 2\pi$.}}
	\label{fig:figure8}
\end{figure}
\begin{example}\label{Drafex9}
	\normalfont
	Let $\Omega=(-2,2)^2$ and
	the interface curve be given by
	$\Gamma:=\{(x,y)\in \Omega \; :\; \psi(x,y)=0\}$ with
	$\psi (x,y)=2x^4+y^2-1/2$. Note that $\Gamma \cap \partial \Omega=\emptyset$ and
    \eqref{Qeques1} is given by
	\begin{align*}
		 &a_{+}=a\chi_{\Op}=100(2+\sin(x)\cos(y)),
		\qquad a_{-}=a\chi_{\Om}=\frac{2+\cos(x-y)}{10},\\
		&f_{+}=f\chi_{\Op}=\sin(\pi x)\sin(\pi y),
		\qquad f_{-}=f\chi_{\Om}=\cos(\pi x)\cos(\pi y), \\
		&g_1=\sin(x)\cos(y)-2, \qquad g_2=\cos(x)\sin(y),
		\qquad g=0.
	\end{align*}
	The numerical results are provided in \cref{table:QSp9}  and \cref{fig:figure9}.		
\end{example}
\begin{table}[htbp]
	\caption{\tiny{Performance in \cref{Drafex9}  of the proposed  high order compact finite difference scheme in \cref{thm:regular,thm:gradient:regular,thm:irregular,thm:gradient:irregular} on uniform Cartesian meshes with $h=2^{-J}\times 4$. $\kappa$ is the condition number of the coefficient matrix.}}
	\centering
	\setlength{\tabcolsep}{2mm}{
		\begin{tabular}{c|c|c|c|c|c|c|c}
			\hline
			$J$
& $\|u_{h}-u_{h/2}\|_{2,\ind_{\Omega}}$

&order &$|u_{h}-u_{h/2}|_{H^1,\ind_{\Omega}}$

&order &  $|u_{h}-u_{h/2}|_{V,\ind_{\Omega}}$

&order &  $\kappa$ \\
\hline

3    &1.9911E-01    &0    &5.6734E+00    &0    &4.7783E+00    &0    &6.3170E+02\\
4    &1.7085E-02    &3.543    &1.9762E-01    &4.843    &2.3331E+00    &1.034    &1.0358E+04\\
5    &1.7892E-03    &3.255    &9.9839E-03    &4.307    &9.2170E-02    &4.662    &2.6379E+03\\
6    &1.0570E-04    &4.081    &9.6792E-04    &3.367    &8.5957E-03    &3.423    &4.3995E+03\\
7    &6.3279E-06    &4.062    &8.2529E-05    &3.552    &6.2430E-04    &3.783    &1.4608E+04\\
\hline
			
	\end{tabular}}
	\label{table:QSp9}
\end{table}
\begin{figure}[htbp]
	\centering
	\begin{subfigure}[b]{0.3\textwidth}
		 \includegraphics[width=5.7cm,height=4.cm]{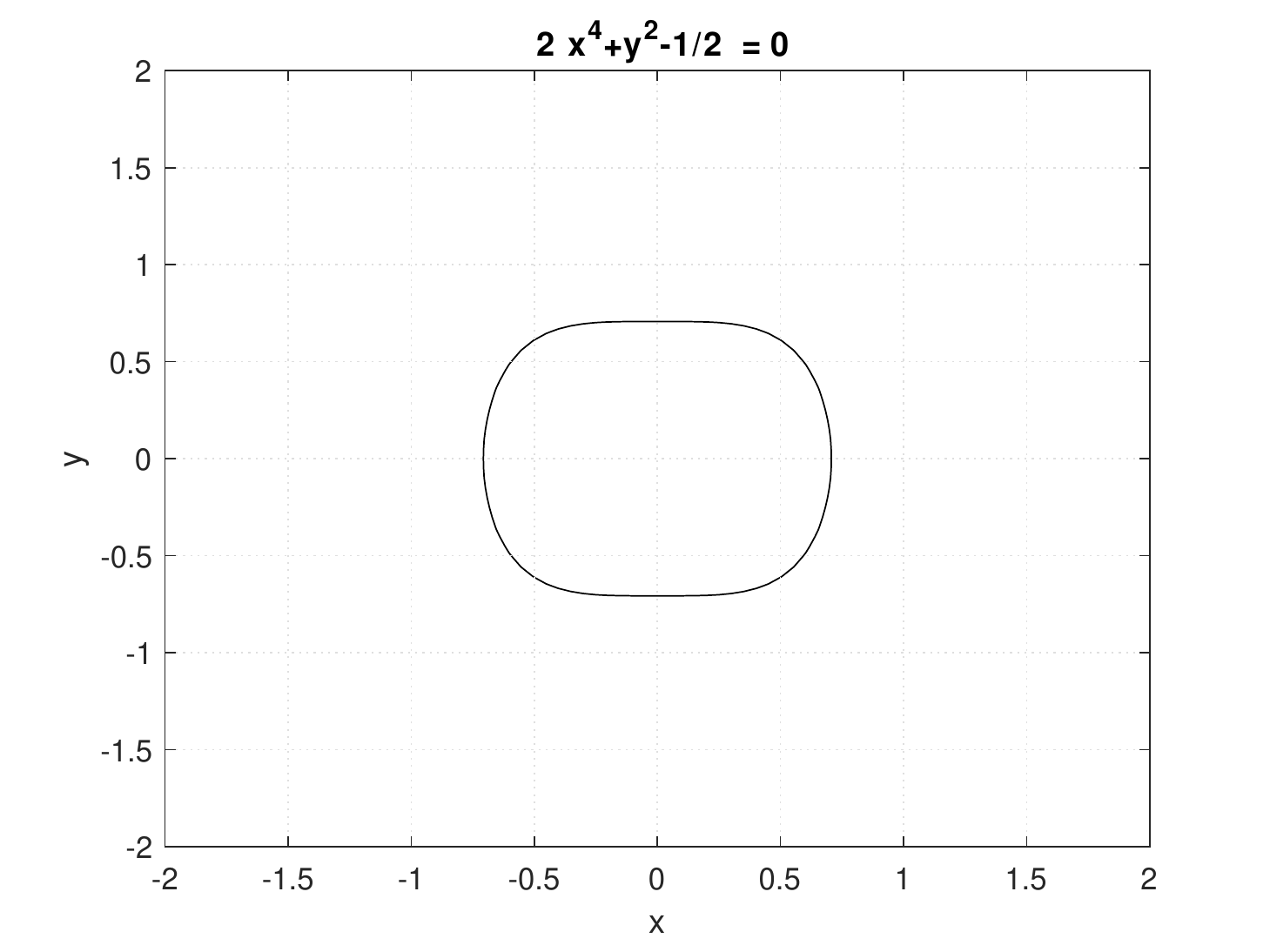}
	\end{subfigure}
	\begin{subfigure}[b]{0.3\textwidth}
		 \includegraphics[width=5.7cm,height=4.5cm]{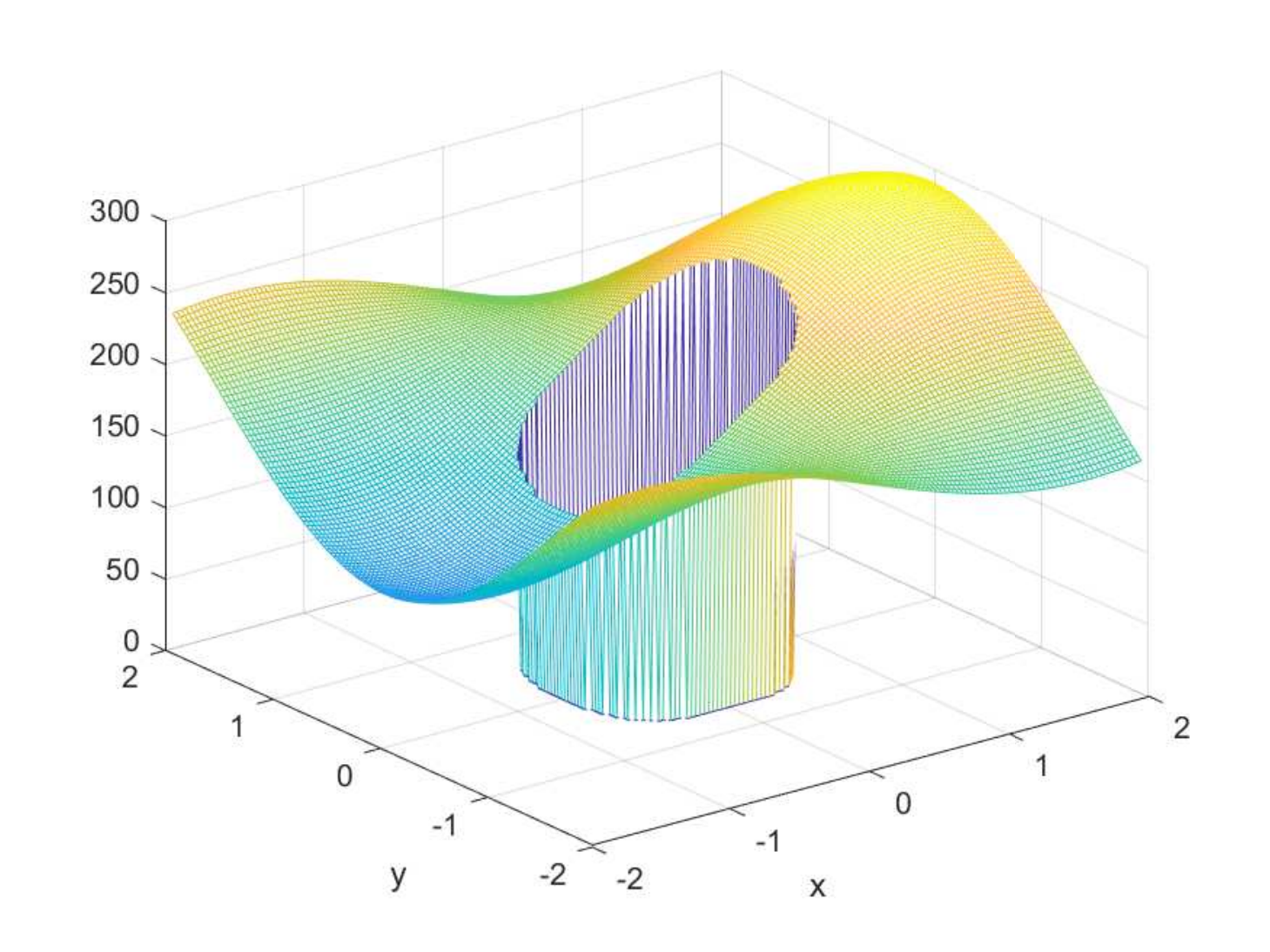}
	\end{subfigure}
	\begin{subfigure}[b]{0.3\textwidth}
		 \includegraphics[width=5.7cm,height=4.5cm]{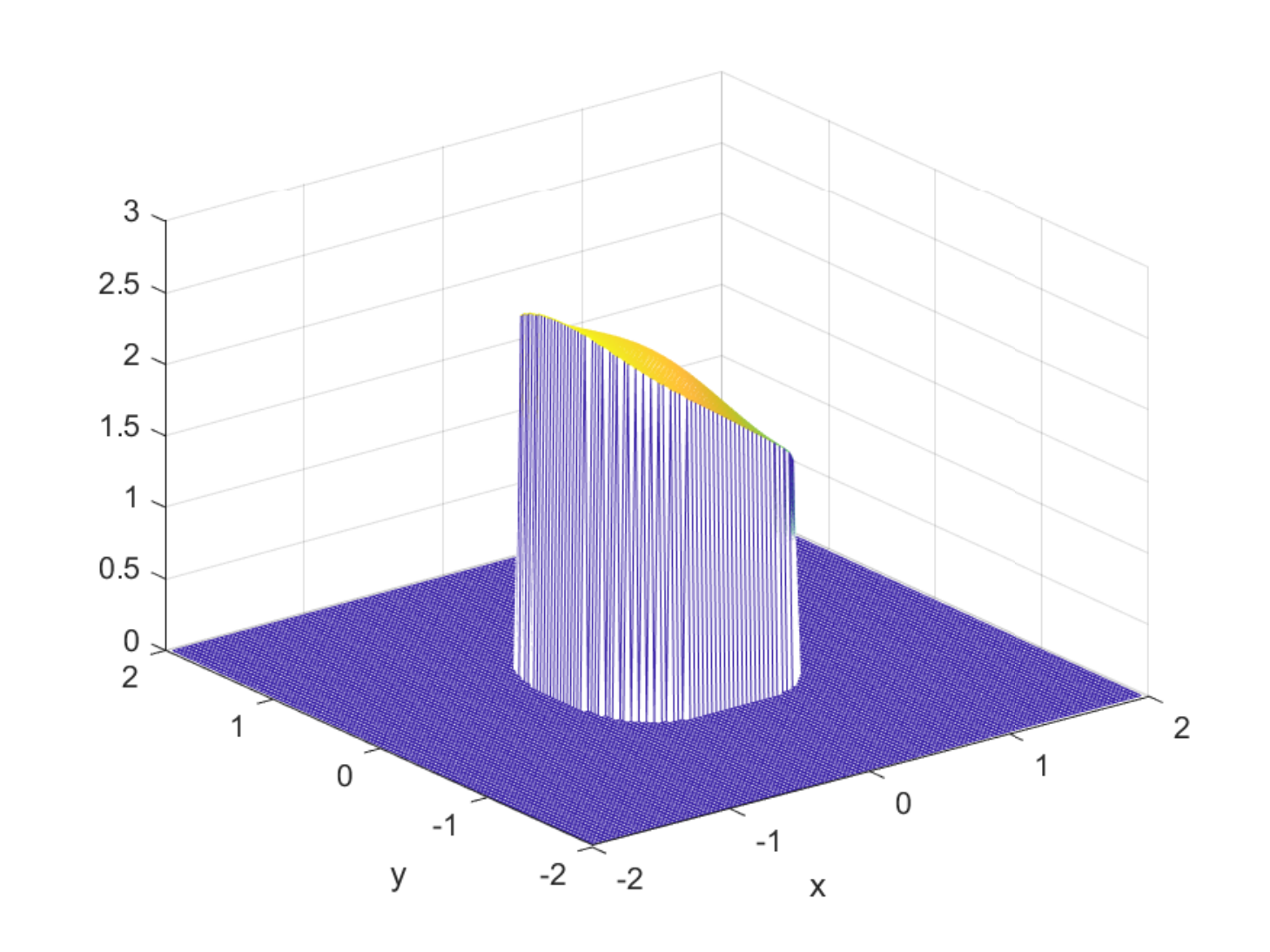}
	\end{subfigure}
	\begin{subfigure}[b]{0.3\textwidth}
		 \includegraphics[width=5.7cm,height=4.5cm]{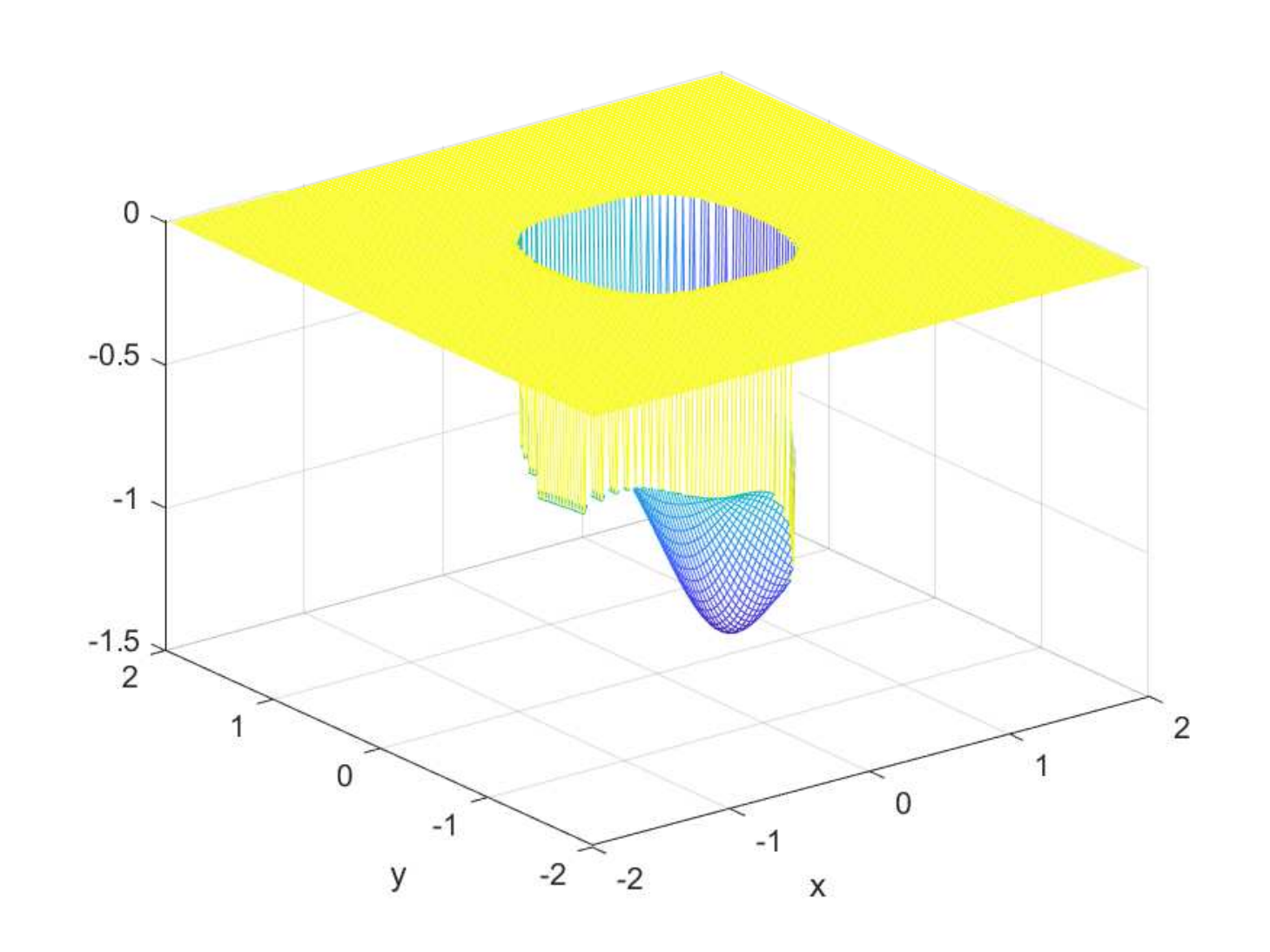}
	\end{subfigure}
	\begin{subfigure}[b]{0.3\textwidth}
	 \includegraphics[width=5.7cm,height=4.5cm]{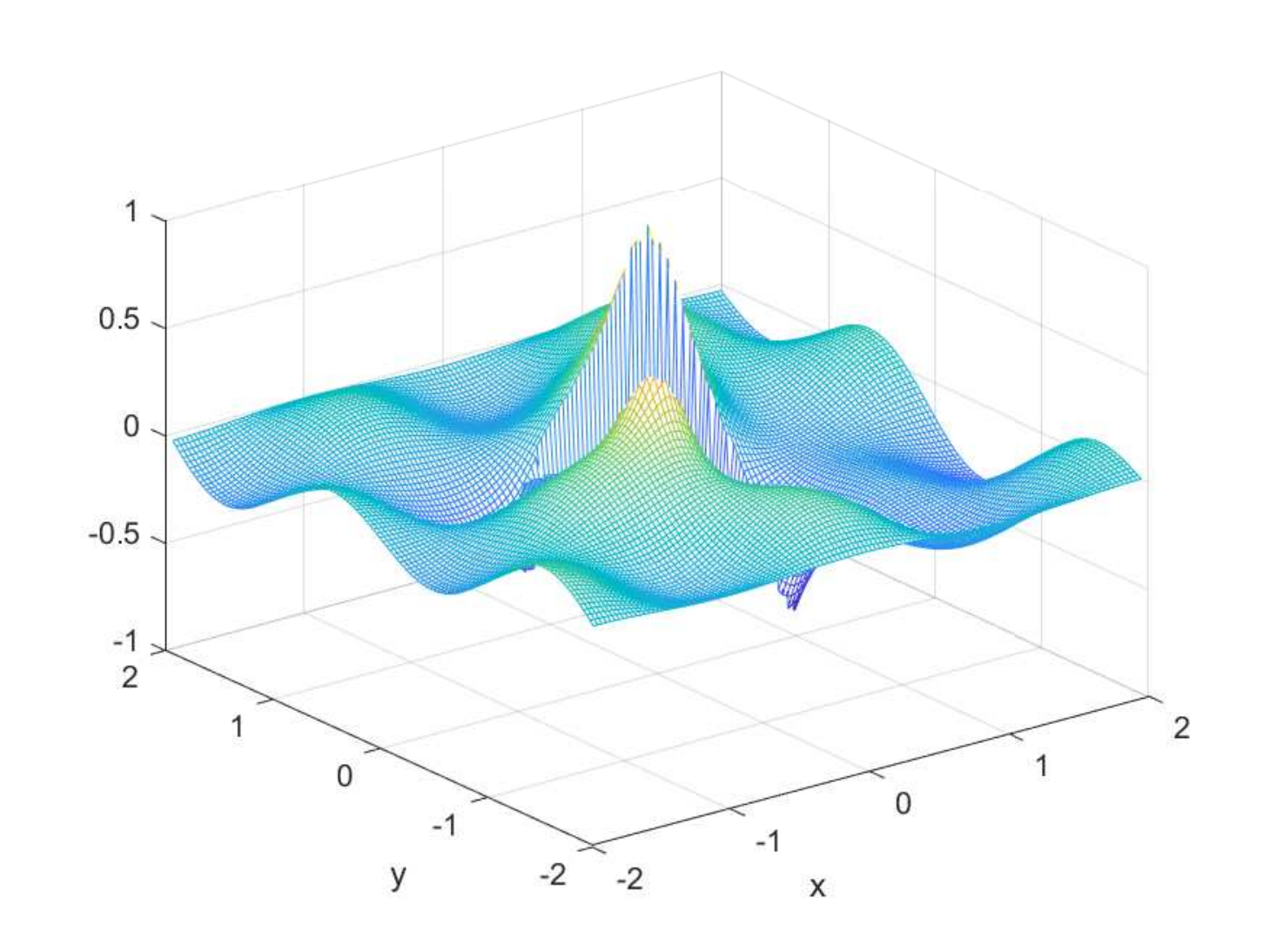}
\end{subfigure}
\begin{subfigure}[b]{0.3\textwidth}
	 \includegraphics[width=5.7cm,height=4.5cm]{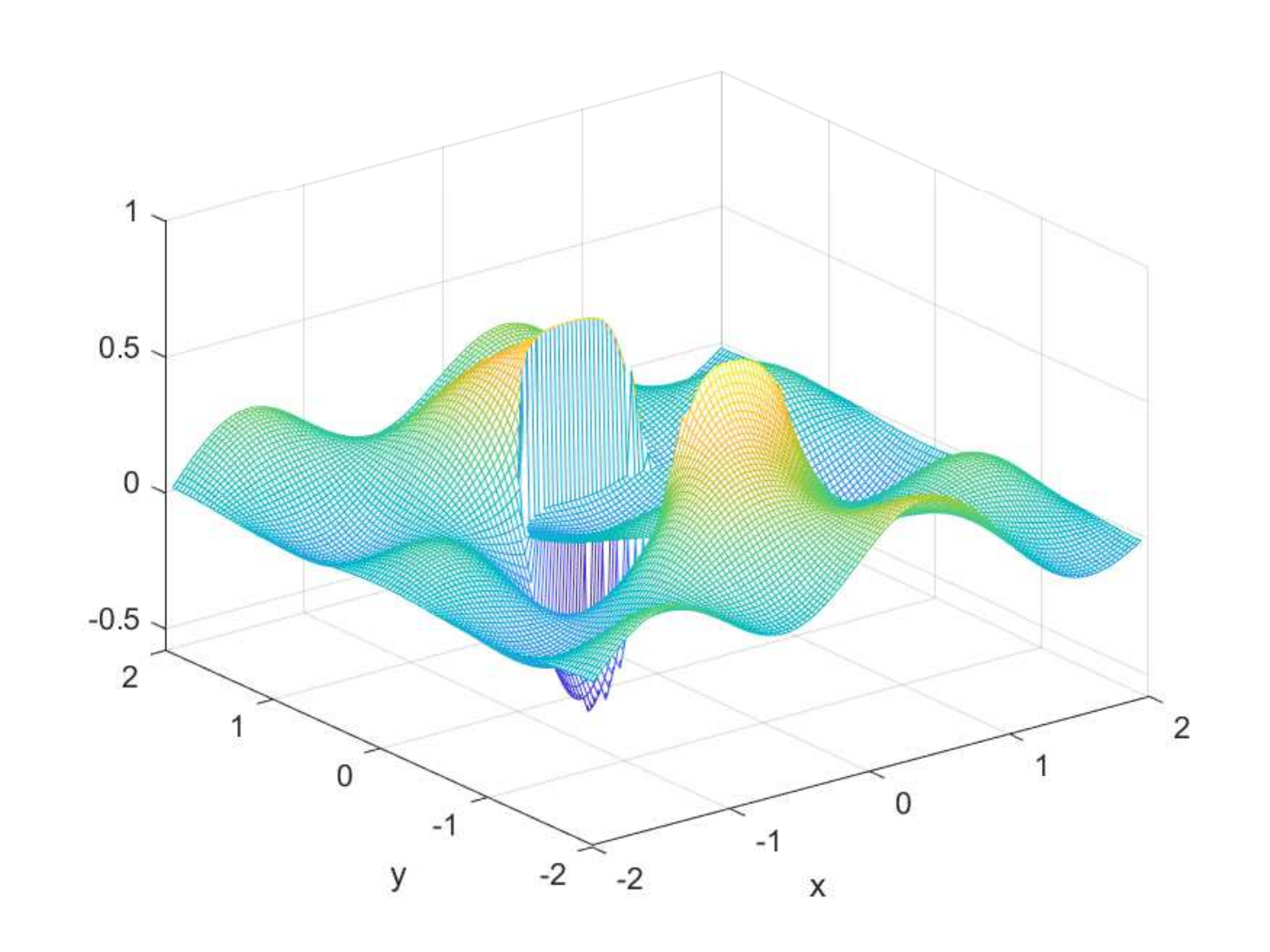}
\end{subfigure}
	\caption
	{\tiny{Top row for \cref{Drafex9}: the interface curve $\Gamma$ (left), the coefficient $a(x,y)$ (middle) and the numerical solution $u_h$ (right) with $h=2^{-7}\times 4$. Bottom row for \cref{Drafex9}:  the numerical  $(u_h)_x$ (left), the numerical  $a(x,y)\times (u_h)_x$ (middle) and  the numerical  $a(x,y)\times(u_h)_y$ (right) with $h=2^{-7}\times 4$.}}
	\label{fig:figure9}
\end{figure}
\begin{example}\label{Drafex11}
	\normalfont
	Let $\Omega=(-\pi,\pi)^2$ and
	the interface curve be given by
	$\Gamma:=\{(x,y)\in \Omega \; :\; \psi(x,y)=0\}$ with
	$\psi (x,y)=x^2+y^2-2$. Note that $\Gamma \cap \partial \Omega=\emptyset$ and
    \eqref{Qeques1} is given by
	\begin{align*}
		 &a_{+}=a\chi_{\Op}=\frac{10+\sin(x)\cos(y)}{100},
		\qquad a_{-}=a\chi_{\Om}=10(10+\sin(x-y)),\\
		 &f_{+}=f\chi_{\Op}=\sin(2x)\sin(2y),
		\qquad f_{-}=f\chi_{\Om}=\sin(2x)\sin(2y), \\
		&g_1=\sin(x)\sin(y)+2, \qquad g_2=\cos(y),
		\qquad g=0.
	\end{align*}
	The numerical results are provided in \cref{table:QSp11}  and \cref{fig:figure11}.		 
\end{example}
\begin{table}[htbp]
	\caption{\tiny{Performance in \cref{Drafex11}  of the proposed  high order compact finite difference scheme in \cref{thm:regular,thm:gradient:regular,thm:irregular,thm:gradient:irregular} on uniform Cartesian meshes with $h=2^{-J}\times 2\pi$. $\kappa$ is the condition number of the coefficient matrix.}}
	\centering
	\setlength{\tabcolsep}{2mm}{
		\begin{tabular}{c|c|c|c|c|c|c|c}
			\hline
			$J$
& $\|u_{h}-u_{h/2}\|_{2,\ind_{\Omega}}$

&order &$|u_{h}-u_{h/2}|_{H^1,\ind_{\Omega}}$

&order &  $|u_{h}-u_{h/2}|_{V,\ind_{\Omega}}$

&order &  $\kappa$ \\
\hline
3    &3.6443E+01    &0    &2.6368E+01    &0    &5.5413E+01    &0    &3.5485E+06\\
4    &3.3319E+00    &3.451    &2.5138E+00    &3.391    &9.4984E-01    &5.866    &4.2304E+06\\
5    &5.1908E-01    &2.682    &3.9809E-01    &2.659    &9.3815E-01    &0.018    &2.0691E+09\\
6    &4.4040E-02    &3.559    &3.4551E-02    &3.526    &1.0907E-02    &6.426    &4.4599E+06\\
7    &1.5339E-03    &4.844    &1.4934E-03    &4.532    &7.1251E-04    &3.936    &1.5863E+07\\
\hline
			
	\end{tabular}}
	\label{table:QSp11}
\end{table}
\begin{figure}[htbp]
	\centering
	\begin{subfigure}[b]{0.3\textwidth}
		 \includegraphics[width=5.7cm,height=4.cm]{AA1.pdf}
	\end{subfigure}
	\begin{subfigure}[b]{0.3\textwidth}
		 \includegraphics[width=5.7cm,height=4.5cm]{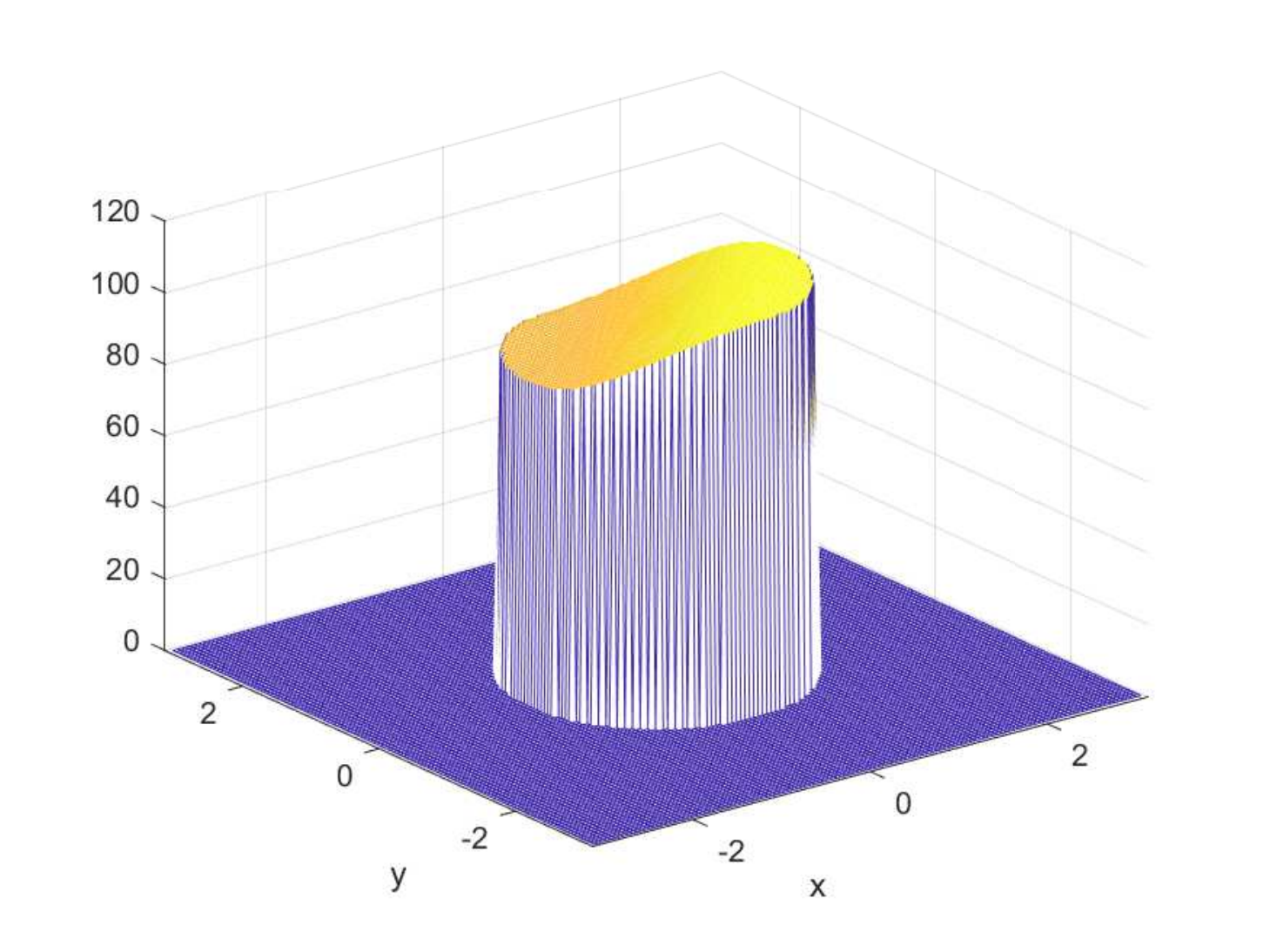}
	\end{subfigure}
	\begin{subfigure}[b]{0.3\textwidth}
		 \includegraphics[width=5.7cm,height=4.5cm]{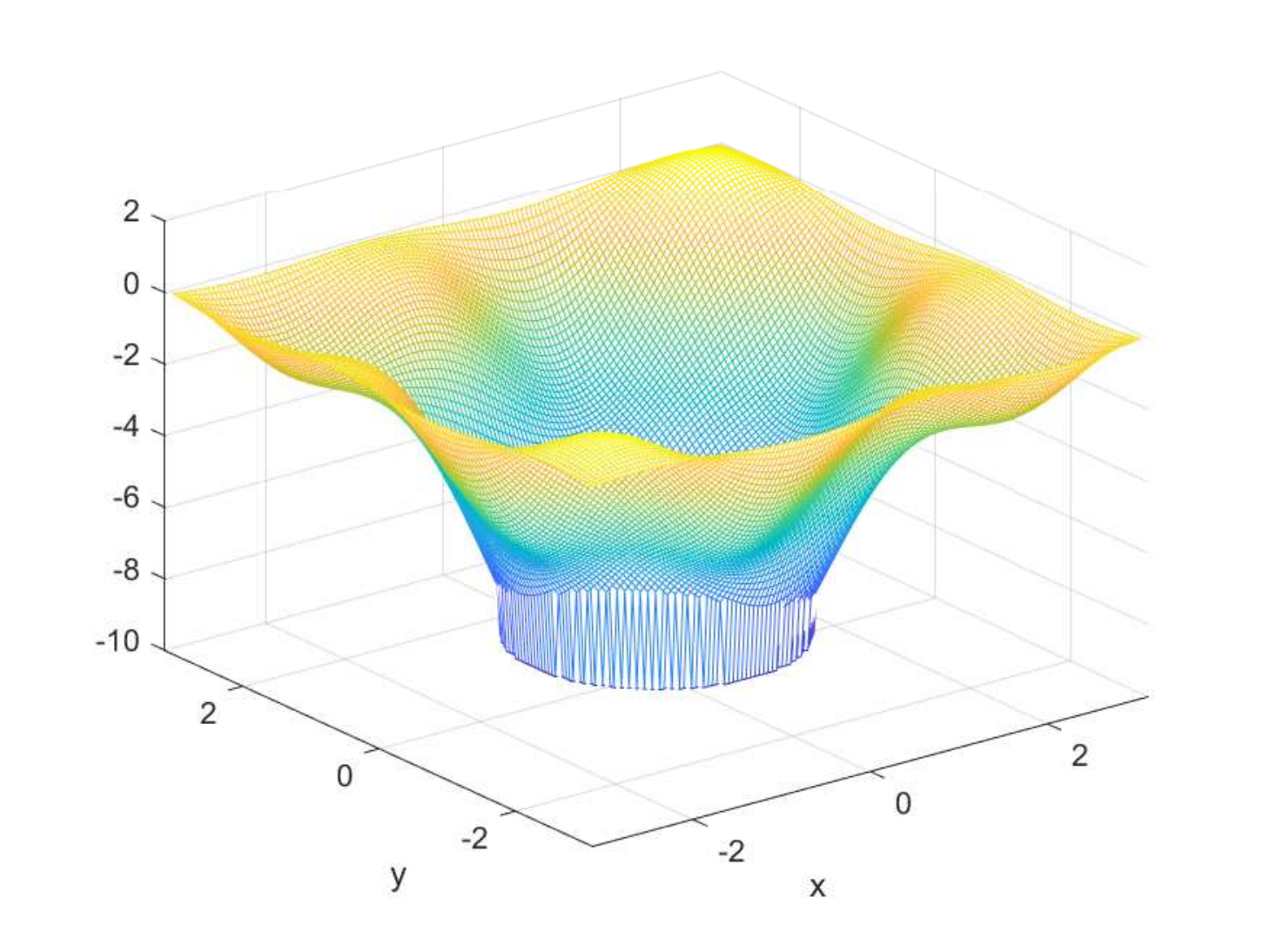}
	\end{subfigure}
	\begin{subfigure}[b]{0.3\textwidth}
		 \includegraphics[width=5.7cm,height=4.5cm]{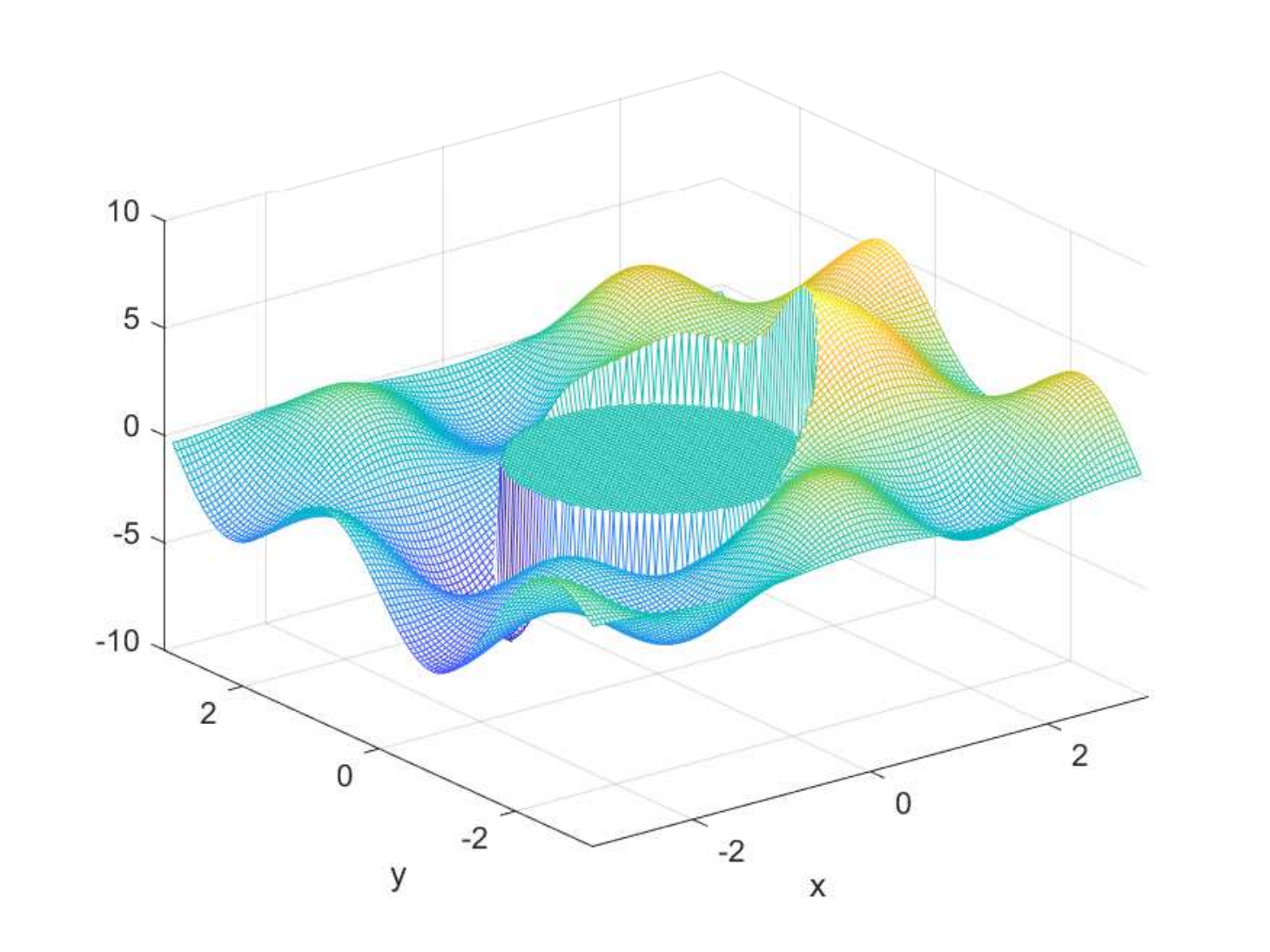}
	\end{subfigure}
	\begin{subfigure}[b]{0.3\textwidth}
	 \includegraphics[width=5.7cm,height=4.5cm]{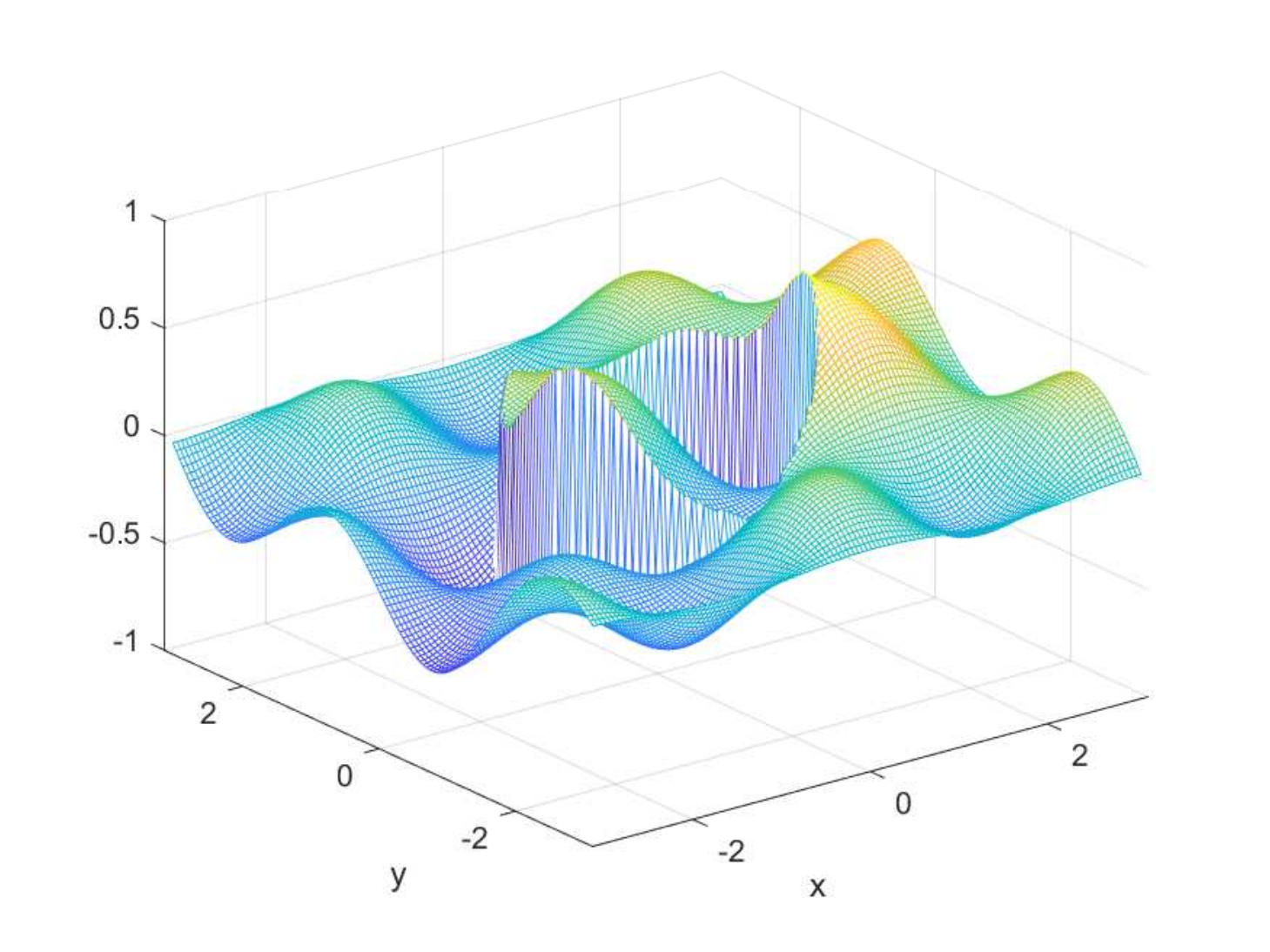}
\end{subfigure}
\begin{subfigure}[b]{0.3\textwidth}
	 \includegraphics[width=5.7cm,height=4.5cm]{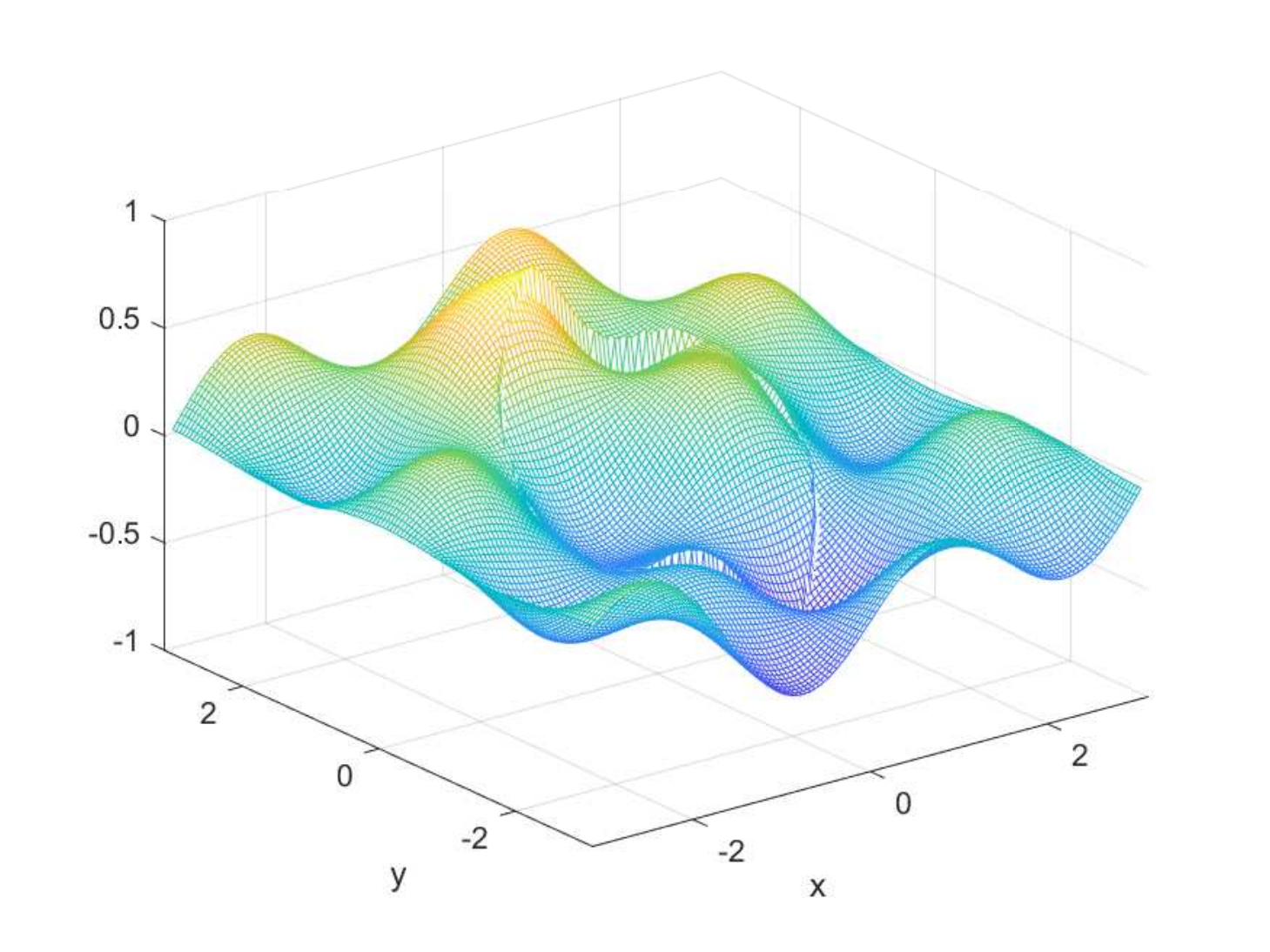}
\end{subfigure}
	\caption
	{\tiny{Top row for \cref{Drafex11}: the interface curve $\Gamma$ (left), the coefficient $a(x,y)$ (middle) and the numerical solution $u_h$ (right) with $h=2^{-7}\times 2\pi$. Bottom row for \cref{Drafex11}:  the numerical  $(u_h)_x$ (left), the numerical  $a(x,y)\times (u_h)_x$ (middle) and  the numerical  $a(x,y)\times(u_h)_y$ (right) with $h=2^{-7}\times 2\pi$.}}
	\label{fig:figure11}
\end{figure}
\begin{example}\label{Drafex12}
	\normalfont
	Let $\Omega=(-2,2)^2$ and
	the interface curve be given by
	$\Gamma:=\{(x,y)\in \Omega \; :\; \psi(x,y)=0\}$ with
	$\psi (x,y)=2x^4+y^2-1/2$. Note that $\Gamma \cap \partial \Omega=\emptyset$ and
    \eqref{Qeques1} is given by
	\begin{align*}
		 &a_{+}=a\chi_{\Op}=\frac{10+\sin(x)\cos(y)}{100},
		\qquad a_{-}=a\chi_{\Om}=10(10+\cos(x-y)),\\
		&f_{+}=f\chi_{\Op}=\sin(\pi x)\sin(\pi y),
		\qquad f_{-}=f\chi_{\Om}=\sin(\pi x)\sin(\pi y), \\
		&g_1=-\sin(x)\sin(y)-2, \qquad g_2=-\cos(y),
		\qquad g=0.
	\end{align*}
	The numerical results are provided in \cref{table:QSp12}  and \cref{fig:figure12}.		 
\end{example}
\begin{table}[htbp]
	\caption{\tiny{Performance in \cref{Drafex12}  of the proposed  high order compact finite difference scheme in \cref{thm:regular,thm:gradient:regular,thm:irregular,thm:gradient:irregular} on uniform Cartesian meshes with $h=2^{-J}\times 4$. $\kappa$ is the condition number of the coefficient matrix.}}
	\centering
	\setlength{\tabcolsep}{2mm}{
		\begin{tabular}{c|c|c|c|c|c|c|c}
			\hline
			$J$
& $\|u_{h}-u_{h/2}\|_{2,\ind_{\Omega}}$

&order &$|u_{h}-u_{h/2}|_{H^1,\ind_{\Omega}}$

&order &  $|u_{h}-u_{h/2}|_{V,\ind_{\Omega}}$

&order &  $\kappa$ \\
\hline
3    &1.0502E+02    &0    &1.0875E+02    &0    &1.4944E+03    &0    &1.3769E+05\\
4    &4.1249E+00    &4.670    &4.9993E+00    &4.443    &1.9580E+01    &6.254    &5.5506E+04\\
5    &1.1253E+00    &1.874    &1.3475E+00    &1.891    &2.0402E-01    &6.585    &3.6392E+06\\
6    &8.9752E-02    &3.648    &1.0730E-01    &3.651    &3.3886E-02    &2.590    &1.7678E+08\\
7    &6.5737E-03    &3.771    &7.9995E-03    &3.746    &2.4428E-03    &3.794    &1.9945E+07\\
\hline
			
	\end{tabular}}
	\label{table:QSp12}
\end{table}
\begin{figure}[htbp]
	\centering
	\begin{subfigure}[b]{0.3\textwidth}
		 \includegraphics[width=5.7cm,height=4.cm]{AA9.pdf}
	\end{subfigure}
	\begin{subfigure}[b]{0.3\textwidth}
		 \includegraphics[width=5.7cm,height=4.5cm]{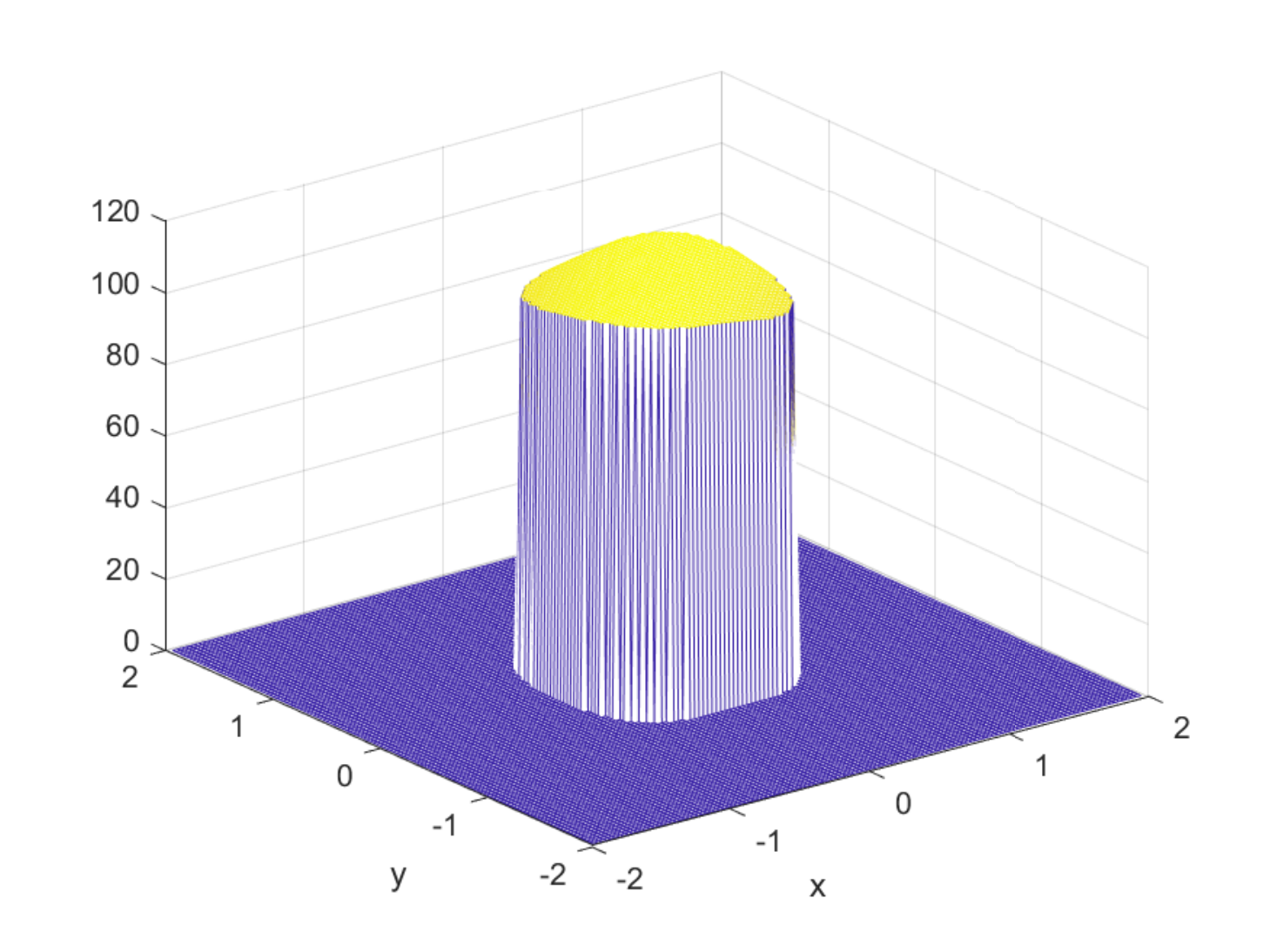}
	\end{subfigure}
	\begin{subfigure}[b]{0.3\textwidth}
		 \includegraphics[width=5.7cm,height=4.5cm]{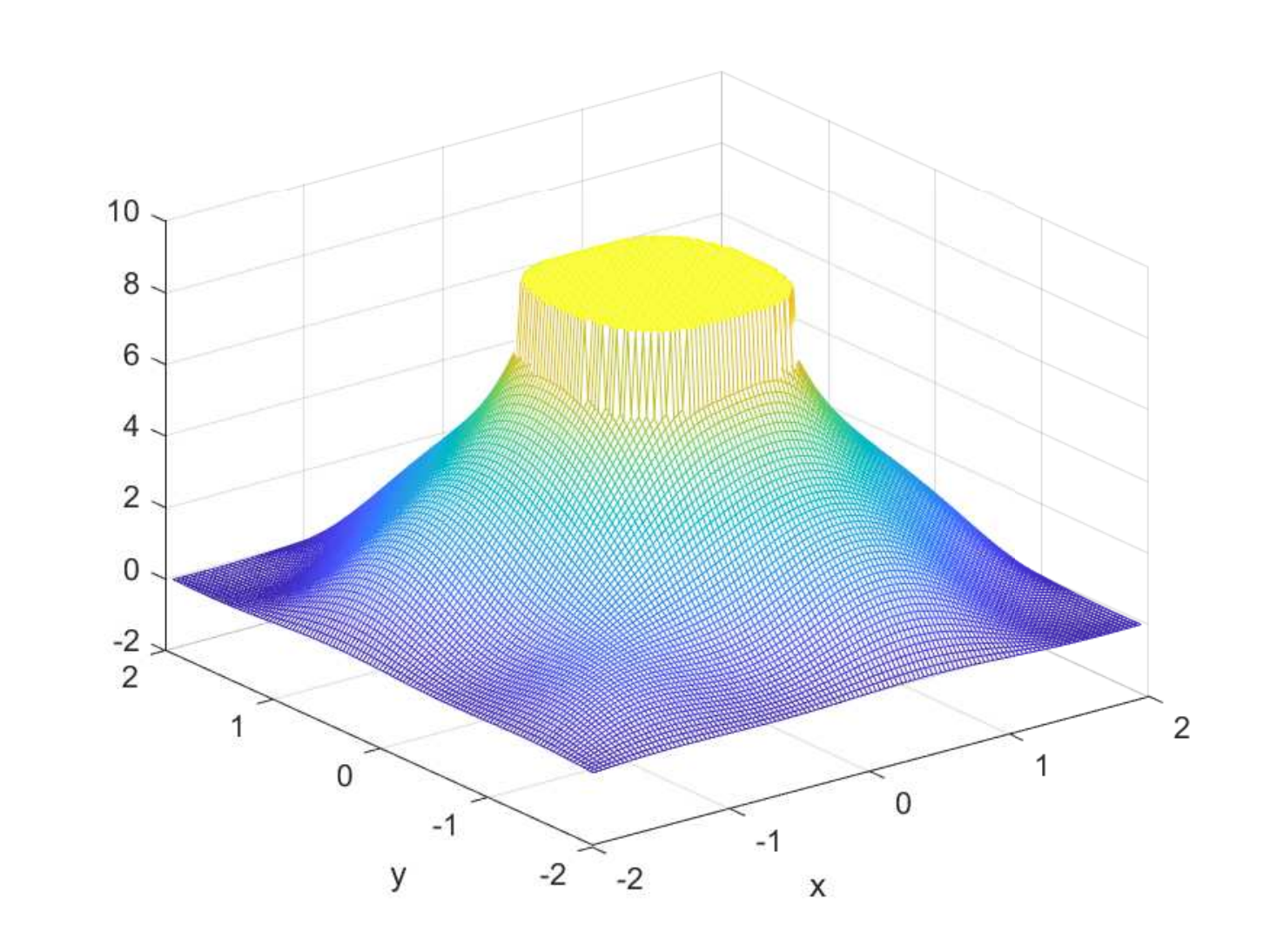}
	\end{subfigure}
	\begin{subfigure}[b]{0.3\textwidth}
		 \includegraphics[width=5.7cm,height=4.5cm]{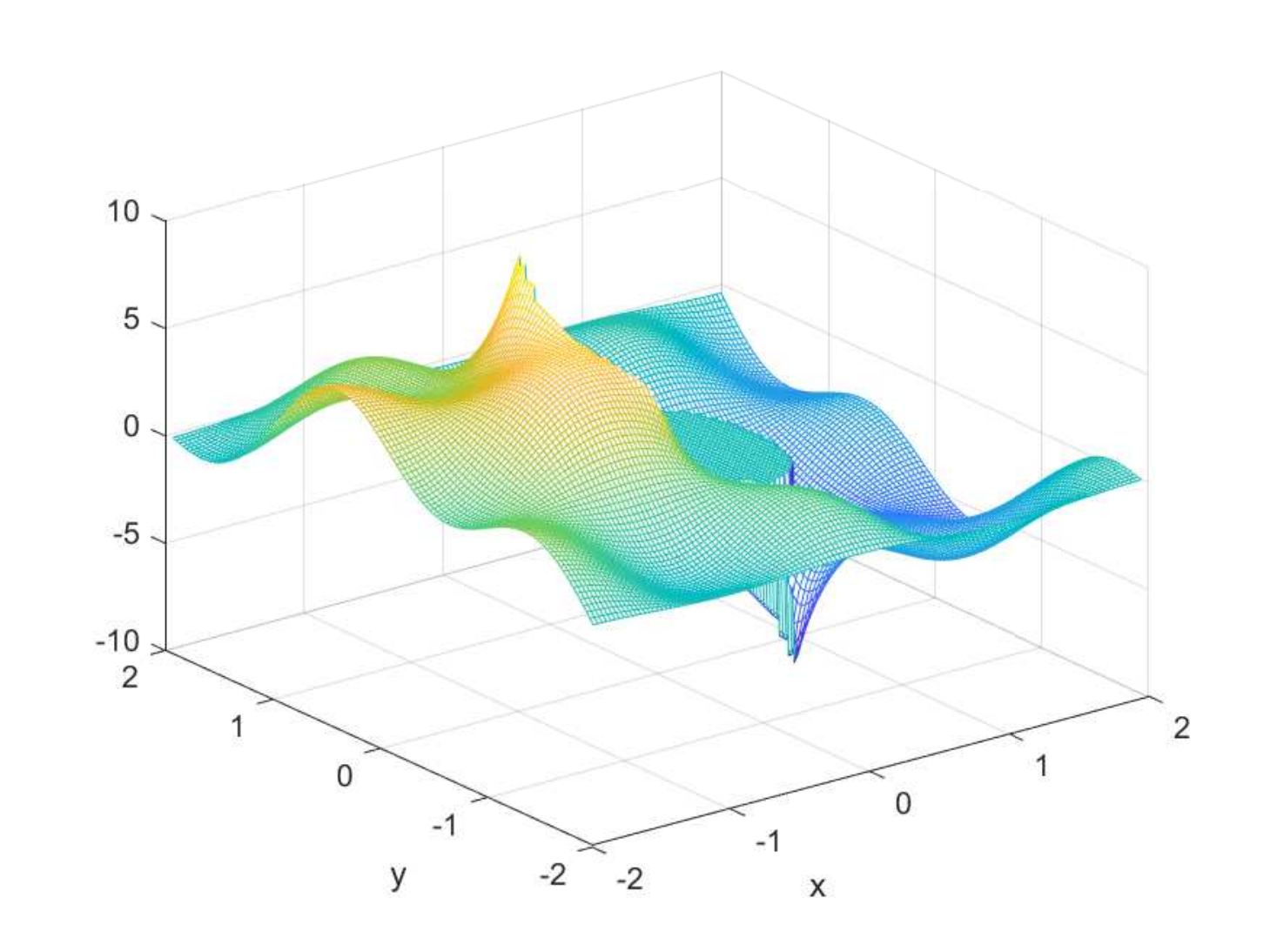}
	\end{subfigure}
	\begin{subfigure}[b]{0.3\textwidth}
	 \includegraphics[width=5.7cm,height=4.5cm]{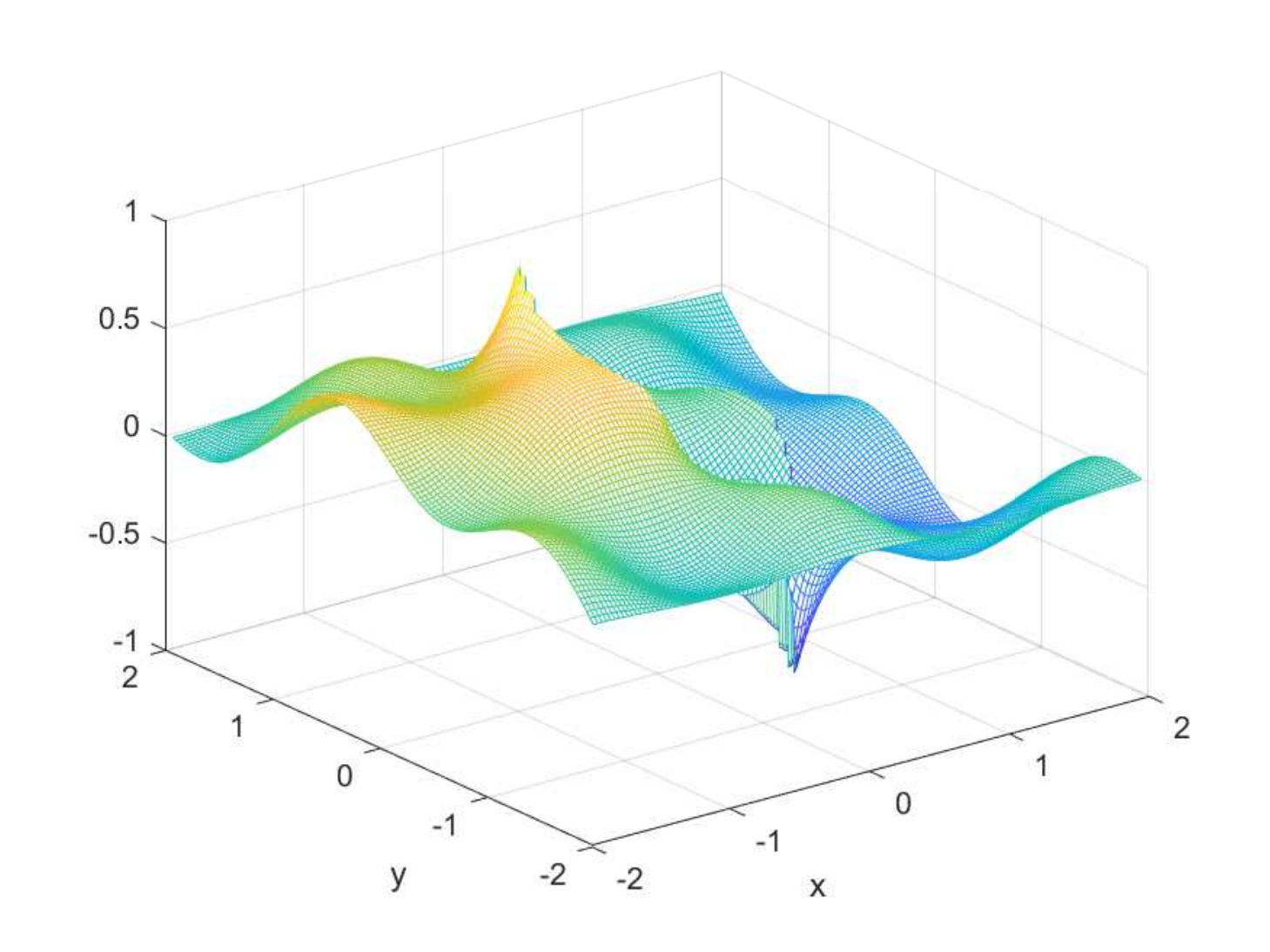}
\end{subfigure}
\begin{subfigure}[b]{0.3\textwidth}
	 \includegraphics[width=5.7cm,height=4.5cm]{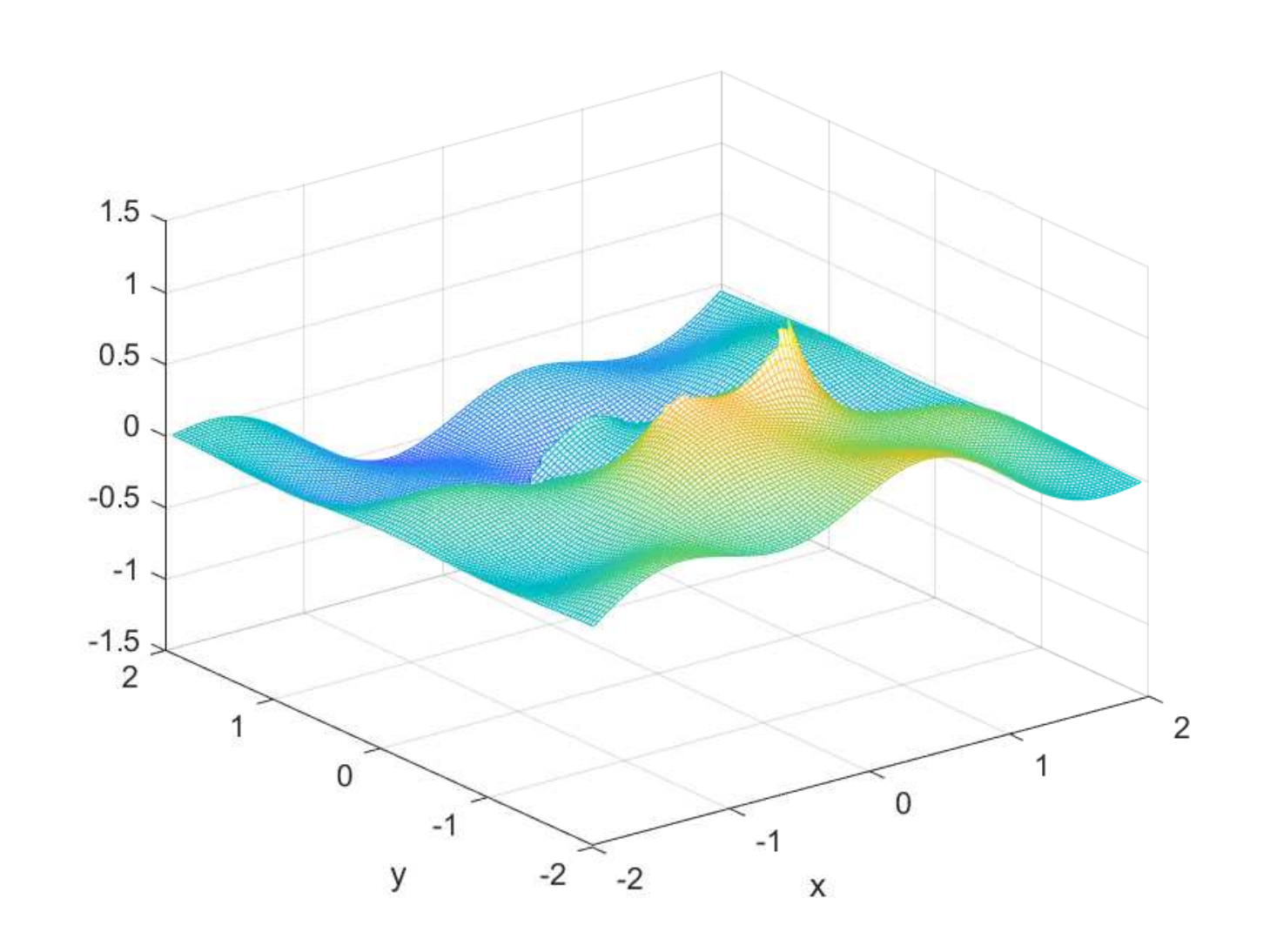}
\end{subfigure}
	\caption
	{\tiny{Top row for \cref{Drafex12}: the interface curve $\Gamma$ (left), the coefficient $a(x,y)$ (middle) and the numerical solution $u_h$ (right) with $h=2^{-7}\times 4$. Bottom row for \cref{Drafex12}:  the numerical  $(u_h)_x$ (left), the numerical  $a(x,y)\times (u_h)_x$ (middle) and  the numerical  $a(x,y)\times(u_h)_y$ (right) with $h=2^{-7}\times 4$.}}
	\label{fig:figure12}
\end{figure}

\section{Conclusion}\label{sec:Conclu}

To our best knowledge, so far there is only one fourth order compact finite difference scheme for the numerical approximated solution for the interface elliptic problems with piecewise  constant coefficients,  continuous source terms and
two homogeneous jump conditions in \cite[Section~7.2.7]{LiIto06} and \cite[Section~7.5.4]{LiIto06} provides the numerical results with $|a_+-a_-|=99$ or 9 for the proposed fourth order compact scheme in uniform and no-nested mesh grids. The compact schemes in \cite[Section~7.2.7]{LiIto06} are based on coordinate transformations and optimization problems.

Our contributions of this paper are as follows:
\begin{enumerate}
\item[(1)] We construct a high order compact finite difference scheme for the numerical  solution on uniform meshes for \eqref{Qeques1} with a discontinuous, piecewise  smooth and high-contrast coefficient (the ratio $\sup(a_+)/\inf(a_-)=10^{-3},10^{-2},10^{3},10^{4}$), discontinuous source terms and two non-homogeneous jump conditions. $a_+$ and $a_-$  can be linearly independent or dependent.

\item[(2)] Since we do not need to change coordinates into the local coordinates and solve an optimization problem to derive the scheme, it is simple for readers to understand the procedure, derive the schemes, and perform the implementations.

\item[(3)]  For the irregular points case,  Eq.(7.73) in \cite[Section~7.2.7]{LiIto06} expands the Taylor series of $u(x,y)$ to $\bo(h^{5})$, while we only need to expand the Taylor series of $u(x,y)$ to $\bo(h^{4})$, which significantly reduces the computational costs to calculate the coefficients of the proposed schemes.  Moreover, we also prove that the maximum order of the  compact finite difference schemes for the numerical approximated solutions at irregular points on uniform meshes is three.

\item[(4)] Since the gradients are crucial in the real world problems to analyze the speeds of fluids, we also derive a high order compact finite difference scheme for the numerical approximated gradients.
Our numerical experiments confirm the flexibility and the fourth order accuracy in numerical approximated $L^2$ norms for the numerical approximated solutions $u_h$, the numerical approximated gradients $\big((u_h)_x,(u_h)_y\big)$ and the numerical approximated velocities  $\big(a(u_h)_x,a(u_h)_y\big)$ of the proposed schemes.
\end{enumerate}

\section{Proof of \cref{thm:interface}}
\label{sec:proof}

\begin{proof}[Proof of \cref{thm:interface}]

	Similar as the proof of Theorem 2.3 in \cite{FengHanMinev21}, by \eqref{parametric}, two jump conditions in \eqref{Qeques1} can be written as
	\be \label{interface:u}
	 u_+(r(t)+x_i^*,s(t)+y_j^*)-u_-(r(t)+x_i^*,s(t)+y_j^*)=g_1(r(t)+x_i^*,s(t)+y_j^*),
	\ee
	\be \label{interface:flux}
	\begin{split}
		\big((a_+\nabla u_+)(r(t)+x_i^*,s(t)+y_j^*)-
		&(a_-\nabla u_-)(r(t)+x_i^*,s(t)+y_j^*)\big) \cdot (s'(t),-r'(t))\\
		 &=g_2(r(t)+x_i^*,s(t)+y_j^*)\sqrt{(r'(t))^2+(s'(t))^2},
	\end{split}
	\ee
	for $t\in (-\epsilon,\epsilon)$.
	Because all involved functions in \eqref{interface:u} and \eqref{interface:flux} are assumed to be smooth,
	to link the two sets $\{u_+^{(m,n)}: (m,n)\in \ind_{M+1}^1\}$ and $\{u_-^{(m,n)}: (m,n)\in \ind_{M+1}^1\}$,
	we now take the Taylor approximation of the above functions near the base parameter $t=0$.
	\eqref{u:approx:ir:key} implies
	{\small{
	\begin{align*}
		u_\pm (r(t)+x_i^*,s(t)+y_j^*)
		&=\sum_{(m,n)\in \ind_{M+1}^1}
		u_\pm^{(m,n)} G^{\pm}_{m,n}(r(t),s(t))
		+\sum_{(m,n)\in \ind_{M-1}} f_\pm^{(m,n)}
		H^{\pm}_{m,n}(r(t), s(t))+\bo(t^{M+2})\\
		&=\sum_{p=0}^{M+1}
		\left(\sum_{(m,n)\in \ind_{M+1}^1}
		u_\pm^{(m,n)} g^{\pm}_{m,n,p}
		+\sum_{(m,n)\in \ind_{M-1}} f_\pm^{(m,n)}
		h^{\pm}_{m,n,p} \right) t^p+\bo(t^{M+2}),
	\end{align*}
}}
	where
	\be\label{gmnhmn0}
	g^{\pm}_{m,n,p}:=\frac{1}{p!} \frac{d^p(G^{\pm}_{m,n}(r(t),s(t)))}{dt^p}\Big|_{t=0},\ \ h^{\pm}_{m,n,p}:=\frac{1}{p!}\frac{d^p(H^{\pm}_{m,n}(r(t),s(t)))}{dt^p}\Big|_{t=0},\quad p=0,\ldots,M+1.
	\ee
	Similarly,
	\begin{align*}
		g_1(r(t)+x_i^*,s(t)+y_j^*)
		&=\sum_{(m,n)\in \ind_{M+1}} \frac{g_1^{(m,n)}}{m!n!} (r(t))^m (s(t))^n
		+\bo(t^{M+2})\\
		 &=\sum_{p=0}^{M+1}\left(\sum_{(m,n)\in \ind_{M+1}} \frac{g_1^{(m,n)}}{m!n!} r_{m,n,p}\right) t^p+\bo(t^{M+2}),
	\end{align*}
	where the constants $r_{m,n,p}:=\frac{1}{p!}\frac{d^p ((r(t))^m (s(t))^n)}{d t^p}\Big|_{t=0}$ for $p=0,\ldots,M+1$.
	Since each entry of ${G}^{\pm}_{m,n}(x,y)$ is a homogeneous polynomial of degree $\ge m+n$  and $r(0)=s(0)=0$,
	we have $g^{\pm}_{m,n,p}=0$ for all $0\le p<m+n$ by \eqref{gmnhmn0}.
	Thus,  \eqref{interface:u} leads to
	\be \label{interface:u:01}
	\sum_{(m,n)\in \ind_{M+1}^1}u_+^{(m,n)}g^{+}_{m,n,p}-u_-^{(m,n)}g^{-}_{m,n,p}
	=F_p,\qquad p=0,\ldots,M+1,
	\ee
	where $F_0:=g_1^{(0,0)}$ and
	\[
	F_p:=\sum_{(m,n)\in \ind_{M-1}} f_-^{(m,n)}h^{-}_{m,n,p}-f_+^{(m,n)}
	h^{+}_{m,n,p}
	+\sum_{(m,n)\in \ind_{M+1}}
	\frac{g_1^{(m,n)}}{m!n!} r_{m,n,p}, \qquad p=1,\ldots,M+1.
	\]
	Clearly, $g^{\pm}_{0,0,0}=1$ and $g^{\pm}_{m,n,p}=0$ for all $0\le p<m+n$. We observe that the identities in \eqref{interface:u:01} become
	%
	\be \label{interface:u:0}
	u_{-}^{(0,0)}
	=u_{+}^{(0,0)}-g_1^{(0,0)},
	\ee
	\be \label{interface:u:1}
	\begin{split}
	 &u_{-}^{(0,p)}g^{-}_{0,p,p}+u_{-}^{(1,p-1)}g^{-}_{1,p-1,p}
	 =u_{+}^{(0,p)}g^{+}_{0,p,p}+u_{+}^{(1,p-1)}g^{+}_{1,p-1,p}-F_p\\
	&+
	\sum_{(m,n)\in \ind_{p-1}^1}
	 u_{+}^{(m,n)}g^{+}_{m,n,p}-u_{-}^{(m,n)}g^{-}_{m,n,p},\qquad p=1,\ldots,M+1.
	\end{split}
	\ee
	%
	By \eqref{Gmn},
	\be\label{Gmn:seperateT}
	 G^{\pm}_{m,n}(x,y):=G^{\pm,1}_{m,n}(x,y)+G^{\pm,2}_{m,n}(x,y),
	\ee
	where
	\be\label{Gmn:seperate1}
	 G^{\pm,1}_{m,n}(x,y):=\sum_{\ell=0}^{\lfloor \frac{n}{2}\rfloor}
	\frac{(-1)^\ell x^{m+2\ell} y^{n-2\ell}}{(m+2\ell)!(n-2\ell)!},
	\ee
	\be\label{Gmn:seperate2}
	G^{\pm,2}_{m,n}(x,y):= \sum_{(m',n')\in \ind_{M+1}^2 \setminus \ind_{m+n}^2 }A^{u}_{m',n',m,n} \frac{ x^{m'} y^{n'}}{m'!n'!}, \quad \forall (m,n)\in \ind_{M+1}^1.
	\ee	
	Since each entry of $ G^{\pm,2}_{m,n}(x,y)$ is a homogeneous polynomial of degree $\ge m+n+1$ and $s(0)=r(0)=0$, \eqref{gmnhmn0} leads to
	\be\label{gmnhmn1}
	g^{\pm}_{m,n,p}:=\frac{1}{p!} \frac{d^p(G^{\pm,1}_{m,n}(r(t),s(t)))}{dt^p}\Big|_{t=0}, \quad (m,n)\in\{(0,p),(1,p-1)\}.
	\ee
	For the flux jump condition \eqref{interface:flux}, \eqref{u:approx:ir:key} implies
		{\footnotesize{
	\be \label{uxy:approx:ir:key}
	\begin{split}
		\nabla \big(u_\pm (x+x_i^*,y+y_j^*))\big)
		&=\sum_{(m,n)\in \ind_{M+1}^1}
		u_\pm^{(m,n)} \nabla\big(G^{\pm}_{m,n}(x,y)\big) +\sum_{(m,n)\in \ind_{M-1}}
		f_\pm ^{(m,n)} \nabla\big( H^{\pm}_{m,n}(x,y)\big)+\bo(h^{M+1}),
	\end{split}
	\ee}}
	for $x,y\in (-2h,2h)$ and clearly
	\be \label{a:approx}
	a_{\pm}(x+x_i^*,y+y_j^*)=
	\sum_{(m,n)\in \ind_{M}} \frac{a_{\pm}^{(m,n)}}{m!n!}x^m y^{n}
	+\bo(h^{M+1}),
	\ee
	for $x,y\in (-2h,2h)$.
	By \eqref{uxy:approx:ir:key} and \eqref{a:approx},
{\small{
		\begin{align*}
			&a_{\pm}(x+x_i^*,y+y_j^*)\nabla u_\pm (r(t)+x_i^*, s(t)+y_j^*)\cdot(s'(t),-r'(t))\\
			&=\sum_{(m,n)\in \ind_{M+1}^1}
			u_\pm^{(m,n)}  \widetilde{G}^{\pm}_{m,n}(r(t),s(t)) \cdot(s'(t),-r'(t))
			+\sum_{(m,n)\in \ind_{M-1}}
			f_\pm ^{(m,n)}  \widetilde{H}^{\pm}_{m,n}(r(t),s(t))\cdot(s'(t),-r'(t))\\
			&=
			\sum_{p=0}^{M}
			\left(
			\sum_{(m,n)\in \ind_{M+1}^1}
			u_\pm^{(m,n)} \tilde{g}^{\pm}_{m,n,p}
			+\sum_{(m,n)\in \ind_{M-1}}
			f_\pm ^{(m,n)} \tilde{h}^{\pm}_{m,n,p}
			\right) t^p+\bo(t^{M+1}),
		\end{align*}
	}}
	where
		{\footnotesize{
		\begin{align*}
		& \widetilde{G}^{\pm}_{m,n}(x,y) = \nabla G^{\pm}_{m,n}(x,y)\Bigg(\sum_{(m,n)\in \ind_{M}} \frac{a_{\pm}^{(m,n)}}{m!n!}x^m y^{n}\Bigg),\ \ \widetilde{H}^{\pm}_{m,n}(x,y) = \nabla H^{\pm}_{m,n}(x,y)\Bigg(\sum_{(m,n)\in \ind_{M}} \frac{a_{\pm}^{(m,n)}}{m!n!}x^m y^{n}\Bigg),
	\end{align*}}}
	{\footnotesize{
	\be\label{gmnhmn2}
	 \tilde{g}^{\pm}_{m,n,p}:=\frac{1}{p!}\frac{d^p( \widetilde{G}^{\pm}_{m,n}(r(t),s(t)) \cdot(s'(t),-r'(t)))}{dt^p}\Big|_{t=0},\ \ 
	 \tilde{h}^{\pm}_{m,n,p}:=\frac{1}{p!}\frac{d^p( \widetilde{H}^{\pm}_{m,n}(r(t),s(t)) \cdot(s'(t),-r'(t)))}{dt^p}\Big|_{t=0}.
	\ee}}
	Note that each entry of $ \widetilde{G}^{\pm}_{m,n}$ is a homogeneous polynomial of degree $\ge m+n-1$. By $r(0)=s(0)=0$ and \eqref{gmnhmn2}, we can say that $\tilde{g}^{\pm}_{m,n,p}=0$ for all $0\le p<m+n-1$.
	Similarly, we have
	{\small{
	\begin{align*}
		 g_2(r(t)+x_i^*,s(t)+y_j^*)\sqrt{(r'(t))^2+(s'(t))^2}
		&=\sum_{(m,n)\in \ind_{M}} \frac{g_2^{(m,n)}}{m!n!} (r(t))^m (s(t))^n \sqrt{(r'(t))^2+(s'(t))^2}
		+\bo(t^{M+1})\\
		&=\sum_{p=0}^{M}\left(\sum_{(m,n)\in \ind_M} \frac{g_2^{(m,n)}}{m!n!} \tilde{r}_{m,n,p}\right) t^p+\bo(t^{M+1}),
	\end{align*}
}}
	as $t\to 0$,
	where
	 \[\tilde{r}_{m,n,p}:=\frac{1}{p!}\frac{d^p \Big((r(t))^m (s(t))^n\sqrt{(r'(t))^2+(s'(t))^2}\Big)}{d t^p}\Bigg|_{t=0}, \qquad p=0,\ldots,M.\]
	Therefore, \eqref{interface:flux} implies
	\be \label{interface:flux:0}
	\sum_{(m,n)\in \ind_{M+1}^1}
	 u_+^{(m,n)}\tilde{g}^+_{m,n,p}-u_-^{(m,n)} \tilde{g}^-_{m,n,p}=G_p,\qquad p=0,\ldots,M,
	\ee
	where
	\[
	G_p:=\sum_{(m,n)\in \ind_{M-1}} f_-^{(m,n)}\tilde{h}^{-}_{m,n,p}-f_+^{(m,n)}
	\tilde{h}^{+}_{m,n,p}
	+\sum_{(m,n)\in \ind_{M}}
	\frac{g_2^{(m,n)}}{m!n!} \tilde{r}_{m,n,p}.
	\]
	Clearly, $\tilde{g}^{\pm}_{0,0,0}=0$ and $\tilde{g}^{\pm}_{m,n,p}=0$ for all $0\le p< m+n-1$. We observe that \eqref{interface:flux:0} become
	{\footnotesize{
	\be \label{interface:flux:1}
	\begin{split}
	u_-^{(0,p)}\tilde{g}^{-}_{0,p,p-1}
	+u_-^{(1,p-1)} \tilde{g}^{-}_{1,p-1,p-1}
	&=u_+^{(0,p)}\tilde{g}^{+}_{0,p,p-1}
	+u_+^{(1,p-1)} \tilde{g}^{+}_{1,p-1,p-1}-G_{p-1}\\
	&+\sum_{(m,n)\in \ind_{p-1}^1}
	 u_{+}^{(m,n)}\tilde{g}^{+}_{m,n,p-1}-u_{-}^{(m,n)}\tilde{g}^{-}_{m,n,p-1},\quad p=1,\ldots,M+1.
	\end{split}
	\ee}}
	Since each entry of $ G^{\pm,2}_{m,n}(x,y)$ is a homogeneous polynomial of degree $\ge m+n+1$ and $s(0)=r(0)=0$, \eqref{gmnhmn2} \eqref{Gmn:seperateT}, \eqref{Gmn:seperate1} and \eqref{Gmn:seperate2} leads to
	\be\label{gmnhmntide}
	 \tilde{g}^{\pm}_{m,n,p-1}:=\frac{a_{\pm}^{(0,0)}}{(p-1)!}\frac{d^{p-1}( \nabla G^{\pm,1}_{m,n}(x,y) \cdot(s'(t),-r'(t)))}{dt^{p-1}}\Big|_{t=0}, \quad (m,n)\in\{(0,p),(1,p-1)\}.
	\ee
	According to the assumption $(r'(0))^2+(s'(0))^2>0$ in \eqref{parametric}, $a_{\pm}^{(0,0)}\ne 0$ in \eqref{Qeques1} and the proof of Theorem 2.3 in \cite{FengHanMinev21},
	\eqref{Gmn:seperate1}, \eqref{gmnhmn1} and \eqref{gmnhmntide} imply
	\be \label{nonzero}
	 g^{\pm}_{0,p,p}\tilde{g}^{\pm}_{1,p-1,p-1}- g^{\pm}_{1,p-1,p}
	\tilde{g}^{\pm}_{0,p,p-1}>0,\qquad \forall\; p=1,\ldots,M.
	\ee
	%
		Let
	\[
	W^{\pm}_p=	\left[ \begin{matrix} g^{\pm}_{0,p,p}
		&g^{\pm}_{1,p-1,p}\\
		\tilde{g}^{\pm}_{0,p,p-1} & \tilde{g}^{\pm}_{1,p-1,p-1}
	\end{matrix}\right] \quad \mbox{and}  \quad
	 Q^{\pm}_p:=\frac{1}{g^{\pm}_{0,p,p}\tilde{g}^{\pm}_{1,p-1,p-1}- g^{\pm}_{1,p-1,p}
	\tilde{g}^{\pm}_{0,p,p-1}}
\left[ \begin{matrix} \tilde{g}^{\pm}_{1,p-1,p-1}
	&-g^{\pm}_{1,p-1,p}\\
	-\tilde{g}^{\pm}_{0,p,p-1} &g^{\pm}_{0,p,p}\end{matrix}\right].
	\]
Then, by  \eqref{nonzero}, we have $W^{\pm}_pQ^{\pm}_p=I_2$, where $I_2$ is a 2 by 2 identity matrix.
	
	Therefore, the solution $\{u_-^{(0,p)}, u_-^{(1,p-1)}\}_{p=1,\ldots M}$ of the linear equations in \eqref{interface:u:1} and \eqref{interface:flux:1} can be recursively and uniquely calculated
	from $p=1$ to $p=M$ by $u_-^{(0,0)}=u_+^{(0,0)}-g_1^{(0,0)}$ due to \eqref{interface:u:0} and
	\be \label{transmissioncoef}
	\begin{split}
	\left[ \begin{matrix}
		u_-^{(0,p)}\\
		u_-^{(1,p-1)}\end{matrix}\right]
	&=Q^{-}_pW^{+}_p\left[ \begin{matrix}
		u_+^{(0,p)}\\
		 u_+^{(1,p-1)}\end{matrix}\right]-Q^{-}_p
	\left[ \begin{matrix}
		F_p\\
		G_{p-1}
	\end{matrix}\right]+\sum_{n=1}^{p-1}
	Q^{-}_p \left[
	 \begin{matrix}u_{+}^{(0,n)}g^{+}_{0,n,p}+u_{+}^{(1,n-1)}g^{+}_{1,n-1,p}\\
		 u_{+}^{(0,n)}\tilde{g}^{+}_{0,n,p-1}
		 +u_{+}^{(1,n-1)}\tilde{g}^{+}_{1,n-1,p-1}
	\end{matrix}\right]\\
&-\sum_{n=1}^{p-1}
Q^{-}_p \left[
\begin{matrix}u_{-}^{(0,n)}g^{-}_{0,n,p}+u_{-}^{(1,n-1)}g^{-}_{1,n-1,p}\\
u_{-}^{(0,n)}\tilde{g}^{-}_{0,n,p-1}
+u_{-}^{(1,n-1)}\tilde{g}^{-}_{1,n-1,p-1}
\end{matrix}\right],
\end{split}
	\ee
	for $p=1,\ldots,M+1$.
	Note that for $p=1$, the above summation $\sum_{n=1}^{p-1}$ is empty.
\end{proof}

\end{document}